\documentclass[11pt,lenq]{amsart}

\usepackage{amsmath}
\usepackage{amscd}
\usepackage{amssymb}
\usepackage{amsbsy}
\usepackage{amsfonts}
\usepackage{mathrsfs}
\usepackage{color}
\usepackage{latexsym}
\usepackage{graphics}
\usepackage{graphicx}
\usepackage{amsmath,amscd,latexsym}
\usepackage{multirow}
\usepackage{array}
\usepackage{paralist}
\usepackage[all]{xy}
\usepackage{titletoc}
\usepackage{url}

\theoremstyle{plain}
\newtheorem{theorem}{Theorem}[section]
\newtheorem{proposition}[theorem]{Proposition}
\newtheorem{lemma}[theorem]{Lemma}
\newtheorem{corollary}[theorem]{Corollary}
\newtheorem{remark}[theorem]{Remark}
\newtheorem{definition}[theorem]{Definition}
\newtheorem{notation}[theorem]{Notation}
\newtheorem{example}[theorem]{Example}
\newtheorem{main theorem}[theorem]{Main Theorem}

\newtheorem{question}[theorem]{Question}
\newtheorem{convention}[theorem]{Convention}

\newtheoremstyle{citing}
  {}
  {}
  {\itshape}
  {}
  {\bfseries}
  {.}
  {.5em}
  {\thmnote{#3}}

\theoremstyle{citing}
\newtheorem*{citingtheorem}{}

\newlength\savewidth

\newcommand{\ZZ}{\mathbb{Z}}

\newcommand{\RR}{\mathbb{R}}

\newcommand{\HH}{\mathbb{H}}

\newcommand{\interior}{\operatorname{int}}
\newcommand{\Fix}{\operatorname{Fix}}
\newcommand{\cl}{\operatorname{cl}}

\newcommand{\Homeo}{\operatorname{Homeo}}
\newcommand{\MCG}{\operatorname{MCG}}
\newcommand{\Diff}{\operatorname{Diff}}
\newcommand{\Mdy}{\Gamma}
\newcommand{\Motion}{\mathcal{M}}
\newcommand{\Out}{\operatorname{Out}}
\newcommand{\Deg}{\operatorname{Deg}}
\newcommand{\RP}{\mathbb{RP}}

\newcommand{\Isom}{\operatorname{Isom}}

\newcommand{\id}{\mathrm{id}}
\newcommand{\Int}{\operatorname{Int}}

\newcommand{\image}{\operatorname{image}}
\newcommand{\SO}{\operatorname{SO}}
\newcommand{\Dh}{\mathbb{D}}

\newcommand{\Inv}{\mathcal{I}}

\newcommand{\Curve}{\mathcal{S}}

\makeatletter
\renewcommand\subsection{\@startsection{subsection}{2}{0mm}
    {-10.5dd plus-8pt minus-4pt}{10.5dd}
     {\normalsize\upshape}}
\makeatother

\begin{document}

\title[Homotopy motions of surfaces]
{Homotopy motions of surfaces in $3$-manifolds}

\author{Yuya Koda}
\address{Department of Mathematics\\
Faculty of Science\\
Hiroshima University\\
Higashi-Hiroshima, 739-8526, Japan}
\email{ykoda@hiroshima-u.ac.jp}

\author{Makoto Sakuma}
\address{Advanced Mathematical Institute\\
Osaka Metropolitan University\\
3-3-138, Sugimoto, Sumiyoshi, Osaka City
558-8585, Japan}
\address{Department of Mathematics\\
Faculty of Science\\
Hiroshima University\\
Higashi-Hiroshima, 739-8526, Japan}
\email{sakuma@hiroshima-u.ac.jp}


\makeatletter
\@namedef{subjclassname@2020}{%
  \textup{2020} Mathematics Subject Classification}
\makeatother

\subjclass[2020]{Primary 57K30, Secondary 57K20, 57M10, 57M12}


\thanks{
Y. K. is supported by JSPS KAKENHI Grant Numbers JP17H06463, 
JP20K03588 and JST CREST Grant Number JPMJCR17J4. 
M. S. is supported by JSPS KAKENHI Grant Number JP20K03614
and by Osaka Central Advanced Mathematical Institute 
(MEXT Joint Usage/Research Center on Mathematics and Theoretical Physics JPMXP0619217849).}

\begin{abstract}
We introduce the concept of a homotopy motion of a subset in a manifold, 
and give a systematic study of homotopy motions of surfaces in closed orientable
3-manifolds. 
This notion arises from various natural problems in 3-manifold theory such as domination of manifold pairs, 
homotopical behavior of simple loops on a Heegaard surface, and 
monodromies of virtual branched covering surface bundles associated to a Heegaard splitting. 
\end{abstract}
\maketitle

\section*{Introduction}
\label{sec:intro}

\subsection{Homotopy motion groups and related groups}

For a manifold $M$ and a compact subspace $\Sigma$, 
a {\it motion} of $\Sigma$ in $M$ is an ambient isotopy of 
$M$ of compact support that ends up with a homeomorphism
preserving the subset $\Sigma$.
The {\it motion group} $\Motion(M,\Sigma)$ of $\Sigma$ in $M$ 
is the group made up of the equivalence classes of such motions
where the product is defined by concatenation of ambient isotopies.
The concept of a motion has its origin in the braid group,
which can be regarded as the motion group of a finite set in the plane.  
In his 1962 PhD thesis \cite{Dahm62} supervised by Fox,
Dahm developed a general theory of motions and
calculated the motion group of a trivial link in the Euclidean space. 
In \cite{Goldsmith81}, Goldsmith published an exposition of Dahm's thesis,
and in the succeeding paper \cite{Goldsmith82},
she obtained generators and relations of the motion groups
of torus links in $S^3$.
Since then (variations of) motion groups have been studied 
by many researchers in various settings.
(See \cite{BrendleHatcher13, Damiani-Kamada19, Fukushima20} and references therein.)

In the case where $M$ is a closed, orientable $3$-manifold and $\Sigma$
is a Heegaard surface,
Johnson-Rubinstein \cite{JohnsonRubinstein13} and 
Johnson-McCullough \cite{JohnsonMcCullough13} studied
the (smooth) motion group $\Motion(M,\Sigma)$ and 
its quotient group $G(M,\Sigma)$ defined by
\begin{align*}
G(M,\Sigma)=
& \{[f]\in \MCG(\Sigma) \mid 
\mbox{There exists a motion $\{f_t\}_{t\in I}$ with $j\circ f=f_1|_{\Sigma}$.}\}\\
= &
\{[f]\in \MCG(\Sigma) \mid 
\mbox{$j\circ f:\Sigma\to M$ is ambient isotopic to $j$.}\},
\end{align*}
where $\MCG(\Sigma)$ is the mapping class group of $\Sigma$ and $j:\Sigma\to M$ is the inclusion map.
These groups are also intimately related to the pairwise mapping class group $\MCG(M,\Sigma)$ 
first studied by Goeritz \cite{Goeritz33} in 1933,
which has been attracting attention of various researchers.
See Section \ref{sec:The homotopy motion groups of Heegaard surfaces and their friends}
for a brief summary.

Motivated by Minsky's question \cite{Gordon07} on 
the homotopical behavior of simple loops on a Heegaard surface
in the ambient $3$-manifold (see Question \ref{question:Minsky} below)
and the second author's joint work with Donghi Lee \cite{LeeSakuma12}
on the corresponding problem for $2$-bridge spheres of $2$-bridge links
(see Subsection \ref{subsec:motivation}),
we are naturally lead to 
a homotopy version of the motion group $\Motion(M,\Sigma)$
and that of the group $G(M,\Sigma)$. 

A {\it homotopy motion} of a closed surface $\Sigma$ in 
a compact $3$-manifold $M$ is
a homotopy $F=\{f_t\}_{t\in I}:\Sigma\times I \to M$,
such that the initial end $f_0$ is the inclusion map $j:\Sigma\to M$ and 
the terminal end $f_1$ is an embedding with image $\Sigma$,
where $f_t:\Sigma\to M$ ($t\in I=[0,1]$) is the continuous map from $\Sigma$ to $M$
defined by $f_t(x)=F(x,t)$.
The {\it homotopy motion group} $\Pi(M,\Sigma)$ is the group of
equivalence classes of homotopy motions of $\Sigma$ in $M$,
where the product is defined by concatenation of homotopies
(see Section $\ref{sec:The homotopy motion groups of incompressible surfaces in Haken manifolds}$ for the precise definition).
There is a natural homomorphism $\partial_+:\Pi(M,\Sigma)\to \MCG(\Sigma)$
which assigns 
(the equivalence class of) a homotopy motion 
with (the mapping class represented by) its terminal end. 
We denote the image of $\partial_+$ by $\Mdy(M,\Sigma)$. 
Then we have
\begin{align*}
\Mdy(M,\Sigma)
&=
\{[f]\in \MCG(\Sigma) \mid 
\mbox{$j\circ f:\Sigma\to M$ is homotopic to the inclusion map $j$.}\}.
\end{align*} 
By denoting the kernel of $\partial_+$ by $\mathcal{K} (M, \Sigma)$,
we have the following exact sequence.
\begin{align}
\label{exact sequence for homotopy motion group}
\xymatrix{
1 \ar[r] & \mathcal{K} (M, \Sigma) \ar[r] & \Pi (M, \Sigma) \ar[r]^{\partial_+} &  \Mdy (M, \Sigma) \ar[r]  & 1 .  }
\end{align}
In the case where $M$ is a closed, orientable $3$-manifold and $\Sigma$
is a Heegaard surface, the above exact sequence  
is a homotopy version
of the following exact sequence studied by Johnson-McCullough \cite{JohnsonMcCullough13}. 
\begin{align}
\label{JohnsonMcCullough exact sequence for homotopy motion group}
\xymatrix{
1 \ar[r] & \pi_1(\Diff(M)) \ar[r] & \Motion (M, \Sigma) \ar[r] &  G (M, \Sigma) \ar[r]  & 1,}
\end{align}
where $\Diff(M)$ is the space of diffeomorphisms of $M$.
(The smooth motion group $\Motion (M, \Sigma)$ corresponds to 
$\mathcal{H}_1(M,\Sigma)$
in \cite{JohnsonMcCullough13},
the fundamental group of the space $\mathcal{H}(M,\Sigma)$
of Heegaard surfaces equivalent to $(M,\Sigma)$.) 

The purpose of this paper is 
to give a systematic study of 
the homotopy motion group $\Pi (M, \Sigma)$
and the related groups in the exact sequence (\ref{exact sequence for homotopy motion group})
for a closed, orientable surface $\Sigma$ in a closed, orientable $3$-manifold $M$. 
 
\subsection{Motivation}
\label{subsec:motivation}

Before stating the main results, 
we explain our motivation.
Let $\Sigma$ be a Heegaard surface of a closed, orientable $3$-manifold $M$,
and let $V_1$ and $V_2$ be the handlebodies obtained by cutting $M$ along $\Sigma$.
Let $\Mdy(V_i)$ be the kernel of the homomorphism
$\MCG(V_i)\to \Out(\pi_1(V_i))$ ($i=1,2$).
Now, let $\Curve(\Sigma)$ be the set of
the isotopy classes of essential loops on $\Sigma$.
Let $\Delta_i\subset\Curve(\Sigma)$ be the set of (isotopy classes of) meridians,
i.e., the essential loops on $\Sigma$ that bound disks 
in $V_i$. 
Set $\Delta := \Delta_1 \cup \Delta_2$. 
Let $Z\subset\Curve(\Sigma)$ be the set of (isotopy classes of) 
essential loops on $\Sigma$ that are null-homotopic in $M$. 
In \cite[Question 5.4]{Gordon07}, Minsky raised the following question.

\begin{question}
\label{question:Minsky}
{\rm
When is $Z$ equal to the orbit 
$\langle \Mdy (V_1), \Mdy(V_2)\rangle \Delta$?
}
\end{question}

Note that the group $\Mdy(V_i)$ is identified with 
the group $\Mdy(V_i,\Sigma)=\partial_+(\Pi(V_i,\Sigma))<\MCG(\Sigma)$
and so $\langle \Mdy (V_1), \Mdy(V_2)\rangle$ 
is contained in 
the group $\Mdy(M,\Sigma)=\partial_+(\Pi(M,\Sigma))$.
Moreover, since the action of $\Mdy(M,\Sigma)$ on $\Curve(\Sigma)$
preserves the homotopy classes of loops in the ambient manifold $M$, we have
\[
\langle \Mdy (V_1), \Mdy(V_2)\rangle \Delta 
\ \subset \
\Mdy(M,\Sigma)\Delta
\ \subset \ Z.
\] 
This suggests that it is natural to work with the group 
$\Mdy(M,\Sigma)$ rather than the group
$\langle \Mdy (V_1), \Mdy(V_2)\rangle$ for Question \ref{question:Minsky},
and we have the following refinement of the question.

\begin{question}
\label{question:Refined:Homotopy class of curves in 3-manifold}
{\rm
Let $\Sigma$ be a Heegaard surface of a closed, orientable $3$-manifold $M$.
\begin{enumerate}
\item
When is $Z$ equal to the orbit 
$\Mdy(M,\Sigma) \Delta$? 
\item
Let $\kappa: \Curve(\Sigma)/\Mdy(M,\Sigma) \to \Curve(\Sigma)/\simeq_M$
be the projection,
where $\simeq_M$ is the equivalence relation on $\Curve(\Sigma)$
induced by homotopy in $M$,
namely two essential simple loops of $\Sigma$ are equivalent 
with respect to $\simeq_M$
if they are homotopic in $M$. 
Then how far is the map $\kappa$ from being injective?
In particular, when is the restriction of $\kappa$ to 
$(\Curve(\Sigma)-Z)/\Mdy(M,\Sigma)$ injective?
\end{enumerate}
}
\end{question}

The corresponding question for $2$-bridge spheres
for $2$-bridge links were completely solved by Lee-Sakuma \cite{LeeSakuma12,LeeSakuma14},
and applied the study of epimorphisms among $2$-bridge knot groups
\cite[Theorem 8.1]{AimiLeeSakaiSakuma20})
and variantions of McShane's identity \cite{LeeSakuma13}
(see \cite{LeeSakuma11} for summary).
This paper, as well as Ohshika-Sakuma \cite{OhshikaSakuma16}, 
is motivated by the natural question to what extent these results hold
in general setting.

To explain the main question treated in this paper,
we note the following facts that easily follow from 
\cite{LeeSakuma14} (cf. \cite{KodaSakuma20b}).
(Below, we use the same symbol $(M, \Sigma)$ for a $2$-bridge decomposition by abusing notation.)

\begin{itemize}
\item
If the Hempel distance 
of the $2$-bridge sphere is $\ge 3$, then
\[
\Mdy(M,\Sigma) =
\langle \Mdy (V_1), \Mdy(V_2)\rangle ,
\]
whereas if the Hempel distance 
of the $2$-bridge sphere is $2$, 
then
\[
\Mdy(M,\Sigma) \gneq
\langle \Mdy (V_1), \Mdy(V_2)\rangle .
\]
In the latter case, the index 
$[\Mdy(M,\Sigma) :\langle \Mdy (V_1), \Mdy(V_2)\rangle ]$
is $2$,
and the gap arises from the open book structure of the link complement
whose binding is the axis of the $2$-strand braid 
representing the $2$-bridge torus link
(see \cite[p.5]{LeeSakuma14} and Section \ref{sec:Open book rotations}).

Moreover, in both cases, the image of $\langle \Mdy (V_1), \Mdy(V_2)\rangle$
in the automorphism group of the curve complex of the $4$-times punctured sphere
is isomorphic to the free product of those of $\Mdy (V_1)$ and $\Mdy(V_2)$.
\end{itemize}

Thus the following question naturally arises.

\begin{question}
\label{question:motivation}
{\rm
(1) When is the group $\langle \Mdy (V_1), \Mdy(V_2)\rangle$ equal to $\Mdy(M,\Sigma)$?

(2) When is the group $\langle \Mdy (V_1), \Mdy(V_2)\rangle$ equal to the free product
$\Mdy (V_1) * \Mdy(V_2)$?}
\end{question}

A partial answer to the second question was given by
Bowditch-Ohshika-Sakuma in \cite[Theorem B]{OhshikaSakuma16}
(cf. Bestvina-Fujiwara \cite[Section 3]{BestvinaFujiwara17}), 
which says that if the Hempel distance is large enough,
then the orientation-preserving subgroup
$\langle \Mdy^+ (V_1), \Mdy^+ (V_2)\rangle$
is equal to the free product 
$\Mdy^+ (V_1) *  \Mdy^+(V_2)$.

\subsection{Main results}

A main purpose of this paper is
to give the following partial answer to Question \ref{question:motivation}(1).

\begin{citingtheorem}[Theorem \ref{thm:well-definedness of homological degree for aspherical manifolds}] 
Let $M = V_1 \cup_{\Sigma} V_2$ be a Heegaard splitting of 
a closed, orientable $3$-manifold $M$ induced from an open book decomposition.  
If $M$ has an aspherical prime summand,  
then we have $\langle \Mdy (V_1), \Mdy (V_2) \rangle  \lneq \Mdy (M, \Sigma)$. 
\end{citingtheorem}

To prove this theorem we construct a $\ZZ^2$-valued invariant of $\Mdy (M, \Sigma)$,
i.e., a map
$\Deg : \Mdy (M, \Sigma) \to \ZZ^2$,
such that its mod $2$ reduction vanishes on the subgroup 
$\langle \Mdy (V_1), \Mdy (V_2) \rangle$. 
This actually comes from a natural 
invariant $\widehat{\Deg} : \Pi (M, \Sigma) \to \ZZ^2$, 
where the well-definedness of $\Deg$ is equivalent to the vanishing of $\widehat{\Deg}$
on the subgroup $\mathcal{K}(M,\Sigma)$.

An element $\alpha$ of $\mathcal{K}(M,\Sigma)$ is represented by a homotopy motion
$F=\{f_t\}_{t\in I}:\Sigma\times I \to M$,
such that both $f_0$ and $f_1$ are equal to the inclusion map $j:\Sigma\to M$.
Thus $F$ determines a continuous map $\hat F:\Sigma\times S^1\to M$. 
Though the homotopy class of $\hat F$ is not 
always uniquely determined by 
$\alpha \in \mathcal{K}(M,\Sigma)$, 
its degree is uniquely determined by $\alpha$, and so
we have a homomorphism 
$\deg : \mathcal{K}(M,\Sigma)\to\ZZ$
(see Definition \ref{def:degree-KernelGroup}). 
The map $\Deg : \Mdy (M, \Sigma) \to \ZZ^2$ 
is well-defined if and only if
the homomorphism $\deg : \mathcal{K}(M,\Sigma)\to\ZZ$ vanishes
(see the paragraph just before
Proposition \ref{prop:well-definedness of homological degree for aspherical manifolds}).
The problem of whether this condition holds can be regarded as
a refined version 
of the problem of dominations  
among $3$-manifolds,
which has been a subject of extensive literatures
(see e.g. \cite{Wang02, KotschickNeofytidis13, Neofytidis18} and references therein). 

\begin{definition}
\label{def:Sigma-domination}
{\rm
We say that a  closed, orientable surface $\Sigma$ in a closed, 
orientable $3$-manifold $M$ 
(or a pair $(M, \Sigma)$) is \textit{dominated by} $\Sigma \times S^1$ if 
there exists a map $\phi:\Sigma \times S^1\to M$ such that 
$\phi|_{\Sigma \times \{ 0 \} } $ is an embedding with 
image $\Sigma \subset M$
and that the degree of $\phi$ is non-zero.}
\end{definition}

Clearly, the homomorphism $\deg : \mathcal{K}(M,\Sigma)\to\ZZ$ vanishes
if and only if 
$(M, \Sigma)$ is not dominated by $\Sigma \times S^1$.
 
We study the question of which Heegaard splitting $(M,\Sigma)$ is 
dominated by $\Sigma \times S^1$, and give a complete answer for the case where
$M$ is irreducible (Theorem \ref{thm:non-zero degree maps 1})
and a partial answer for the generic case (Theorem \ref{thm:non-zero degree maps 2}). 
In particular, we show that
if $M$ has an aspherical prime summand then
$(M,\Sigma)$ is not dominated by $\Sigma \times S^1$
for any Heegaard surface $\Sigma$ of $M$.
This guarantees 
the existence of the map $\Deg : \Mdy (M, \Sigma) \to \ZZ^2$
when $M$ has an aspherical prime summand, 
and Theorem \ref{thm:well-definedness of homological degree for aspherical manifolds}
is proved by using this fact.

We remark here that Theorems \ref{thm:non-zero degree maps 1}
and \ref{thm:non-zero degree maps 2} are
intimately related with the result of
Kotschick-Neofytidis \cite[Theorem 1]{KotschickNeofytidis13},
which says that a closed, orientable 3-manifold $M$ is 
dominated by 
a product $\Sigma \times S^1$ for some closed, 
orientable surface $\Sigma$ if and only if 
$M$ is finitely covered by either 
a product $F \times S^1$, for some aspherical surface $F$, 
or a connected sum 
$\#_g (S^2 \times S^1)$ for some 
non-negative integer $g$. 
(In \cite{KotschickNeofytidis13} and the present paper, we employ the
usual convention that the empty connected sum $\#_0 (S^2 \times S^1)$ represents $S^3$.) 
Thus part of our non-existence result for domination follows from their result.
However, the construction of dominating maps in Theorem \ref{thm:non-zero degree maps 2}
require more subtle arguments,
for we impose that the product $\Sigma \times S^1$ 
dominates not only the manifold $M$ itself but also the pair $(M, \Sigma)$.

\smallskip

In this paper, we also study incompressible surfaces in Haken manifolds.
In Theorem \ref{thm:Haken} and Corollary \ref{cor:Haken},
we completely describe the structures of their homotopy motion groups and related groups. 
The proof of that theorem is inspired by the work of 
Jaco-Shalen \cite{JacoShalen76} 
(see also \cite[Chapter VII]{Jaco80}),
where they introduced the concept of a {\it spatial deformation}
of a subset $\Sigma$ in the boundary of a manifold.
The concept of a homotopy motion is also regarded as a variation
of that of a spatial deformation.
As in \cite{JacoShalen76} and \cite[Chapter 5]{Jaco80},
the proof of Theorem \ref{thm:Haken} uses 
the covering spaces of compact $3$-manifolds corresponding to the surface fundamental groups.

The opposite case where $\Sigma$ is {\it homotopically trivial}, in the sense that
the inclusion map $j:\Sigma\to M$ is homotopic to the contant map, is 
studied as well (see Theorem \ref{thm:homotopy motion group for a local surface}).
In that case, we prove that if $M$ is aspherical then
$\Pi (M, \Sigma) \cong \pi_1(M) \times \MCG(\Sigma)$:
the factors $\pi_1(M)$ and $\MCG(\Sigma)$ correspond to
$\mathcal{K}(M,\Sigma)$ and $\Mdy(M,\Sigma)$, respectively.
Conversely, if $\Mdy(M,\Sigma)=\MCG(\Sigma)$ then $\Sigma$ is homotopically trivial
provided that $M$ is irreducible
(Corollary \ref{cor:homotopy motion group for a local surface}). 

Our interest in the group $\Mdy(M,\Sigma)$
has also its origin 
in the second author's observation in \cite[Addendum 1]{Sakuma81}
(cf. \cite{Brooks85, Montesinos87, HiroseKin20}),
called the virtual branched fibration theorem,
which says that, 
for every Heegaard surface $\Sigma$ of a closed, orientable $3$-manifold $M$,
there exists a double branched covering of $M$ 
which fibers over the circle, such that
the inverse image of $\Sigma$ is the union of two fiber surfaces. 
We show that this theorem is intimately related to the subgroup
$\langle \Mdy (V_1), \Mdy(V_2)\rangle$ of $\Mdy(M,\Sigma)$.
Let $\Inv(V_i)$ ($\subset \MCG(\Sigma)$)
be the set of torsion elements of $\Mdy(V_i)$. 
By slightly refining
the arguments
of Zimmermann \cite[Proof of Corollary 1.3]{Zimmermann79}, we can 
see that 
this is nothing but the set of vertical $I$-bundle involutions of $V_i$ (Lemma \ref{lem:I-bundle-Involution}). 
Here, a {\it vertical $I$-bundle involution} of a handlebody $V$ is an involution $h$ 
for which there exists an $I$-bundle structure of $V$
such that $h$ preserves each fiber setwise and acts on it as a reflection.
We then prove the following refinement of \cite[Addendum 1]{Sakuma81}.

\begin{citingtheorem}[Theorem  \ref{thm:vbf-theorem}] 
Let $M=V_1\cup_{\Sigma} V_2$ be a Heegaard splitting of 
a closed, orientable $3$-manifold $M$.
Then there exists a double branched covering $p:\tilde M \to M$
that satisfies the following conditions.
\begin{enumerate}
\renewcommand{\labelenumi}{$(\roman{enumi})$}
\item
$\tilde M$ is a surface bundle over $S^1$
whose fiber is homeomorphic to $\Sigma$.
\item
The preimage $p^{-1}(\Sigma)$ of the Heegaard surface $\Sigma$
is a union of two $($disjoint$)$ fiber surfaces.
\end{enumerate}
Moreover, the set $D(M,\Sigma)$ of monodromies of such bundles
is equal to the set $\{h_1 \circ h_2 \mid h_i\in \Inv(V_i)\}$,
up to conjugation and inversion.
\end{citingtheorem}

\subsection{Structure of the paper}

This paper is organized as follows.
In Section \ref{sec:The homotopy motion groups},
we give formal definitions of the homotopy motion group $\Pi(M,\Sigma)$,
its subgroup $\mathcal{K}(M,\Sigma)$ and its quotient group $\Mdy(M,\Sigma)$. 
In Section \ref{sec:Basic properties}, we present basic properties of these groups
for surfaces in $3$-manifolds. 
Section \ref{sec:The homotopy motion groups of incompressible surfaces in Haken manifolds}
is devoted to the case where $\Sigma$ is an incompressible surface in 
a Haken manifold $M$. 
Section \ref{sec:homotopically trivial case} 
treats the opposite case where $\Sigma$ is homotopically trivial.
The remaining sections are devoted to the case 
where $\Sigma$ is a Heegaard surface.
In Section \ref{sec:The homotopy motion groups of Heegaard surfaces and their friends},
we recall various natural subgroups of $\MCG(\Sigma)$
associated with a Heegaard surface, and describe their relationships with
the group $\Mdy(M,\Sigma)$. 
In Section \ref{sec:Open book rotations},
we consider the Heegaard splitting obtained from an open book decomposition,
and introduce two homotopy motions,
the half book rotation $\rho$ and the unilateral book rotation $\sigma$,
which play key roles in the proofs of the main theorems 
given in the succeeding 
two sections.
In Section \ref{sec:The group K(M,Sigma) for Heegard surfaces},
we study the problem of which Heegaard surface $(M, \Sigma)$ is dominated by
$\Sigma \times S^1$. 
In Section \ref{sec:Gap between Mdy(M,Sigma) and the subgroup Mdy (V_1), Mdy (V_2)}, 
we discuss gaps between $\Mdy(M,\Sigma)$ and the subgroup $\langle \Mdy (V_1), \Mdy (V_2) \rangle$, 
and 
prove Theorem \ref{thm:well-definedness of homological degree for aspherical manifolds}, which  
provides a partial answer to Question \ref{question:motivation}(1).
In Section \ref{sec:The virtual branched fibration theorem and the group Mdy (V_1), Mdy (V_2)},
we prove the branched fibration theorem (Theorem \ref{thm:vbf-theorem}), 
which gives another motivation for defining and studying 
the group $\Mdy(M,\Sigma)$.

\medskip
{\bf Acknowledgement.}
Part of this work was first announced at a zoom workshop held at RIMS
in May 2020, and it is summarized in the unrefereed conference paper \cite{KodaSakuma20a},
which includes brief outlines of some of the proofs. 
Both authors would like to thank the organizers, 
Tomotada Ohtsuki and Hirotaka Akiyoshi, 
for giving us the opportunity to announce the results in a difficult time. 
The second author would like to thank Ken Baker 
for pointing out the importance of book rotations \cite{Baker}
that motivated the definition of the group $\Mdy(M,\Sigma)$.
Both authors would like to thank Michel Boileau 
for his valuable suggestions \cite{Boileau} concerning 
nonzero degree maps from $\Sigma\times S^1$ to a closed, orientable $3$-manifold.
His suggestions were indispensable for the proof of
Theorems \ref{thm:non-zero degree maps 1}
and \ref{thm:non-zero degree maps 2}. 
Both authors would also like to thank
Norbert A'Campo, Jos\'e Maria Montesinos, and Christoforos Neofytidis
whose suggestions greatly improved this paper. 
The authors would like to thank the anonymous referee
for his or her valuable comments and suggestions 
that helped them to improve the exposition. 

\section{The homotopy motion groups}
\label{sec:The homotopy motion groups} 

Let $X$ and $Y$ be topological spaces. 
We denote by $C(X, Y)$ the space of continuous maps from $X$ to $Y$,
endowed with the compact-open topology.
For a subspace $A$ of $X$, we denote by 
$J(A, X)$ the subspace of $C(A, X)$ consisting of 
embeddings of $A$ into $X$ with image $A=j(A)$, where 
$j : A \to X$ is the inclusion map. 
For subspaces $A_1,\ldots,A_n$ of $X$, 
let 
$\Homeo(X,A_1,\ldots,A_n)$
denote the topological group of self-homeomorphisms of $X$ 
that preserves each $A_i$ ($1\le i\le n$). 
By $\MCG(X,A_1,\ldots,A_n)$ 
we mean the {\it mapping class group}
of $(X,A_1,\ldots,A_n)$, i.e., 
the group of connected components
of $\Homeo(X,A_1,\ldots,A_n)$.
We usually do not distinguish notationally between 
$f \in \Homeo(X,Y_1,\ldots,Y_n)$ 
and the element $[f]\in \MCG(X,A_1,\ldots,A_n)$ represented by $f$. 
Note that we allow orientation-reversing homeomorphisms when $X$ is an orientable manifold, 
so our $\MCG(X,A_1,\ldots,A_n)$ is what is often called the extended mapping class group.
A ``plus" symbol, as in $\MCG^+(X,A_1,\ldots,A_n)$, indicates 
the subgroup, of index $1$ or $2$,
consisting of the elements represented by orientation-preserving
homeomorphisms of $X$. 

{\it Throughout the paper, we identify $S^1$ with $\RR / \ZZ$. 
In our notation, we will not distinguish between an element of $S^1$ and 
its representative in $\RR$.}

\vspace{1em}

Let $\Sigma$ be a subspace of a manifold $M$,
and $j : \Sigma \to M$ the inclusion map. 
In this section, we first introduce formal definitions of the 
homotopy motion group 
$\Pi(M,\Sigma)$, its subgroup $\mathcal{K}(M,\Sigma)$, 
and the quotient group $\Gamma (M, \Sigma)$. 

Let $C (\Sigma, M)$ be the space of continuous maps from $\Sigma$ to $M$, 
and $J (\Sigma, M)$ the subspace of $C (\Sigma, M)$ consisting of embeddings 
of $\Sigma$ into $M$ with image $j (\Sigma)$. 
We call a path 
\[
\alpha:(I, \{ 1\}, \{ 0 \}) \to (C (\Sigma, M) , J(\Sigma, M) , \{j\} )
\]
a {\it homotopy motion} of $\Sigma$ in $M$. 
We call the maps $\alpha(0)$ and $\alpha(1)$ from $\Sigma$ to $M$
the {\it initial end} and the {\it terminal end}, respectively, 
of the homotopy motion. 
Two homotopy motions $\alpha, \beta: 
(I, \{ 1\}, \{ 0 \}) \to (C (\Sigma, M) , J(\Sigma, M) , \{j\} )$ 
are said to be {\it equivalent} 
if they are homotopic via a homotopy through maps of the same form. 
We remark that in that case, 
thinking of the codomains of the two maps $\alpha (1), \beta (1) \in J (\Sigma, M)$ as $\Sigma$, 
$\alpha (1)$ and $\beta (1)$ are isotopic as self-homeomorphisms of $\Sigma$.
When there is no fear of confusion,
we do not distinguish notationally between a homotopy motion
$\alpha$ and the element $[\alpha]$ of $\Pi (M, \Sigma)$ represented by $\alpha$. 

We define 
\[ \Pi (M, \Sigma) := \pi_1 (C (\Sigma, M), J(\Sigma, M) , j)\]
to be the set of equivalence classes of homotopy motions, 
as usual in the definition of relative homotopy groups 
$\pi_n (X, A, x_0)$ for $x_0 \in A \subset X$, where $X$ is a topological space. 
We equip $\Pi (M, \Sigma)$ with a group structure 
as in the following way. 
 
Let $\alpha$ and $\beta$ be homotopy motions. 
Then the {\it concatenation} 
\[\alpha \cdot \beta : 
(I, \{ 1\}, \{ 0 \}) \to (C (\Sigma, M) , J(\Sigma, M) , \{j\} )\] 
of them is defined by 
\[ 
\alpha \cdot \beta (t) = 
\left\{ 
\begin{array}{ll} 
\alpha (2t) & (0 \leq t \leq 1/2) \\
\beta (2t-1) \circ \alpha(1) & (1/2 \leq t \leq 1) .
\end{array}
\right.
\]
We can easily check that the concatenation naturally induces a product of 
elements of $\pi_1 (C (\Sigma, M), J(\Sigma, M))$. 
The {\it identity motion} 
\[e : (I, \{ 1\}, \{ 0 \}) \to (C (\Sigma, M) , J(\Sigma, M) , \{j\} )\]
defined by $e (t) = j$ ($t \in I$) represents the identity element of 
$\Pi (M, \Sigma)$. 
The {\it inverse} $\bar{\alpha}$ of a homotopy motion $\alpha$ is defined by 
\[ \bar{\alpha} (t) = \alpha (1-t) \circ \alpha(1)^{-1},  \]
where we regard $\alpha(1)$ as a self-homeomorphism of $\Sigma$,
and $\alpha(1)^{-1}$ denotes its inverse.
Then the inverse of $[\alpha]$ in the group
$\pi_1 (C (\Sigma, M), J(\Sigma, M))$ is given by $[\bar{\alpha}]$.

\begin{definition}
\label{def:homotopy motion group}
{\rm
We call the group $\Pi (M, \Sigma)$ 
the {\it homotopy motion group} of $\Sigma$ in $M$. 
}
\end{definition}

\begin{remark}
\label{rem:homotopy motion group and the fundamental group}
{\rm
When $\Sigma$ is a single point $x_0$, $\Pi (M, \Sigma)$ is nothing 
but the fundamental group $\pi_1 (M, x_0)$ of $M$. 
Thus, the group $\Pi (M, \Sigma)$ is a sort of generalization of $\pi_1 (M, x_0)$. 
See also Theorem \ref{thm:homotopy motion group for a local surface} below. 
}
\end{remark}

\begin{notation}
\label{notation:homotopy motion}
{\rm 
For a homotopy motion 
\[\alpha:(I, \{ 1\}, \{ 0 \}) \to (C (\Sigma, M) , J(\Sigma, M) , \{j\} )\]
we employ the following notation.
\begin{enumerate}
\item
We occasionally regard $\alpha$ as a continuous map $\Sigma\times I \to M$
defined by $\alpha(x,t)=\alpha(x)(t)$ 
(cf. \cite[Theorem 6.5]{DavisKirk01}, \cite[Introduction 1.9]{Spanier66}). 
\item
When we regard $\alpha$ as a continuous family of maps,
we occasionally write $\alpha=\{f_t\}_{t\in I}$ where $f_t=\alpha(t):\Sigma\to M$.
\item
When $\alpha$ is a closed path, i.e., $\alpha(1)=\alpha(0)=j$,
$\alpha$ induces a continuous map 
$\Sigma\times S^1 \to M$, which we denote by $\hat\alpha$,
that sends $(x,t)\in \Sigma\times S^1$ to $\alpha(t)(x)=\alpha(x,t)$. 
The homotopy class of this map relative to $\Sigma\times\{0\}$ 
is uniquely determined by the element $[\alpha]\in \pi_1 ( C (\Sigma , M ) , j)$.
\end{enumerate}
}
\end{notation}

Since the inclusion map $j$ is nothing but the identity if we think of the codomain of $j$ as $\Sigma$, 
$J (\Sigma, M)$ can be canonically identified with $\Homeo (\Sigma)$. 
Thus, the terminal end $\alpha (1) = f_1$ of a homotopy motion $\alpha = \{ f_t \}_{t \in I} $ can be 
regarded as an element of $\Homeo (\Sigma)$. 
Therefore, we obtain a map 
\[ \partial_+ : \Pi (M, \Sigma) \to \MCG(\Sigma) \]
by taking the equivalence class of a homotopy motion $\alpha = \{f_t\}_{t \in I}$ to 
the mapping class of $\alpha (1) = f_1 \in \Homeo (\Sigma)$. 
Clearly, this map is a homomorphism. 
(To be precise, this holds when we think of $\Homeo (\Sigma)$ as acting on $\Sigma$ 
from the right: under the usual convention where $\Homeo (\Sigma)$ acts on $\Sigma$
from the left, which we employ in this paper,
the map $\partial_+$ is actually an anti-homomorphism.) 
\begin{definition}
\label{def:monodromy group}
{\rm  
We denote the image of $\partial_+$ by $\Gamma (M, \Sigma)$.
Namely,
$\Mdy(M,\Sigma)$ is 
the subgroup of the 
mapping class group $\MCG(\Sigma)$
defined by
\begin{align*}
\Mdy(M,\Sigma)
&=
\{[f]\in \MCG(\Sigma) \mid 
\mbox{There exists a homotopy motion $\{f_t\}_{t\in I}$ with $f=f_1$.}\}\\
&=
\{[f]\in \MCG(\Sigma) \mid 
\mbox{$j\circ f:\Sigma\to M$ is homotopic to the inclusion map $j$.}\}.
\end{align*} 
The kernel of $\partial_+$ is denoted by $\mathcal{K} (M, \Sigma)$: thus we have
the exact sequence (\ref{exact sequence for homotopy motion group}) in the introduction.
}
\end{definition}

\section{Basic properties of homotopy motion groups
of surfaces in $3$-manifolds.}
\label{sec:Basic properties}

In this section, we provide a few basic properties concerning
the groups defined in the above for surfaces in $3$-manifolds, by using 
elementary arguments in homotopy theory.
{\it Throughout the remainder of this paper,
$\Sigma$ denotes a connected, closed, orientable surface embedded in
a connected, orientable $3$-manifold 
$M$, unless otherwise stated,
and $j : \Sigma \to M$ denotes the inclusion.
When we mention the degrees of maps,
we assume that the surfaces and $3$-manifolds are endowed with (suitable) orientations.
}

Note that we have the following long exact sequence. 
\begin{align*}
  \cdots &\to \pi_1 (J(\Sigma, M), j) \xrightarrow{\mathscr{I}} \pi_1 (C(\Sigma, M), j) \to \pi_1 (C(\Sigma, M), J (\Sigma, M), j) \\
  &\to 
\pi_0 (J (\Sigma, M)) \to \pi_0 (C (\Sigma, M)).
\end{align*}
Here, $\mathscr{I}: \pi_1 (J(\Sigma, M), j) \to \pi_1 (C(\Sigma, M), j)$ is the map 
induced from the inclusion $(J(\Sigma, M), j) \hookrightarrow (C(\Sigma, M), j)$.
The boundary map 
$\pi_1 (C(\Sigma, M), J (\Sigma, M), j) \to \pi_0 (J (\Sigma, M))$
respects the group structures of $\Pi(M,\Sigma)=\pi_1 (C(\Sigma, M), J (\Sigma, M), j)$
and $\MCG(\Sigma)=\pi_0 (J (\Sigma, M))$,
and it is identical with the (anti-)homomorphism $\partial_+$.  
Thus we have the following description of the kernel $\mathcal{K} (M, \Sigma)$.

\begin{lemma}
\label{lem:monodromy group and kernel group}
We have the isomorphism
\[ \mathcal{K} (M, \Sigma) \cong \pi_1 ( C (\Sigma , M) , j) / \mathscr{I} ( \pi_1 ( J (\Sigma, M ) , j) ).  \]
Moreover, if the genus of $\Sigma$ is at least $2$, then we have
\[ \mathcal{K} (M, \Sigma) \cong \pi_1 ( C (\Sigma , M ) , j) .  \]
\end{lemma}

\begin{proof}
The first assertion is a direct consequence of the exact sequence.
The second assertion follows from the fact that 
$J (\Sigma, M)$ can be canonically identified with $\Homeo (\Sigma)$ as discussed before,
and the result of Hamstrom \cite{Hamstrom66}
that $\pi_1 (J(\Sigma, M), j)$ is 
the trivial group when the genus of $\Sigma$ is at least $2$. 
\end{proof}

As noted in Notation \ref{notation:homotopy motion}(3),
a closed path
$\alpha:(I, \partial I) \to (C (\Sigma, M), \{j\})$
determines a continuous map $\hat\alpha:\Sigma\times S^1 \to M$
whose homotopy class is uniquely determined by the element 
$[\alpha]\in \pi_1 ( C (\Sigma , M) , j)$.
Thus we have a well-defined map $[\alpha]\mapsto \deg(\hat\alpha)$  
from $\pi_1 ( C (\Sigma , M) , j)$ to $\ZZ$,
which is obviously a homomorphism.
If $[\alpha]$ belongs to the subgroup $\mathscr{I} ( \pi_1 ( J (\Sigma, M ) , j) )$,
then $\hat\alpha$ is homotopic to a map with 
$\image(\hat\alpha)=\Sigma\times\{0\}$, and therefore
$[\alpha]$ belongs to the kernel of the homomorphism.
Hence it descends to a homomorphism 
$\mathcal{K} (M, \Sigma)\cong \pi_1 ( C (\Sigma , M ) , j)/ \mathscr{I} ( \pi_1 ( J (\Sigma, M ) , j) ) \to \ZZ$.

\begin{definition}
\label{def:degree-KernelGroup}
{\rm
We denote by 
$\deg : \mathcal{K} (M, \Sigma)\to \ZZ$
the homomorphism defined by 
\[
\deg([\alpha])=\deg(\hat\alpha:\Sigma\times S^1 \to M),
\]
where $\alpha:(I, \partial I) \to (C (\Sigma, M), \{j\})$
is a closed path.
We call $\deg([\alpha])$ the {\it degree} of 
the element $[\alpha]\in \mathcal{K} (M, \Sigma)$.
}
\end{definition}

\medskip

In order to prove further basic properties of the homotopy motion groups,  
we recall a few results from classical obstruction theory.

\begin{definition}
{\rm
Let $X$ and $Y$ be arcwise-connected topological spaces, 
and let $\theta_i:\pi_1(X,x_0) \to \pi_1(Y, y_i)$ ($i=0,1$)
be homomorphisms, where $x_0\in X$ and $y_0,y_1\in Y$.
Then we say that $\theta_0$ and $\theta_1$ are {\it equivalent}
if there is a path $u:(I,0,1)\to (Y,y_0, y_1)$
such that $\theta_1 =\iota_u\circ \theta_0$
where $\iota_u:\pi_1(Y,y_0) \to \pi_1(Y,y_1)$ 
is the isomorphism induced by $u$. 
In the case where $y_0=y_1$,
$\iota_u$ is the inner-automorphism induced by $[u]\in \pi_1(Y,y_0)$;
so, we say that 
$\theta_0$ and $\theta_1$ are {\it conjugate} 
if they are equivalent.
}
\end{definition}

When we are concerned with the equivalence class of a homomorphism
$\theta:\pi_1(X,x_0)\to \pi_1(Y,y_0)$,
the choice of base point does not matter. 
So, we often drop the description of the base points and
denote the homomorphism by $\theta:\pi_1(X)\to \pi_1(Y)$.

\begin{proposition}
\label{prop:obstruction1}
Let $X$ be a connected $n$-dimensional CW-complex
and $Y$ an arcwise connected topological space.

{\rm (1)} Suppose $\pi_r(Y)=0$ for every $r$ with $1<r<n$. 
Then any homomorphism 
$\theta:\pi_1(X)\to \pi_1(Y)$
is realized by a continuous map, namely there is a continuous map
$f:X\to Y$ such that $f_*$ is equivalent to $\theta$.
To be precise, 
if $\theta$ is a homomorphism from 
$\pi_1(X,x_0)$ to $\pi_1(Y, y_0)$
$($$x_0\in X$, $y_0\in Y$$)$,
then there is a continuous map
$f:(X,x_0)\to (Y,y_0)$ such that $f_*=\theta$.

{\rm (2)} Suppose $\pi_r(Y)=0$ for every $r$ with $1<r \leq n$.  
Then two continuous maps $f_0$ and $f_1:X\to Y$ are homotopic
if and only if 
they induce equivalent homomorphisms between the fundamental groups.
\end{proposition}

\begin{proof}
(1) It is obvious that any homomorphism $\theta$ is realized by a 
continuous map from the $2$-skeleton of $X$ to $Y$. 
For $r$ with $1<r<n$, the condition $\pi_r(Y)=0$ guarantees that
any continuous map from the $r$-skeleton of $X$ to $Y$
extends over the $(r+1)$-skeleton (see \cite [Theorem 7.1(1)]{DavisKirk01}).
By applying this fact inductively, we obtain (1).

(2) This is obtained by a similar inductive argument 
by using \cite [Theorem 7.12]{DavisKirk01} (cf. \cite[Theorem 25.3]{Olum50}).
\end{proof}

We also need the following relative version of 
Proposition \ref{prop:obstruction1}(1).

\begin{proposition}
\label{prop:obstruction2}
Let $(X,X_0)$ be a relative CW-complex of dimension $n$
$($i.e., $X$ is a topological space obtained from $X_0$ by attaching cells of 
dimension $\le n$$)$, and let $x_0\in X_0$.
Let $f:(X_0,x_0)\to (Y,y_0)$ be a continuous map to an arcwise connected topological space $Y$
with base point $y_0$.
Suppose $\pi_r(Y)=0$ for every $r$ with $1<r<n$.
Then $f$ extends to a continuous map from $X$ to $Y$
if and only if there is a homomorphism
$\theta:\pi_1(X,x_0)\to \pi_1(Y,y_0)$
such that $f_*=\theta\circ i_*:\pi_1(X_0,x_0)\to \pi_1(Y, y_0)$
where $i:X_0\to X$ is the inclusion map.
\end{proposition}

\begin{proof}
This can be proved by an inductive argument using \cite[Theorem 25.1]{Olum50}.
\end{proof}

The following result,
which is a consequence of
\cite[Theorems IIa and IIb]{Olum53b}, 
refines  Proposition \ref{prop:obstruction1}(2)
in the case where $X$ and $Y$ are closed, orientable $n$-manifolds.

\begin{proposition}
\label{prop:obstruction3}
Let $X$ and $Y$ be connected, closed, oriented $n$-manifolds,
and assume that $\pi_r(Y)=0$ for every $r$ with $1<r<n$
and that $\pi_1(Y)$ is finite.
Then two continuous maps $f_0$ and $f_1:X \to Y$ are homotopic
if and only if the homomorphisms $(f_i)_*:\pi_1(X) \to \pi_1(Y)$ $(i=0,1)$
are equivalent and $\deg (f_0)= \deg (f_1)$.

Moreover, for a given homomorphism $\theta:\pi_1(X)\to \pi_1(Y)$,
the set of $\deg (f)$, where $f:X\to Y$ runs over the continuous maps 
such that $f_*:\pi_1(X)\to \pi_1(Y)$ is equivalent to $\theta$,
is of the form $d+|\pi_1(Y)|\cdot \ZZ$ for some $d \in \ZZ$.
\end{proposition}

\medskip
Now we state two basic properties 
(Lemmas \ref{lem:monodromy-fundgp} and 
\ref{lem:classification of pi1(C(S,M), j) for an aspherical manifold}) 
of the homotopy motion groups and related groups,
which are obtained by using the above results.
The following lemma gives a characterization of 
the group $\Mdy (M, \Sigma)$ in terms of the induced homomorphisms 
between the fundamental groups. 

\begin{lemma}
\label{lem:monodromy-fundgp}
Let $\Sigma$ be a closed, orientable surface embedded in
a $3$-manifold $M$, and let $f$ be a self-homeomorphism of $\Sigma$.
Then the following hold.

\begin{enumerate}\renewcommand{\labelenumi}{$(\arabic{enumi})$}
\item 
If the mapping class $[f] \in \MCG(\Sigma)$ 
belongs to the subgroup $\Mdy(M,\Sigma)$,
then the homomorphism $(j\circ f)_*:\pi_1(\Sigma) \to \pi_1(M)$
is equivalent to the homomorphism $j_*$.
\item
Suppose $M$ is irreducible. Then the converse to the above also holds.
\end{enumerate}
\end{lemma}

\begin{proof}
The first assertion is obvious from the definition of $\Mdy(M,\Sigma)$.
To prove the second assertion, assume that $M$ is irreducible
and let $[f]\in \MCG(\Sigma)$ be a mapping class 
such that $(j\circ f)_*$ is equivalent to $j_*$.
Then by the sphere theorem we have $\pi_2(M)=0$.
So, we can apply Proposition \ref{prop:obstruction1}(2)  
to show that $j\circ f$ is homotopic to $j$.
Hence $[f]$ belongs to $\Mdy(M,\Sigma)$.
\end{proof}

Now fix a base point $x_0 \in \Sigma \subset M$,
and consider
a closed path 
$\alpha:(I, \partial I) \to (C (\Sigma, M), \{j\})$.
Let $w$ be the element of 
$\pi_1(\Sigma\times S^1, (x_0,0))=\pi_1(\Sigma,x_0)\times \pi_1(S^1,0)$
representing the canonical generator of $\pi_1(S^1,0)$.
Then $\hat\alpha_*(w)$ belongs to the centralizer 
$Z(j_*(\pi_1(\Sigma), x_0),\pi_1(M, x_0))$
of $j_*(\pi_1(\Sigma , x_0))$ in $\pi_1(M, x_0)$,
and it is represented by the based loop $t \mapsto \alpha (t) (x_0)$.
Since the homotopy class of $\hat\alpha : \Sigma \times S^1 \to M $ relative to $(x_0,0)$
is uniquely determined by $[\alpha] \in \pi_1(C (\Sigma, M), j )$,
we obtain the following homomorphism.

\begin{definition}
\label{def:classification of pi1(C(S,M), j) for an aspherical manifold}
We denote by $\Phi$ the homomorphism
\[
\Phi: \pi_1(C (\Sigma, M), j ) \to Z(j_*(\pi_1(\Sigma,x_0)),\pi_1(M, x_0)), \
\Phi([\alpha]) = [u],
\]
where $\alpha:(I, \partial I) \to (C (\Sigma, M), \{j\})$ and 
$u: (I, \partial I) \to (M, \{x_0\}), ~u(t)=\alpha (t) (x_0)$.
\end{definition}

The next lemma 
plays important roles in the proofs of 
Theorems \ref{thm:Haken}, \ref{thm:homotopy motion group for a local surface}  and
\ref{thm:non-zero degree maps 1}.

\begin{lemma}
\label{lem:classification of pi1(C(S,M), j) for an aspherical manifold}
Let $\Sigma$ be a closed, orientable surface embedded in a 
$3$-manifold $M$, and $x_0\in\Sigma$.
Then the following hold.
\begin{enumerate}\renewcommand{\labelenumi}{$(\arabic{enumi})$}
\item
If $M$ is irreducible, then $\Phi$ is surjective.
\item
If $M$ is aspherical, then $\Phi$ is injective.
\end{enumerate}
\end{lemma}

\begin{proof}
(1)
Assume that $M$ is irreducible, and let $[u]$ be an element of 
$Z(j_*(\pi_1(\Sigma, x_0)),\allowbreak \pi_1(M, x_0))$.
Consider the pair 
$(X,X_0)=(\Sigma\times I, \Sigma\times\partial I \cup \{ x_0 \}  \times I)$
and the map $F:X_0\to M$ defined by
$F(x,0)=F(x,1)=x$ ($x\in \Sigma$) and 
$F(x_0,t)=u(t)$ ($t\in I$).
Put $\hat x_0=(x_0,0) \in X$ and let
$\theta:\pi_1(X,\hat x_0)\to \pi_1(M, x_0)$ be the homomorphism induced by
$j\circ p:(X,\hat x_0)\to (M, x_0)$ where
$p:X=\Sigma\times I \to \Sigma$ is the projection.
Then, since $u$ represents an element of $Z(j_*(\pi_1(\Sigma,x_0)),\pi_1(M, x_0))$,
we see that $F_*:\pi_1(X_0,\hat x_0) \to \pi_1(M, x_0)$
is identical with $\theta\circ i_*$, 
where $i:X_0 \to X$ is the inclusion. See Figure \ref{fig_extension}. 
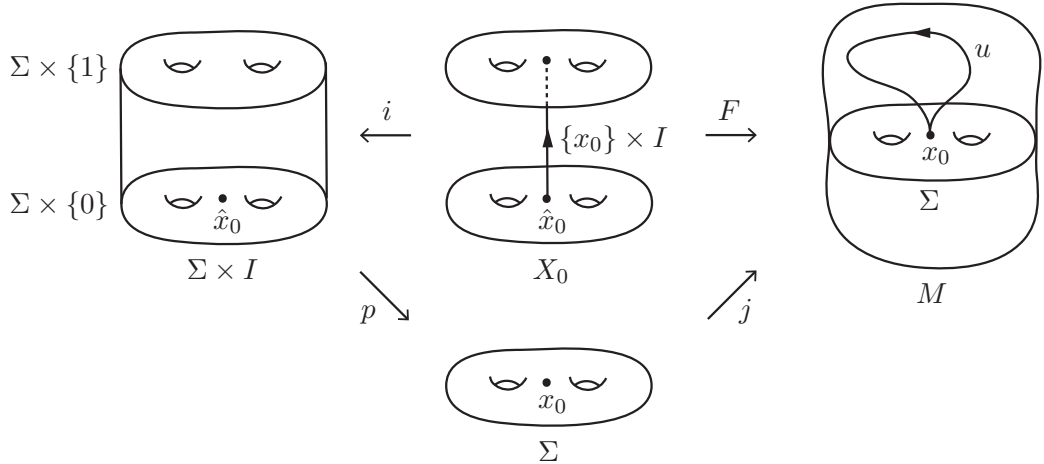
\begin{figure}[htbp]
\centering\includegraphics[width=13.5cm]{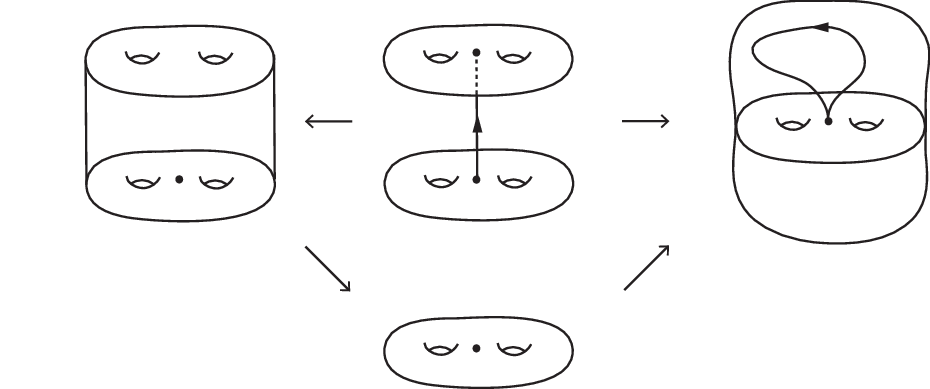}
\begin{picture}(400,0)(0,0)
\put(0,146){$\Sigma \times \{1\}$}
\put(0,94){$\Sigma \times \{0\}$}
\put(67,69){$\Sigma \times I$}
\put(77,88){$\hat{x}_0$}
\put(141,129){$i$}
\put(198,69){$X_0$}
\put(200,88){$\hat{x}_0$}
\put(208,119){$\{ x_0 \} \times I$}
\put(268,129){$F$}
\put(133,54){$p$}
\put(200,20){$x_0$}
\put(200,0){$\Sigma$}
\put(276,54){$j$}
\put(344,95){$\Sigma$}
\put(343,60){$M$}
\put(345,114){$x_0$}
\put(365,153){$u$}
\end{picture} 
\caption{The maps $F$, $i$ and $p$.}
\label{fig_extension}
\end{figure}
Since $\pi_2(M)=0$ by the irreducibility of $M$,
we see by Proposition \ref{prop:obstruction2}
that the map $F$ extends over $X=\Sigma\times I$.
The resulting map $F:\Sigma\times I \to M$ determines a closed path, $\alpha$, in $C (\Sigma, M)$
based on $j$,
and the image of $[\alpha]\in \pi_1(C (\Sigma, M), j )$
by $\Phi$ is equal to $[u]$.
Hence $\Phi$ is surjective.

(2) 
Assume that $M$ is aspherical, 
and let $[\alpha]$ be an element of $\pi_1(C (\Sigma, M), j )$ 
which is contained in $\ker \Phi$. 
Consider the maps $\hat \alpha$ and $\hat e:\Sigma\times S^1 \to M$
induced by $\alpha$ and the identity motion $e$.
Since $[\alpha]\in \ker \Phi$,
the homomorphisms $\hat\alpha_*$ and
$\hat e_*:\pi_1(\Sigma\times S^1)\to \pi_1(M)$ are equivalent
(in fact, identical).
By Proposition \ref{prop:obstruction1}(2),
this implies that
$\hat \alpha$ and $\hat e$ are homotopic,
for $M$ is aspherical. 
Hence $[\alpha]$ is conjugate to the identity element $[e]$
in $\pi_1(C (\Sigma, M), j )$. Hence $[\alpha]=[e] \in \pi_1(C (\Sigma, M), j)$.
\end{proof}

\section{The homotopy motion groups of incompressible surfaces in Haken manifolds}
\label{sec:The homotopy motion groups of incompressible surfaces in Haken manifolds}

In this section, we consider the 
groups $\Pi (M, \Sigma)$ and $\Mdy(M, \Sigma)$ 
in the case where $\Sigma$ is a closed, orientable, incompressible surface
in a closed, orientable Haken manifold $M$. 
Let us begin with two examples 
of non-trivial elements of $\Pi (M, \Sigma)$. 
We will see soon in Theorem \ref{thm:Haken} 
that they are in fact the only elements necessary to generate 
$\Pi (M, \Sigma)$. 

\begin{example}
\label{ex:fiber-surface}
{\rm
Let $\varphi$ be an element of $\MCG(\Sigma)$. 
Consider the 3-manifold $M:=\Sigma\times\RR/(x,t)\sim (\varphi(x),t+1)$, 
which is the $\Sigma$-bundle over $S^1$ with monodromy $\varphi$.
We denote the image of $\Sigma\times \{0 \} $ in $M$
by the same symbol $\Sigma$ and call it a {\it fiber surface}.
Then we have a natural homotopy motion $\lambda = \{f_t\}$ of $\Sigma$ in $M$
defined by $f_t(x)=[x,t]$,
where  $[x,t]$ is the element of $M$ represented by $(x,t)$ 
(see Figure \ref{fig_incompressible}(i)). 
Its terminal end is equal to $\varphi^{-1}$,
because $f_1(x)=[x,1]=[\varphi^{-1}(x),0]=\varphi^{-1}(x)$.
Thus $\varphi$ belongs to $\Mdy(M,\Sigma)$.
}
\end{example}

\begin{figure}[htbp]
\centering\includegraphics[width=11cm]{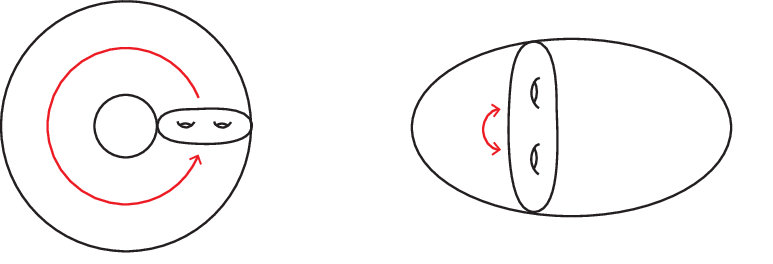}
\begin{picture}(400,0)(0,0)
\put(86,0){(i)}
\put(265,0){(ii)}
\put(147,66){$\Sigma$}
\put(50,66){\color{red}$\lambda$}
\put(228,64){\color{red}$\mu$}
\put(86,23){$M$}
\put(255,107){$\Sigma$}
\put(290,64){$N'$}
\put(213,64){$N$}
\put(267,20){$M$}
\end{picture} 
\caption{(i) The homotopy motion $\lambda$. (ii) The homotopy motion $\mu$.} 
\label{fig_incompressible}
\end{figure}
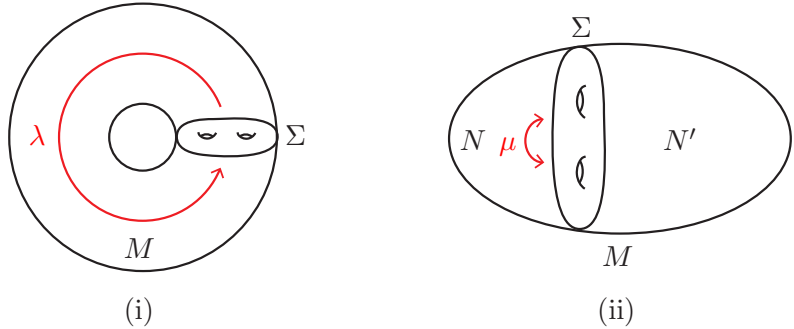

\begin{example}
\label{ex:twisted-bundle}
{\rm
Let $h$ be an orientation-reversing free involution of 
a closed, orientable surface $\Sigma$. 
Consider the 3-manifold 
$N:= \Sigma\times [0,1]/(x,t)\sim (h(x),1-t)$, which is 
the orientable twisted $I$-bundle over the closed, non-orientable surface $\Sigma/h$.
The boundary $\partial N$ is identified with $\Sigma$
by the homeomorphism $\Sigma\to\partial N$ mapping $x$ to $[x,0]$,
where $[x,t]$ denotes the element of $N$ represented by $(x,t)$.
Then we have a natural homotopy motion $\mu = \{f_t\}_{t \in I}$ of $\Sigma=\partial N$ in $N$,
defined by 
$f_t(x)= [x, t]$. 
Its terminal end is equal to $h$,
because 
$f_1(x)=[x,1]=[h(x),0]=h(x)$ 
for every $x\in\Sigma=\partial N$.
Let $N'$ be any compact, orientable $3$-manifold 
whose boundary is identified with $\Sigma$,
i.e., a homeomorphism $\partial N'\cong \Sigma$ is fixed,
and let $M=N\cup N'$ be the closed, orientable $3$-manifold
obtained by gluing $N$ and $N'$ along the common boundary $\Sigma$.
Then the homotopy motion $\mu = \{f_t\}_{t \in I}$ of $\Sigma$ in $N$ defined as above 
can be regarded as 
that of $\Sigma$ in $M$, and thus 
$h$ is an element of $\Mdy(M,\Sigma)$ 
(see Figure \ref{fig_incompressible}(ii)). 
 If $N'$ is also a twisted $I$-bundle associated with 
 an orientation-reversing involution $h'$ of $\Sigma$, 
 then we have another homotopy motion $\mu'$ of $\Sigma$ in $N'$ 
 with terminal end $h' \in \Mdy(M,\Sigma)$.
}
\end{example}

The following theorem is proved by using the positive solution
of Simon's conjecture \cite{Simon76} concerning manifold compactifications
of covering spaces, with finitely generated fundamental groups,
of compact $3$-manifolds,
which in turn is proved by using the geometrization theorem 
established by Perelman 
\cite{Perelman02, Perelman03a, Perelman03b}
and the tameness theorem of hyperbolic manifolds
established by Agol \cite{Agol04} and Calegari-Gabai \cite{CalegariGabai06}
(see also Soma \cite{Soma06} and Bowditch \cite{Bowditch10}).
A proof of Simon's conjecture can be found in Canary's expository article
\cite[Theorem 9.2]{Canary08},
where he attributes it to Long and Reid.

\begin{theorem}
\label{thm:Haken}
Let $M$ be a closed, orientable Haken manifold, 
and suppose that $\Sigma$ is 
a closed, orientable, incompressible surface in $M$.
Then the following hold.
\begin{enumerate}
\renewcommand{\labelenumi}{$(\arabic{enumi})$}
\item
If $M$ is a $\Sigma$-bundle over $S^1$ with monodromy $\varphi$ 
and $\Sigma$ is a fiber surface, then 
$\Pi (M,\Sigma)$ is the infinite cyclic group 
generated by the homotopy motion $\lambda$ 
described in Example $\ref{ex:fiber-surface}$.
\item
If $\Sigma$ separates $M$ into two submanifolds, $M_1$ and $M_2$,
precisely one of which is a twisted $I$-bundle, then 
$\Pi (M,\Sigma)$ is the order-$2$ cyclic group
generated by the homotopy motion $\mu$ 
described in Example $\ref{ex:twisted-bundle}$. 
\item
If $\Sigma$ separates $M$ into two submanifolds, $M_1$ and $M_2$,
both of which are twisted $I$-bundles, then 
$\Pi (M,\Sigma)$ is the 
infinite dihedral group
generated by the homotopy motions $\mu$ and $\mu'$ 
described in Example $\ref{ex:twisted-bundle}$.  
\item
Otherwise, $\Pi  (M,\Sigma)$ is the trivial group. 
\end{enumerate}
\end{theorem}

To show the above theorem, we require the following two lemmas.

\begin{lemma}
\label{lem:Haken1}
Let $\Sigma$ be a closed, orientable surface of genus at least $1$. 
Then 
\[\Pi(\Sigma\times\RR,\Sigma\times \{0\})=1.\]
\end{lemma}

\begin{proof}
Consider the projection $q:\Sigma\times \RR \to \Sigma\times\{0\}$.
Then for any homotopy motion $\alpha = \{f_t\}_{t \in I}$ of 
$\Sigma\times\{0\}$ in $\Sigma\times\RR$,
the composition $\{q \circ f_t\}_{t \in I}$ 
is a homotopy of maps from $\Sigma \times \{0\}$ to itself with 
initial end $\id_{\Sigma \times \{0\}}$ and terminal end $f_1$. 
It follows form Baer \cite{Baer28} 
(cf. \cite[Theorem 1.12]{FarbMargalit12})
that $f_1$ is isotopic to $\id_{\Sigma \times \{0\}}$.
Hence $\Mdy(\Sigma\times\RR,\Sigma\times \{0\})$ is trivial, and so 
$\Pi(\Sigma\times\RR,\Sigma\times \{0\})=\mathcal{K}(\Sigma\times\RR,\Sigma\times \{0\})$.

Suppose first that the genus of $\Sigma$ is at least $2$.
Then $\mathcal{K}(\Sigma\times\RR,\Sigma\times \{0\})\cong 
\pi_1 ( C (\Sigma\times \{0\} , \Sigma\times\RR) , j)$
by Lemma \ref{lem:monodromy group and kernel group}, 
and this group 
is isomorphic to $Z(j_*(\pi_1(\Sigma \times \{0\} )),\pi_1(\Sigma\times \RR))\cong Z(\pi_1(\Sigma))$
by Lemma \ref{lem:classification of pi1(C(S,M), j) for an aspherical manifold}.
Since $Z(\pi_1(\Sigma))=1$, we have 
$\mathcal{K}(\Sigma\times\RR,\Sigma\times \{0\})=1$.

Suppose next that $\Sigma$ is the torus. 
Then, by Lemmas \ref{lem:monodromy group and kernel group} and
\ref{lem:classification of pi1(C(S,M), j) for an aspherical manifold}, 
$\mathcal{K}(\Sigma\times\RR,\Sigma\times \{0\})$ 
is isomorphic to the quotient of 
the centralizer $Z(j_*(\pi_1(\Sigma \times \{0\} )),\pi_1(\Sigma\times \RR))\cong \pi_1(\Sigma)$
by $\Phi(\mathscr{I} ( \pi_1 ( J (\Sigma\times\{0\}, \Sigma\times\RR ) ) ))$.
Now, identify $\Sigma$ with $\RR^2 / \ZZ^2$,
and denote by $[x,y]$ the point of $\Sigma$ represented by $(x,y)\in\RR^2$.
For an element $(m,n)\in \ZZ^2=\pi_1(\Sigma)$, let 
$\xi_{m,n}= \{g_t\}_{t \in I}$ be the ambient isotopy of $\Sigma$,
defined by $g_t ([x,y]) = ( [x + mt,   y + nt)]) $,
and regard it as an element of 
$\pi_1(J (\Sigma \times \{ 0\} , \Sigma\times\RR))$.
Then we can easily check that $\Phi(\mathscr{I}(\xi_{m,n}))=(m,n)$.
Thus $\Phi\circ \mathscr{I}$ is surjective, 
and therefore we again have $\mathcal{K}(\Sigma\times\RR,\Sigma\times \{0\})=1$. 
Hence $\Pi(\Sigma\times\RR,\Sigma\times \{0\})=1$ as desired.
\end{proof}

\begin{lemma}
\label{lem:Haken2}
Let $\Sigma$ be a closed, orientable surface of genus at least $1$. 
For $t \in [0,1]$, let $j_t : \Sigma \to \Sigma\times \RR$ be the embedding defined by 
$j_{t} (x) = (x, t)$, and set $\kappa: = \{j_t\}_{t \in I} : I \to C(\Sigma, \Sigma \times  \RR)$. 
Then any continuous map $\alpha=\{ f_t\}_{t \in I}:
(I, \{ 1\} , \{ 0\} ) \to ( C(\Sigma, \Sigma \times \RR) , 
J (\Sigma \times \{ 1\} , \Sigma \times \RR) \circ j_1 , \{j_0\} )$ 
is homotopic to $\kappa$ 
via a homotopy through maps of the same form. 
\end{lemma}

\begin{proof}
Let $\alpha = \{ f_t\}_{t \in I}$ be a continuous map 
from $(I, \{ 1\} , \{ 0\} )$  to $( C(\Sigma, \Sigma \times \RR) , 
J (\Sigma \times \{ 1\} , \Sigma \times \RR) \circ j_1 , \{j_0\} )$. 
Let $q:\Sigma\times\RR \to \Sigma$ be the projection. 
Then, by Baer \cite{Baer28}, $q \circ \alpha(1)= q \circ f_1$ is isotopic to the identity map
as a self-homeomorphism of $\Sigma$.  
Thus deforming $\alpha$ by a homotopy through maps 
$(I, \{ 1\} , \{ 0\} ) \to ( C(\Sigma, \Sigma \times \RR) , 
J (\Sigma \times \{ 1\} , \Sigma \times \RR) \circ j_1 , \{j_0\} )$
if necessary, 
we may assume that $\alpha(1)=f_1=j_1$.

Consider the path
\[u: (I, \{1\}, \{0\}) \to (\Sigma \times \RR , \{ (x_0, 0) \}, \{ (x_0, 1) \} ), ~t \mapsto f_t (x_0). \]
We see that the closed loop $q\circ u$ represents an element of the center 
$Z(\pi_1(\Sigma))$, 
by an argument similar to that 
in the paragraph preceding Definition \ref{def:classification of pi1(C(S,M), j) for an aspherical manifold}.
(Here, we use the map $\bar\alpha:\Sigma\times S^1 \to \Sigma$
defined by $\bar\alpha(x,t)=q(f_t(x))$.) 
If $\Sigma$ has genus at least $2$, then $q\circ u$ represents the trivial element 
of $\pi_1(\Sigma)$.
If $\Sigma$ is a torus, 
using the ambient isotopy $\xi_{m,n}$ for $(m,n)\in \ZZ^2$ defined in
the proof of Lemma \ref{lem:Haken1}, 
we can deform $\alpha$ by 
a homotopy through maps 
$(I, \{ 1\} , \{ 0\} ) \to ( C(\Sigma, \Sigma \times \RR) , 
J (\Sigma \times \{ 1\} , \Sigma \times \RR) \circ j_1 , \{j_0\} )$ 
so that $q\circ u$ represents the trivial element of $\pi_1(\Sigma)$. 

Let $X=(\Sigma \times I) \times I$ and $X_0=\partial X$,
and consider the map $F:X_0 \to \Sigma\times \RR$ defined by
$F(x,t,0)=\alpha(x,t)$, $F(x,t,1)=\kappa(x,t)=(x,t)$, 
$F(x,0,s)= (x,0)$, $F(x,1,s)= (x,1)$ ($x\in \Sigma$, $t, s\in I$). 
Note that $F$ is a well-defined continuous map,
because $\alpha(1)=j_1$.
Put ${\hat{\hat x}}_0=(x_0,0,0)$, $\hat x_0=(x_0,0)$ and let 
$\theta: \pi_1(X,{\hat{\hat x}}_0) \to \pi_1(\Sigma\times \RR, \hat x_0)$
be the homomorphism induced by the map 
$X=(\Sigma \times I) \times I \to \Sigma\times \RR$
obtained as the composition of the projection
$(\Sigma \times I) \times I \to \Sigma\times I$
and the inclusion $\Sigma\times I \to \Sigma\times \RR$.
Then, since $q\circ u$ represents the trivial element of $\pi_1(\Sigma)$,
we see that $F_*:\pi_1(X,{\hat{\hat x}}_0) \to \pi_1(\Sigma\times \RR, \hat x_0)$
is equal to $\theta\circ i_*$.
Hence, by Proposition \ref{prop:obstruction2},
the map $F$ extends over $X=(\Sigma \times I) \times I$, and 
the resulting map $F:(\Sigma \times I) \times I \to \Sigma\times \RR$ 
gives the desired homotopy between $\alpha$ and $\kappa$.
\end{proof}

\begin{proof}[Proof of Theorem $\ref{thm:Haken}$]
Let $p:\tilde M\to M$ be the covering
corresponding to $\pi_1(\Sigma)<\pi_1(M)$.
Then, by the positive solution of Simon's conjecture 
(see \cite[Theorem 9.2]{Canary08}),
$\tilde M$ admits a manifold compactification,
that is, there exists a compact $3$-manifold $\hat M$ with boundary,
such that $\tilde M$ is homeomorphic to 
$\hat M - \hat{C}$,
where $\hat{C}$ is a closed subset of $\partial\hat M$. 
We actually have $\hat{C}=\partial\hat M$ and $\tilde M\cong \Int \hat M$,
because  $M$ is closed.
Brown's theorem \cite[Theorem 3.4]{Brown66} implies that $\hat M\cong \Sigma\times [-\infty,\infty]$,
where $[-\infty,\infty]$ is the closed interval 
that is obtained by compactifying $\RR=(-\infty,\infty)$.
Thus $\tilde M$ is identified with $\Sigma\times\RR$.
We assume that the restriction of the covering projection $p$ to 
$\Sigma\times \{0 \}$ is given by $p(x,0)=x$.
In other words, the inclusion map $j:\Sigma\to M$ has a lift
$\tilde j:\Sigma \to \tilde M  = \Sigma\times \RR$
such that $\tilde j(x)=(x,0)$.

Suppose that $\Pi (M,\Sigma)$ has a nontrivial element $\alpha = \{f_t\}_{t \in I}$ with $f_1=f$.
Let $\tilde{\alpha }= \{\tilde f_t\}_{t \in I}$ be the lift of $\alpha$ with $\tilde f_0=\tilde j$,
and let $\tilde f$ be the lift of $f$ defined by $\tilde f =\tilde f_1$.
Since $f(\Sigma)=\Sigma$, the image $\tilde f(\Sigma)$ is a component of
$p^{-1}(\Sigma)$ to which the restriction of $p$ is a homeomorphism onto $\Sigma$.
In particular, we have either $\tilde f(\Sigma)=\Sigma\times \{0 \}$ or
$\tilde f(\Sigma) \cap (\Sigma\times \{0 \}) =\emptyset$.  
If $\tilde f(\Sigma)=\Sigma\times \{0 \}$, the map $\tilde{\alpha}$ 
is homotopic to the constant map from $I$ to $\tilde{j} \in  C(\Sigma, \tilde{M} )$
via a homotopy through maps 
$(I, \{ 1\} , \{ 0\} ) \to ( C(\Sigma, \tilde{M} ) , J (\Sigma \times \{ 0\} , \tilde{M} ) \circ \tilde{j} , \{ \tilde{j} \} )$ 
by Lemma \ref{lem:Haken1}. 
This homotopy projects to a homotopy from $\alpha$ to the trivial homotopy motion of $\Sigma \subset M$. 
This contradicts the assumption that
$\alpha$ is a nontrivial element of $\Pi (M,\Sigma)$.
Therefore we have $\tilde f(\Sigma) \cap (\Sigma\times \{0 \}) =\emptyset$. 
Then, by \cite{Brown66}, $\tilde f(\Sigma)$ is parallel to $\Sigma\times \{0 \}$ in
$\tilde M= \Sigma\times \RR$.
We may choose the product structure so that $\tilde f(\Sigma)=\Sigma\times \{1 \}$.
Since $p$ is a covering and since $\Sigma$ is incompressible, 
we see that $p^{-1}(\Sigma)\cap (\Sigma \times (0,1))$
is a finite disjoint union of compact surfaces 
that are incompressible in $\tilde M$ and so in 
$\Sigma\times \RR$.
The result of \cite{Brown66} implies that all components 
of $p^{-1}(\Sigma)\cap (\Sigma \times (0,1))$ 
are parallel to $\Sigma\times \{0 \}$ in $\tilde M$.
Hence there exists a component that is closest to $\Sigma\times \{0 \}$.
We choose $\alpha$ so that
$\tilde f(\Sigma)=\Sigma\times \{1 \}$ is the closest component,
namely $p^{-1}(\Sigma)\cap (\Sigma \times (0,1))=\emptyset$.

Fix an orientation of the surface $\Sigma\subset M$,
and orient the surfaces 
$\Sigma\times \{t \} \subset \tilde M= \Sigma\times \RR$
($t\in\RR$) via the canonical identification with the oriented $\Sigma$. 
Consider the homeomorphism $\psi :\Sigma\times \{0 \} \to \Sigma\times \{1 \}$
defined by
$\psi =(p|_{\Sigma\times \{1\}})^{-1}\circ p|_{\Sigma\times \{0 \}}$. 
It should be noted that $\psi$ is a ``local covering transformation",
in the sense that $\psi$ extends to a homeomorphism 
between a neighborhoods of $\Sigma\times \{0 \}$ and $\Sigma\times \{1 \}$
that commutes with the covering projection $p$.
Let 
$q:\Sigma\times \RR\to \Sigma$ 
be the projection 
to the first factor, which is identified with the surface $\Sigma$ in $M$. 

\smallskip

Case 1. Suppose that $\psi$ is orientation-preserving.
Consider the $3$-manifold
$M':=(\Sigma \times [0,1])/ (x, 0) \sim (q (\psi(x)),1)$.
Then the restriction of the covering projection $p$ to $\Sigma \times [0,1]$
descends to a continuous map $p':M'\to M$,
which is a local-homeomorphism at the image of $\Sigma \times (0,1)$ in $M'$.
The condition that $\psi$ is orientation-preserving implies that
$p'$ is also a local homeomorphism at an open neighborhood of 
the image of $\Sigma \times \{0 \}$ 
(which is equal to that of $\Sigma \times \{1 \}$) in $M'$.
(Here, we use the fact that
$\psi$ is a local covering transformation.)
Thus $p':M'\to M$ is a local homeomorphism.
Since $M$ is a compact, connected manifold, 
it follows that $p'$ has the path-lifting property, and
hence $p'$ is a covering  (see e.g. Forster \cite[Theorem 4.19]{Forster91}). 
Since  $p^{-1}(\Sigma)\cap (\Sigma \times (0,1))=\emptyset$, 
the preimage of a point in $\Sigma\subset M$ by $p'$
is a singleton.
Hence $p'$ is a homeomorphism and so 
$M$ is identified with 
the $\Sigma$-bundle $(\Sigma\times\RR)/(x,t)\sim (\varphi(x), t+1)$ over $S^1$, 
where the monodromy $\varphi$ is defined by $\varphi = q \circ \psi$.

\smallskip
Case 2. Suppose that $\psi$ is orientation-reversing.
Consider the submanifold 
$M':=\Sigma \times [0,1]$ of $\tilde M$
and its image $M_1:=p(M')$ in $M$.
Since $p^{-1}(\Sigma)\cap (\Sigma \times (0,1))=\emptyset$, 
$p(\interior M')$ is disjoint from $\Sigma$.
This together with the assumption that $\psi$ is orientation-reversing
implies that $M_1$ is a submanifold of $M$ with boundary $\Sigma$.
Moreover, as in Case 1, we can see that the restriction 
$p':M'\to M_1$ of $p$ to $M'$ is a covering and that it has geometric degree $2$.
(Note that the preimage of a point $x\in \Sigma=\partial M_1$ by $p'$ 
consists of the two points $(x,0)$ and $(\psi(x), 1)$ of $M'=\Sigma \times [0,1]$.)
By \cite[Theorem 10.3]{Hempel76},
this implies that $M_1$ is a twisted $I$-bundle,
$\Sigma\times [0,1]/(x,t)\sim (h(x),1-t)$,
associated with some orientation-reversing free involution
$h$ of $\Sigma$. 

\smallskip
The above arguments show that if $\Pi (M,\Sigma)$ is nontrivial,
then either (i) $M$ is a $\Sigma$-bundle over $S^1$
or (ii) $\Sigma$ separates $M$ into two submanifolds,
at least one of which is a twisted $I$-bundle. 
In particular, we obtain the assertion (4) of the theorem.

Suppose that the conclusion (i) holds, 
namely $M \cong (\Sigma\times\RR)/(x,t)\sim (\varphi(x), t+1)$
for some $\varphi\in \MCG_+(\Sigma)$. 
Consider the map 
$\zeta:\Pi (M, \Sigma) \to \ZZ$ that sends the homotopy motion
$\alpha=\{f_t\}_{t\in I}$ to $n\in\ZZ$ given by 
$\tilde{f}_1 (\Sigma) = \Sigma \times \{ n \}$, 
where $\tilde{\alpha}= \{\tilde f_t\}_{t \in I}$ 
is the lift of $\alpha$ with $\tilde f_0=\tilde j$. 
Then $\zeta$ is injective,
because if two homotopy motions $\alpha$ and $\alpha'$ are mapped to the same element $n\in\ZZ$,
then the homotopy between $\tilde{\alpha}$ and $\tilde{\alpha}'$, 
given by Lemmas \ref{lem:Haken1} and \ref{lem:Haken2}
according to whether $n=0$ or not,
projects to a homotopy which gives the equivalence of $\alpha$ and $\alpha'$
as elements of $\Pi(M,\Sigma)$. 
It is obvious that $\zeta$ is a group homomorphism
and maps $\lambda$ to $1$, where $\lambda$ is the homotopy motion 
described in Example \ref{ex:fiber-surface}.
Hence $\Pi (M, \Sigma)$ is the infinite cyclic group generated by $\lambda$,
proving the assertion (1).

Suppose that the conclusion (ii) holds,
namely $M=M_1\cup_{\Sigma} M_2$
and $M_1=\Sigma\times [0,1]/(x,t)\sim (h(x),1-t)$, where $h$ is an orientation-reversing involution of $\Sigma$. 
By the preceding argument, we may assume that $\tilde M=\Sigma\times \RR$
and the restriction of the covering projection $p$ to $\Sigma\times [0,1]$
is the double covering of $M_1$ that maps $(x,t)$ to the point 
$[x,t]\in M_1$ it represents.
Note that the homotopy motion $\mu=\{f_t\}_{t\in I}$ defined in Example \ref{ex:twisted-bundle}
lifts to the map $\tilde\mu=\{\tilde f_t\}_{t\in I}:\Sigma\to \Sigma\times\RR$ 
given by $\tilde f_t(x)=(x,t)$.
Moreover, Lemma \ref{lem:Haken2} implies
if a homotopy motion $\alpha=\{f'_t\}$ has a lift  
$\tilde\alpha=\{\tilde f_t'\}_{t\in I}:\Sigma\to \Sigma\times\RR$
such that $\tilde f_0'(x)=(x,0)$ and $\image(\tilde f_1')=\Sigma\times \{1\} $,
then it is equivalent to $\mu$.
We can easily see from the definition of the concatenation that 
$\mu \cdot \mu$ is equivalent to the 
identity motion,
thus, the order of $\mu$ in $\Pi (M, \Sigma)$ is $2$. 

Suppose that there exists a nontrivial element $\beta = \{g_t\}_{t \in I}$ 
of $\Pi(\Sigma, M)$ that is not equivalent to $\alpha$.
Then the previous arguments imply that $\tilde{g}_1(\Sigma)$ is 
equal to neither $\Sigma \times \{0 \}$ nor $\Sigma \times \{1 \}$,
and so $\tilde{g}_1(\Sigma)\cap (\Sigma\times [0,1]) =\emptyset$.
(Here, $\tilde{\beta}= \{\tilde g_t\}_{t \in I}$ is the lift of $\beta$ with $\tilde g_0=\tilde j$.)
By choosing $\beta = \{ g_t \}_{t \in I}$ suitably, we may assume that $\tilde g_1(\Sigma)$ is 
the lift of $\Sigma$ closest 
to $\Sigma \times \{0 \}$ in $\Sigma\times (-\infty,0)$. 
Then the argument in Case 2 implies that 
$M_2$ is a twisted $I$-bundle
and that the terminal end $g_1$ is (represented by) the involution 
$h'$ corresponding to 
the twisted $I$-bundle structure of $M_2$.
This, in particular, proves the assertion (2).
In order to prove the assertion (3), observe that
$p:\tilde M \to M$ is a regular covering
and that 
the covering transformation group is the infinite dihedral group 
generated by the two involutions $\gamma$ and $\gamma'$ of 
$\tilde M = \Sigma\times \RR$ defined by
\[
\gamma(x,t)=(h(x),1-t), \quad \gamma'(x,t)=(h'(x),-1-t).
\]
In this case, $p^{-1}(\Sigma)=\Sigma\times\ZZ$,
and the argument for the case (i) implies that
$\Pi( \Sigma , M)$ is the infinite dihedral group 
generated by the two elements $\mu$ and $\mu'$ of order $2$.
This completes the proof of the assertion (3).
\end{proof}

\begin{corollary}
\label{cor:Haken}
Let $M$ be a closed, orientable Haken manifold $M$,
and suppose that $\Sigma$ is 
an orientable incompressible surface in $M$.
Then the following hold.
\begin{enumerate}
\renewcommand{\labelenumi}{$(\arabic{enumi})$}
\item
If $M$ is a $\Sigma$-bundle over $S^1$ with monodromy $\varphi$ 
and $\Sigma$ is a fiber surface, then 
$\Mdy(M,\Sigma)$ is the cyclic group $\langle \varphi\rangle$,
and $\mathcal{K}(M,\Sigma)$ is the $($possibly trivial$)$
subgroup generated by $\lambda^n$ of the infinite cyclic group
$\Pi(M,\Sigma)=\langle \lambda\rangle$, 
where $n$ is the order of $\varphi$.
Moreover, the homomorphism $\deg : \mathcal{K}(M,\Sigma)=\langle \lambda^n\rangle \to \ZZ$
is given by 
$\deg ( \lambda^n) =n$ under a suitable orientation convention.  
\item
If $\Sigma$ separates $M$ into two submanifolds, $M_1$ and $M_2$,
precisely one of which is a twisted $I$-bundle, then 
$\Mdy(M,\Sigma)$ is the order-$2$ cyclic group generated by
the orientation-reversing involution of $\Sigma$
associated with the twisted $I$-bundle structure,
and $\mathcal{K}(M,\Sigma)$ is the trivial group. 
\item
If $\Sigma$ separates $M$ into two submanifolds, $M_1$ and $M_2$,
both of which are twisted $I$-bundles, then 
$\Mdy(M,\Sigma)$ is the $($finite or infinite, and possibly cyclic$)$ dihedral group generated by
the two orientation-reversing involutions
$h_1$ and $h_2$ of $\Sigma$
associated with the twisted $I$-bundle structures, and 
$\mathcal{K}(M,\Sigma)$ is the subgroup of 
the inifinite dihedral group 
$\Pi(M,\Sigma)=\langle \mu, \mu' \ | \
\mu^2, \mu'^2 \rangle$
generated by $(\mu\mu')^n$,
where $n$ is the order of $hh'$.
Moreover, the homomorphism $\deg : \mathcal{K}(M,\Sigma)=\langle (\mu\mu')^n\rangle \to \ZZ$
is given by 
$\deg ((\mu\mu')^n)=2n$
under a suitable orientation convention. 
\item
Otherwise, 
both $\Mdy(M,\Sigma)$ and $\mathcal{K}(M,\Sigma)$ are the trivial group.
\end{enumerate}
\end{corollary}

\begin{proof}
The assertions except for those concerning the homomorphism $\deg : \mathcal{K}(M,\Sigma) \to \ZZ$
follow immediately from Theorem \ref{thm:Haken}
and the exact sequence (\ref{exact sequence for homotopy motion group}). 
It is also easy to see that 
$\deg (\lambda^n) =n$.
The identity 
$\deg ((\mu\mu')^n) =  2n$
can be verified by considering the double covering of $M$,
which is the $\Sigma$-bundle over $S^1$ with monodromy $h_1h_2$. 
\end{proof}

\section{Homotopy motion groups of homotopically trivial surfaces}
\label{sec:homotopically trivial case}

In this section, 
we study the case 
contrastive to that treated in the previous section.
We say that a closed, orientable surface $\Sigma$ embedded in a closed, orientable $3$-manifold $M$ 
is {\it homotopically trivial} if 
the inclusion map $j:\Sigma\to M$ is homotopic to a constant map. 

\begin{lemma}
\label{lem:homotopically trivial}
A closed, orientable surface $\Sigma$ embedded in a closed, orientable, irreducible $3$-manifold $M$
is homotopically trivial if and only if $j_*:\pi_1(\Sigma)\to \pi_1(M)$ is the trivial 
homomorphism.
\end{lemma}

\begin{proof}
The ``only if" part is obvious. 
The ``if" part follows from
Proposition \ref{prop:obstruction1}(2). 
\end{proof}

We have the following theorem.

\begin{theorem}
\label{thm:homotopy motion group for a local surface}
Let $\Sigma$ be a closed, orientable surface embedded in 
a closed, orientable $3$-manifold $M$. 
Then the following hold.
\begin{enumerate}
\renewcommand{\labelenumi}{$(\arabic{enumi})$}
\item
If $\Sigma$ is homotopically trivial and if
$M$ is aspherical, then
$\Pi (M, \Sigma) \cong \pi_1(M) \times \MCG(\Sigma)$. 
To be more precise, 
$\Mdy(M,\Sigma)=\MCG(\Sigma)$, and 
$\mathcal{K}(M,\Sigma)$ is identified with the factor $\pi_1(M)$.
Moreover, the homomorphism $\deg : \mathcal{K}(M,\Sigma)\to \ZZ$ vanishes. 
\item
Conversely, if $\Mdy(M,\Sigma)  =  \MCG(\Sigma)$ and if $M$ is irreducible,
then $\Sigma$ is homotopically trivial.
\end{enumerate}
\end{theorem}

\begin{proof}
(1) Suppose that $\Sigma$ is homotopically trivial and
$M$ is aspherical.
Pick a base point $x_0\in \Sigma\subset M$, and define 
a homomorphism 
$\Psi : \Pi (M, \Sigma) \to \pi_1(M, x_0) $ as follows.
For an element of $\Pi (M, \Sigma)$,
choose a representative homotopy motion $\alpha$ such that $\alpha (1) (x_0) = x_0$.
Then the element of $\pi_1(M, x_0)$ represented by the closed path
\[ (I, \partial I) \to (M, x_0), ~t  \mapsto \alpha(t) (x_0) \]
does not depend on the choice of a representative $\alpha$,
by the following reason.
Two such closed paths are related, up to homotopy relative to $\partial I$,
by concatenation of a closed path on $\Sigma$ based at $x_0$.
However, since $\Sigma$ is homotopically trivial in $M$, 
any closed path on $\Sigma$ is null-homotopic in $M$.
Thus two such closed paths represent the same element of $\pi_1(M, x_0)$.
We define $\Psi ([\alpha]) \in \pi_1 (M, x_0)$
to be that element.

Suppose first that the genus of $\Sigma$ is at least $2$. 
By Lemma \ref{lem:monodromy group and kernel group}, 
$\mathcal{K} (M, \Sigma)$ can be canonically identified with $\pi_1 (C (\Sigma, M), j)$, and 
the restriction of $\Psi$ to $\pi_1 (C (\Sigma, M), j)$ is nothing but the 
homomorphism $\Phi$ 
in Definition \ref{def:classification of pi1(C(S,M), j) for an aspherical manifold}.
Therefore, we have the following commutative diagram: 
\[
  \xymatrix{
 1 \ar[r] & \pi_1 (C (\Sigma, M), j) \ar[r] \ar[d]_{\Phi}  & \Pi (M, \Sigma) \ar[r]^{\partial_+} \ar[d]_{\Psi \times \partial_+} & \Mdy (M, \Sigma) \ar[r] \ar[d]_{\iota} & 1 \\
 1 \ar[r] & \pi_1 (M, x_0) \ar[r] & \pi_1 (M, x_0) \times \MCG(\Sigma) \ar[r] & \MCG(\Sigma) \ar[r] & 1,
}
\]
where the two rows are exact, and $\iota: \Mdy (M, \Sigma) \to \MCG (\Sigma)$ is the inclusion 
homomorphism. 
Lemmas \ref{lem:monodromy-fundgp} and
\ref{lem:classification of pi1(C(S,M), j) for an aspherical manifold}, respectively,
imply that 
$\iota$ and $\Phi$ are isomorphisms,
so does $\Psi \times \partial_+$.

Suppose that the genus of $\Sigma$ is less than $2$. 
In that case, by replacing $\pi_1 ( C (\Sigma , M) , j)$ with 
$\pi_1 ( C (\Sigma , M) , j) / \mathscr{I} ( \pi_1 ( J (\Sigma, M ) , j) )$, 
the same argument as above still works because 
the homomorphism $\Phi$ vanishes on $\mathscr{I} ( \pi_1 ( J (\Sigma, M ) , j) )$ 
due to the assumption that $\Sigma$ is homotopically trivial. 

The vanishing of $\deg : \mathcal{K}(M,\Sigma)\to \ZZ$ can be seen as follows.
Suppose that $\deg (\alpha) \ne 0$ for some $\alpha \in \mathcal{K}(M,\Sigma)$.
Then the image of $\hat\alpha_*:\pi_1(\Sigma\times S^1)\to \pi_1(M)$ has finite index in $\pi_1(M)$
(cf. \cite[Lemma 15.12]{Hempel76}).
Since $M$ is an aspherical, closed, orientable $3$-manifold, 
this implies that the cohomological dimension of $\image(\hat\alpha_*)$ is $3$.
On the other hand, since $\Sigma$ is homotopically trivial, $\image(\hat\alpha_*)$ is cyclic.
This is a contradiction, because the cohomological dimension of a cyclic group is
$0$, $\infty$, or $1$,
according as it is trivial, nontrivial finite cyclic, or infinite cyclic.  
This completes the proof of (1). 

(2)
Note that the assertion is trivial when $\Sigma = S^2$. 
We assume that the genus of $\Sigma$ is at least $1$, and we show 
the assertion by induction on the genus $g$ of $\Sigma$. 
Suppose that $g=1$ and $\Mdy (M, \Sigma) = \MCG (\Sigma)$. 
By Corollary \ref{cor:Haken}, 
$\Sigma$ cannot be incompressible in $M$. 
Thus, there exists an essential simple closed curve on $\Sigma$ that is null-homotopic in $M$. 
Since $\MCG(\Sigma)$ acts on the set $\Curve(\Sigma)$ of (isotopy classes of) essential simple closed curves on $\Sigma$ transitively, 
every simple closed curve on $\Sigma$ is null-homotopic in $M$. 
Thus, the homomorphism $j_* : \pi_1 (\Sigma) \to \pi_1 (M)$ vanishes, which implies 
by Lemma \ref{lem:homotopically trivial}
that $\Sigma$ is homotopically trivial in $M$, as $M$ is irreducible. 
For the inductive step, suppose that the assertion holds for any surface $\Sigma$ with genus at most $g$. 
Let $\Sigma$ be a closed, orientable surface of genus $g+1$ embedded in $M$ so that 
$\Mdy (M, \Sigma) = \MCG(\Sigma)$. 
Again, by Corollary \ref{cor:Haken}, $\Sigma$ is compressible. 
Let $D$ be a compression disk for $\Sigma$. 
If $\partial D$ is non-separating in $\Sigma$, the proof runs as in the case of $g=1$, for 
$\MCG(\Sigma)$ acts on the set of (isotopy classes of) non-separating simple closed curves on $\Sigma$ transitively, and $\pi_1 (\Sigma)$ is generated by 
elements represented by non-separating simple closed curves. 
Suppose that $\partial D$ is separating. 
Let $\Sigma_1$ and $\Sigma_2$ be the connected surfaces 
obtained by compressing $\Sigma$ along $D$. 
Then we can see that $\Mdy (M, \Sigma_i) = \MCG(\Sigma_i)$ ($i=1,2$) as follows.
Let $\Sigma_1'$ and $\Sigma_2'$ ($i=1,2$) be the closures of components of $\Sigma - \partial D$. 
We regard $\Sigma_i$ as $\Sigma_i' \cup D$ ($i=1,2$), so $\Sigma_1 \cap \Sigma_2 = D$.  
Let $f_1$ be an arbitrary element of $\MCG(\Sigma_1)$. 
We show that $f_1 \in \Mdy (M, \Sigma_1)$. 
We can assume that $f_1(D) = D$. 
Then there exists an element $f_2 \in \MCG (\Sigma_2)$ such that 
$f_2 (D) = D$ and $f_1|_D = f_2 |_D$. 
Let $\bar f:\Sigma\cup D \to M$ be the map obtained by gluing $f_1$ and $f_2$,
and let $f$ be the restriction of $\bar f$ to $\Sigma$.
By the assumption, there exists a homotopy motion $\alpha:\Sigma\times I \to M$
with terminal end $f$.
Then we can extend $\alpha$ to a homotopy motion 
$\bar\alpha:(\Sigma\cup D)\times I \to M$
with terminal end $\bar f$
because $M$ is irreducible and hence $\pi_2(M)=0$.
By restriction, $\bar\alpha$ determines a homotopy motion of 
$\Sigma_1(\subset \Sigma\cup D)$ with terminal end $f_1$, which implies that 
$\Mdy (M, \Sigma_1) = \MCG(\Sigma_1)$. 
Clearly, the same consequence holds for $\Sigma_2$.
By the assumption of induction, both $\Sigma_1$ and $\Sigma_2$ are homotopically trivial in $M$.
Hence the image of $\pi_1(\Sigma\cup D))\cong \pi_1(\Sigma_1)*\pi_1(\Sigma_2)$ in $\pi_1(M)$
is trivial.
Thus $j_*:\pi_1(\Sigma)\to\pi_1(M)$ is the trivial homomorphism, and so $\Sigma$ is homotopically trivial in $M$,
by Lemma \ref{lem:homotopically trivial}. 
\end{proof}

\begin{remark}
{\rm
The assumption that $M$ is aspherical in 
Theorem \ref{thm:homotopy motion group for a local surface}(1) is essential. 
In fact, in 
Theorem \ref{thm:non-zero degree maps 1}(2) 
we will see that 
when $\Sigma$ is a Heegaard surface of $S^3$, which is homotopically trivial, 
the kernel of the homomorphism 
$\Psi \times \partial_+ : \Pi (S^3, \Sigma) \to  
\pi_1 (S^3) \times \MCG (\Sigma) = \MCG (\Sigma)$ 
defined in the above proof consists of infinitely many elements.
}
\end{remark}

\begin{corollary}
\label{cor:homotopy motion group for a local surface}
Let $\Sigma$ be a closed, orientable surface embedded in 
a closed, orientable, irreducible 
$3$-manifold $M$. 
Then $\Sigma$ is homotopically trivial and if and only if $\Mdy(M,\Sigma)  =  \MCG(\Sigma)$.
\end{corollary}

\begin{proof}
The ``if" part is nothing other than Theorem \ref{thm:homotopy motion group for a local surface}(2), 
and the ``only if" part follows from Lemma \ref{lem:monodromy-fundgp}.
\end{proof}

\section{The group $\Mdy(M,\Sigma)$ for a Heegaard surface and its friends}
\label{sec:The homotopy motion groups of Heegaard surfaces and their friends}

From this section, we are going to 
study the homotopy motion group $\Pi (M, \Sigma)$ and 
related groups 
$\Gamma (M, \Sigma)$ and $\mathcal{K}(M,\Sigma)$ 
for a Heegaard surface $\Sigma$ of a closed, orientable $3$-manifold 
$M$.  
Recall that a closed surface $\Sigma$ in a closed orientable $3$-manifold $M$ 
is called a {\it Heegaard surface} if $\Sigma$ 
separates $M$ into two handlebodies $V_1$ and $V_2$. 
Such a decomposition $M=V_1\cup_{\Sigma}V_2$ is then called a {\it Heegaard splitting} of 
$M$, and the {\it genus} of the splitting is defined to be the genus of $\Sigma$. 

In this section, we mainly consider the group $\Mdy (M, \Sigma)$ rather than 
$\Pi (M, \Sigma)$.
We recall various natural subgroups of $\MCG(\Sigma)$
associated with a Heegaard surface, and describe their relationships with
the group $\Mdy(M,\Sigma)$. 
We also describe the group $\Mdy(M,\Sigma)$ and give answers to 
Questions \ref{question:Refined:Homotopy class of curves in 3-manifold} and \ref{question:motivation}
for the very special cases where $\Sigma$ is either an arbitrary Heegaard surface of $M=S^3$ 
and where $\Sigma$ is a minimal genus Heegaard surface of $M= \#_g (S^2 \times S^1)$.

By definition, the group $\Mdy (M, \Sigma)$ is a subgroup of 
the extended mapping class group $\MCG(\Sigma)$ of the Heegaard surface $\Sigma$. 
For a Heegaard splitting $M=V_1\cup_{\Sigma}V_2$, many other (and similar) subgroups 
of $\MCG (\Sigma)$ associated with the Heegaard splitting $M=V_1\cup_{\Sigma}V_2$ has 
been studied as follows.  

\begin{enumerate}
\item
The {\it handlebody group} $\MCG(V_i)$ of the handlebody $V_i$,
which is identified with a subgroups of $\MCG(\Sigma)$,
by restricting a self-homeomorphism of $V_i$ to its boundary $\partial V_i=\Sigma$.
This has been a target of various works 
(see a survey by Hensel \cite{Hensel} and references therein).
\item
The intersection $\MCG(V_1)\cap \MCG(V_2)$,
which is identified with $\MCG(M,V_1,V_2)$.
This group or 
its orientation-subgroup $\MCG^+(M,V_1,V_2)$ is called the {\it Goeritz group}
of the Heegaard splitting $M=V_1\cup_{\Sigma}V_2$
and it has been extensively studied. 
In particular, the problem of when this group is 
finite, finitely generated, or finitely presented
attracts attention of various researchers (cf. Minsky \cite[Question 5.1]{Gordon07}). 
The work on this problem goes back to Goeritz \cite{Goeritz33}, 
which gave a finite generating set of the Goeritz group of 
the genus-$2$ Heegaard splitting of $S^3$. 
In these two decades, great progress was achieved by many authors 
\cite{Scharlemann04, Namazi07, Akbas08, Cho08, Johnson10, Johnson11a, Cho13, ChoKoda14, ChoKoda15, ChoKoda16, FreedmanScharlemann18, ChoKoda19, IguchiKoda20}, however, 
it still remains open whether the Goeritz group 
a Heegaard splitting of $S^3$ is finitely generated when the genus is at least $4$.  

\item
The group 
$\langle \MCG(V_1), \MCG(V_2)\rangle$
generated by $\MCG(V_1)$ and $\MCG(V_2)$.
Minsky \cite[Question 5.2]{Gordon07} asked when this subgroup is the free product with 
amalgamated subgroup
$\MCG(V_1)\cap\MCG(V_2)$.
A partial answer to this question was given 
by Bestvina-Fujiwara \cite{BestvinaFujiwara17}.
\item
The mapping class group $\MCG(M,\Sigma)$ of the pair $(M,\Sigma)$.
This contains $\MCG(M,V_1,V_2)$ as a subgroup of index $1$ or $2$.
The result of Scharlemann-Tomova \cite{ScharlemannTomova06} says that
the natural homomorphism from $\MCG(M,\Sigma)$ to $\MCG(M)$
is surjective if the Hempel distance $d(\Sigma)$ (see \cite{Hempel01}) is greater than $2g(\Sigma)$.
On the other hand, it is proved by Johnson \cite{Johnson10},
improving the result of Namazi \cite{Namazi07},
that the natural homomorphism $\MCG(M,\Sigma) \to \MCG(M)$
is injective if the Hempel distance $d(\Sigma)$ is greater than $3$.
Hence, the natural homomorphism gives an isomorphism $\MCG(M,\Sigma)\cong\MCG(M)$
if $g(\Sigma)\ge 2$ and $d(\Sigma)>2g(\Sigma)$.
Building on the work of McCullough-Miller-Zimmermann \cite{McCulloughMillerZimmermann89}
on finite group actions on handlebodies,
finite group actions on the pair $(M,\Sigma)$ are extensively studied 
(see Zimmermann \cite{Zimmermann92, Zimmermann20}
and references therein).
\item
The subgroup $G (M, \Sigma) :=\ker(\MCG(M,\Sigma) \to \MCG(M))$,
which forms a subgroup of the group $\Mdy(M,\Sigma)$.  
We can write this group as 
\[ G(M,\Sigma) =
\{ [f] \in \MCG (\Sigma) \mid 
\mbox{$j\circ f$ is ambient isotopic to $j$.} \}, \]
where $j : \Sigma \to M$ is the inclusion map, 
and thus we can think of 
$\Mdy(M,\Sigma)$ as a ``homotopy version" of $G(M,\Sigma)$.
Johnson-Rubinstein \cite{JohnsonRubinstein13} gave systematic constructions
of periodic, reducible, pseudo-Anosov elements in this group.
Johnson-McCullough \cite{JohnsonMcCullough13}
called this group the Goeritz group instead of the one  we described in (2), 
and they used this group to study the 
homotopy type of the space of Heegaard surfaces. 
In particular,
they prove that if $\Sigma$ is a Heegaard surface 
of a closed, orientable, aspherical $3$-manifold $M$,
then, except the case where $M$ is a non-Haken infranilmanifold,
the exact sequence (\ref{JohnsonMcCullough exact sequence for homotopy motion group})
in the introduction is refined to the following exact sequence 
\[
1 \to Z(\pi_1(M))\to \Motion(M,\Sigma) \to G(M,\Sigma)\to 1,
\]
where $\Motion(M,\Sigma)$ is the 
smooth motion group of $\Sigma$ in $M$
\cite[Corollary 1]{JohnsonMcCullough13}.

\item
The group $\Mdy(V_i):=\ker(\MCG(V_i)\to \Out(\pi_1(V_i)))$.
As noted in the introduction, 
the group $\Mdy(V_i)$ is identified with the group
$\Mdy(V_i,\Sigma)<\MCG(\Sigma)$. 
It was shown by Luft \cite{Luft78} that 
its index-$2$ subgroup $\Mdy^+(V_i):=\ker(\MCG^+(V_i)\to \Out(\pi_1(V_i)))$ 
is the {\it twist group}, that is, the subgroup of 
$\MCG^+(V_i)$ generated by the Dehn twists about meridian disks. 
McCullough \cite{McCullough85} proved that $\Mdy(V_i)$ is not finitely generated
by showing that it admits a surjection onto a free abelian group of infinite rank. 
A typical orientation-reversing element of $\Mdy(V_i)$ ($ < \MCG(\Sigma)$)
is the restriction to $\Sigma=\partial V_i$ of
a  vertical $I$-bundle involution of $V_i$. 

\item
The group $\langle \Mdy(V_1), \Mdy(V_2)\rangle$ 
generated by $\Mdy(V_1)$ and $\Mdy(V_2)$,
which is contained in $\Mdy(M,\Sigma)$. 
It was proved by 
Bowditch-Ohshika-Sakuma in \cite[Theorem B]{OhshikaSakuma16}
(see also Bestvina-Fujiwara \cite[Section 3]{BestvinaFujiwara17})
that its orientation-preserving subgroup
$\langle \Mdy^+(V_1), \Mdy^+(V_2)\rangle$ is 
the free product $\Mdy^+(V_1)\ast\Mdy^+(V_2)$
if the Hempel distance $d(\Sigma)$ is high enough.
(The question of whether the same conclusion holds
for $\langle \Mdy(V_1), \Mdy(V_2)\rangle$
is still an open question.)
\end{enumerate}

In summary, the 
subgroups of $\MCG(\Sigma)$ introduced above are related as follows: 
\begin{align*}
&G (M, \Sigma) < \Mdy (M, \Sigma) \cap \MCG (M, \Sigma) < \Mdy (M, \Sigma), 
\\
&  \langle \Mdy (V_1), \Mdy(V_2)\rangle < \Mdy (M, \Sigma) \cap \langle \MCG(V_1), \MCG(V_2) \rangle < 
\Mdy (M, \Sigma). 
\end{align*}

As noted in the introduction, our interest in $\Mdy (M, \Sigma)$ was motivated by
Minsky's Question \ref{question:Minsky} and its refinement 
Question \ref{question:Refined:Homotopy class of curves in 3-manifold},
and our main concern is 
Question \ref{question:motivation}(1) 
about the relationship between $\Mdy (M, \Sigma)$ and its subgroup 
$\langle \Mdy (V_1), \Mdy(V_2)\rangle$. 
We end this section by giving an answer to 
Questions \ref{question:Refined:Homotopy class of curves in 3-manifold} and \ref{question:motivation}
in two very special cases.

\begin{example}
\label{example:Minsky's question for S3}
{\rm
Let $S^3 = V_1 \cup_\Sigma V_2$ be the genus-$g$ Heegaard splitting of $S^3$.   
Recall that the (orientation-preserving) mapping class group 
$\MCG^+(\Sigma)$
is generated by the Dehn twists 
about certain $3 g -1 $ simple closed curves on $\Sigma$ by Lickorish \cite{Lickorish64}, 
where $g$ is the genus of $\Sigma$. 
Since we can find those simple closed curves in $\Delta$, we have 
$\langle \Mdy^+(V_1), \Mdy^+(V_2)\rangle  = \MCG^+ (\Sigma)$. 
It is thus easy to see that 
\[ \langle \Mdy (V_1), \Mdy(V_2)\rangle = \Mdy (S^3, \Sigma) = \MCG(\Sigma) \]
 and 
\[
\langle \Mdy (V_1), \Mdy(V_2)\rangle \Delta = \Mdy (S^3, \Sigma) \Delta 
=  \Curve(\Sigma) = Z.
\]
}
\end{example}

We note that the group $\Gamma(M, \Sigma)$ detects the $3$-sphere 
as in the following meaning. 

\begin{proposition}
\label{prop:monodromy group detects S3}
Let $M = V_1 \cup_\Sigma V_2$ be a Heegaard splitting of a closed, orientable $3$-manifold. 
Then we have $\Mdy (M, \Sigma) = \MCG (\Sigma )$ if and only if 
$M = S^3$. 
\end{proposition}
\begin{proof}
This is straightforward from  
Corollary \ref{cor:homotopy motion group for a local surface}
and 
the Poincar\'{e} conjecture proved by Perelman \cite{Perelman02, Perelman03a, Perelman03b}. 
\end{proof}

\begin{example}
\label{example:Minsky's question for S2 times S1}
{\rm
Let $M= \#_g (S^2 \times S^1)$, and  $M= V_1\cup_{\Sigma}V_2$ the genus-$g$ Heegaard splitting. 
In this case, $M$ is the double of the handlebody $V_1$, and thus $\Delta = \Delta_i = Z$ $(i=1,2)$. 
Further,  we can check easily that  
\[ \Mdy (V_i)  =  \langle \Mdy (V_1), \Mdy(V_2)\rangle = \Mdy (M, \Sigma) \]
and 
\[ 
\langle \Mdy (V_1), \Mdy(V_2)\rangle \Delta = \Mdy (M, \Sigma) \Delta =  Z . 
\]
}
\end{example}

In the above easy examples, the group $\langle \Mdy (V_1), \Mdy(V_2)\rangle$ 
coincides with the whole group $\Mdy (M, \Sigma)$ in an obvious way. 
However, this is not the case in general,
as indicated in the introduction and proved in 
Theorem \ref{thm:well-definedness of homological degree for aspherical manifolds}.

\section{Open book rotations}
\label{sec:Open book rotations}

In this section, 
we first recall the definition of an open book decomposition
and the Heegaard splitting obtained from an open book decomposition. 
We then introduce two homotopy motions of the Heegaard surface,
the ``half book rotation'' $\rho$ and the ``unilateral book rotation'' $\sigma$,
which play key roles in 
the subsequent two sections.

Let $M$ be a closed, orientable $3$-manifold. 
Recall that an \textit{open book decomposition} of $M$ is defined to be the pair $(L, \pi)$, where 
\begin{enumerate}
\item
$L$ is a (fibered) link in $M$; and 
\item
$\pi : M - L \to S^1$ is a fibration such that $\pi^{-1} (\theta)$ is the interior of 
a Seifert surface $\Sigma_\theta$ of $L$ for each $\theta \in S^1$. 
\end{enumerate} 
We call $L$ the \textit{binding} and $\Sigma_\theta$ a \textit{page} of the open book decomposition $(L, \pi)$. 
The monodromy of the fibration $\pi$ is called the {\it monodromy} of $(L, \pi)$. 
We think of the monodromy $\varphi$ of $(L, \pi)$ as an element of
$\MCG(\Sigma_0, \mathrm{rel}\ \partial\Sigma_0)$, 
the mapping class group of $\Sigma_0$ relative to $\partial \Sigma_0$,
i.e., the group of self-homeomorphisms of $\Sigma_0$ 
that fix $\partial\Sigma_0$, 
modulo isotopy fixing $\partial\Sigma_0$. 
The pair $(M, L)$, as well as the projection $\pi$, is 
then recovered from $\Sigma_0$ and $\varphi$. 
Indeed, we can identify $(M, L)$ with 
\[ (\Sigma_0 \times \RR , \partial \Sigma_0 \times \RR ) / \sim ,\]
where $\sim$ is defined by $(x, s) \sim (\varphi (x) , s+1) $ for $x \in \Sigma_0$ and 
$s \in \RR$, and  
$(y, 0) \sim (y , s) $ for $y \in \partial \Sigma_0$ and any $s \in \RR$.
So, we occasionally denote the open book decomposition $(L, \pi)$ 
by $(\Sigma_0,\varphi)$. 
Under this identification,  
the Seifert surface $\Sigma_{\theta}$ is identified with the image $\Sigma\times\{\theta\}$.
We define an $\RR$-action $\{ r_t \}_{t \in \RR}$ on $M$, 
called a {\it book rotation},  by $r_t ([ x, s ]) = [x, s + t]$, 
where $[x,s]$ denotes the element of $M$ represented by $(x,s)$. 

Given an open book decomposition $(L, \pi)$ of $M$, 
we obtain a Heegaard splitting $M = V_1 \cup_\Sigma V_2$, where    
\begin{align*}
V_1 &= \cl(\pi^{-1} ([0, 1/2])) =\pi^{-1} ([0, 1/2]) \cup L 
=\cup_{0\le \theta\le 1/2} \Sigma_{\theta},\\
V_2 &= \cl(\pi^{-1} ([1/2, 1])) =\pi^{-1} ([1/2, 1])\cup L
=\cup_{1/2\le \theta\le 1} \Sigma_{\theta},\\
\Sigma &= \Sigma_{0} \cup \Sigma_{1/2}.
\end{align*}
We call this the Heegaard splitting of $M$ \textit{induced from} the open book decomposition $(L, \pi)$.
For the resulting Heegaard surface $\Sigma$, 
we define two particular 
homotopy motions in $M$. 
The first one, $\rho = \rho_{(L, \pi)}=\rho_{(\Sigma_0,\varphi)}$, is defined by restricting 
the book rotation, with time parameter rescaled by the factor $1/2$,
to the Heegaard surface $\Sigma$, namely 
$\rho(t) =  r_{t/2}|_{\Sigma}$, 
see Figure \ref{fig_rho}.
\begin{figure}[htbp]
\centering\includegraphics[width=14cm]{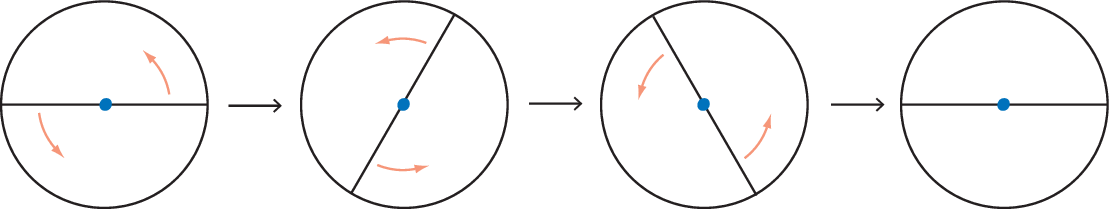}
\begin{picture}(400,0)(0,0)
\put(35,38){\begin{small}\color{blue}$L$\end{small}}
\put(148,44){\begin{small}\color{blue}$L$\end{small}}
\put(256,50){\begin{small}\color{blue}$L$\end{small}}
\put(357,56){\begin{small}\color{blue}$L$\end{small}}
\put(46,41){\begin{scriptsize}$f_0(\Sigma_0)$\end{scriptsize}}
\put(2,55){\begin{scriptsize}$f_0(\Sigma_{1/2})$\end{scriptsize}}
\put(118,63){\begin{scriptsize}$f_{1/3}(\Sigma_0)$\end{scriptsize}}
\put(138,34){\begin{scriptsize}$f_{1/3}(\Sigma_{1/2})$\end{scriptsize}}
\put(246,66){\begin{scriptsize}$f_{2/3}(\Sigma_0)$\end{scriptsize}}
\put(221,33){\begin{scriptsize}$f_{2/3}(\Sigma_{1/2})$\end{scriptsize}}
\put(327,55){\begin{scriptsize}$f_1(\Sigma_{0})$\end{scriptsize}}
\put(363,41){\begin{scriptsize}$f_1(\Sigma_{1/2})$\end{scriptsize}}
\end{picture} 
\caption{The homotopy motion $\rho = \{f_t\}_{t \in I}$.}
\label{fig_rho}
\end{figure}
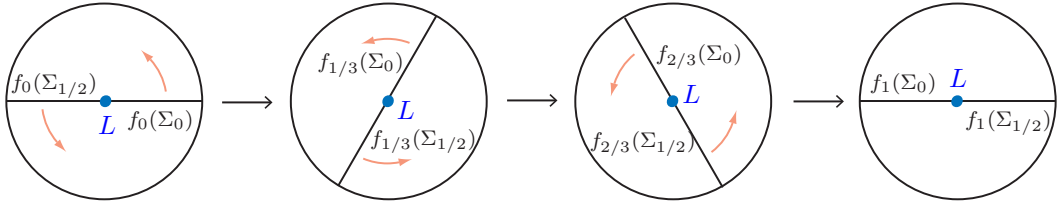
The second one, $\sigma = \sigma_{(L, \pi)}=\sigma_{(\Sigma_0,\varphi)}$, is defined by
\[\sigma(t) (x) = 
\left\{ 
\begin{array}{ll} 
r_t (x) & (x \in \Sigma_0)
\\
x & (x \in \Sigma_{1/2}),
\end{array}
\right.\]
see Figure \ref{fig_sigma}. 
\begin{figure}[htbp]
\centering\includegraphics[width=14cm]{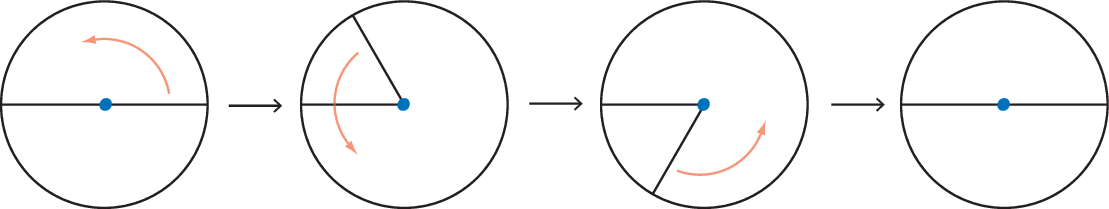}
\begin{picture}(400,0)(0,0)
\put(35,56){\begin{small}\color{blue}$L$\end{small}}
\put(148,50){\begin{small}\color{blue}$L$\end{small}}
\put(256,47){\begin{small}\color{blue}$L$\end{small}}
\put(357,56){\begin{small}\color{blue}$L$\end{small}}
\put(46,41){\begin{scriptsize}$g_0(\Sigma_0)$\end{scriptsize}}
\put(3,41){\begin{scriptsize}$g_0(\Sigma_{1/2})$\end{scriptsize}}
\put(136,70){\begin{scriptsize}$g_{1/3}(\Sigma_0)$\end{scriptsize}}
\put(110,41){\begin{scriptsize}$g_{1/3}(\Sigma_{1/2})$\end{scriptsize}}
\put(243,30){\begin{scriptsize}$g_{2/3}(\Sigma_0)$\end{scriptsize}}
\put(217,57){\begin{scriptsize}$g_{2/3}(\Sigma_{1/2})$\end{scriptsize}}
\put(366,41){\begin{scriptsize}$g_1(\Sigma_{0})$\end{scriptsize}}
\put(326,41){\begin{scriptsize}$g_1(\Sigma_{1/2})$\end{scriptsize}}
\end{picture} 
\caption{The homotopy motion $\sigma = \{ g_t \}_{t \in I}$.}
\label{fig_sigma}
\end{figure}
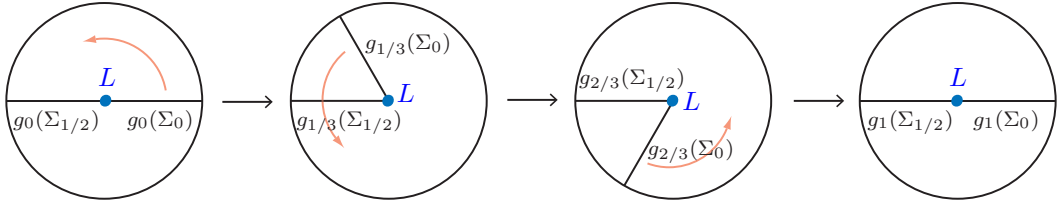
We call $\rho$ and $\sigma$, respectively,
the \textit{half book rotation} and the \textit{unilateral book rotation}
associated with the open book decomposition $(L,\pi)$ (or $(\Sigma_0,\varphi)$).

The elements of the group $\Mdy(M,\Sigma)$ obtained as
the terminal ends
$\rho(1)=\partial_+(\rho)$ and $\sigma(1)=\partial_+(\sigma)$
play a key role in the proof of 
the main Theorem \ref{thm:well-definedness of homological degree for aspherical manifolds}.
We note that $\rho(1)$ 
is orientation-reversing whereas
$\sigma(1)$ 
is orientation-preserving,
and they are related as follows.

\begin{lemma}
$\sigma(1)=\rho(1)\circ h$, 
where 
$h$ is the restriction to $\Sigma$ of the vertical $I$-bundle involution on $V_1$ 
with respect to the natural $I$-bundle structure given by $(L, \pi)$.
\end{lemma}

\begin{proof}
Under the identification 
$(M, L)=(\Sigma_0 \times \RR , \partial \Sigma_0 \times \RR ) / \sim$,
the following formulas hold for every $x\in\Sigma_0$.
\begin{align*}
h([x,0])&=[x,1/2], & h([x,1/2])&=[x,0], \\
\rho(1)([x,0])&=[x,1/2], & \rho(1)([x,1/2])&=[x,1]=[\varphi^{-1}(x),0], \\
\sigma(1)([x,0])&=[x,1]=[\varphi^{-1}(x),0], & \sigma(1)([x,1/2])&=[x,1/2] .
\end{align*}
By using these formulas, we see that the following hold for every $x\in\Sigma_0$,
which in turn imply the desired identity.
\begin{align*}
\rho(1)\circ h([x,0])
=\rho(1)([x,1/2])
=[\varphi^{-1}(x),0]=\sigma(1)([x,0]) , \\
\rho(1)\circ h([x,1/2]) 
=\rho(1)([x,0])
=[x,1/2]=\sigma(1)([x,1/2]).
\end{align*}
\end{proof}

For open book decompositions with trivial monodromies,
we have the following lemma.
Though it should be well-known, 
we provide a brief proof here, for
this plays 
an important role in Section \ref{sec:The group K(M,Sigma) for Heegard surfaces}.

\begin{lemma}
\label{lemma:trivial mondromy openbook1}
Let $M$ be a closed, orientable $3$-manifold 
that admits
an open book decomposition $(\Sigma_0,\id_{\Sigma_0})$ with trivial monodromy $\id_{\Sigma_0}$,
where $\Sigma_0$ is a compact, connected, orientable surface embedded in $M$. 
Let $\Sigma$ be the Heegaard surface of $M$
associated with the open book decomposition $(\Sigma_0,\id_{\Sigma_0})$. 
Then the following hold.
\begin{enumerate}
\renewcommand{\labelenumi}{$(\arabic{enumi})$}
\item
$M \cong \#_g (S^2 \times S^1)$, where $g$ is the first Betti number of $\Sigma_0$,
and $\Sigma$ is the unique minimal genus Heegaard surface of $M$. 
\item
The unilateral book rotation
$\sigma$ associated with $(\Sigma_0,\id_{\Sigma_0})$
determines a nontrivial element of $\mathcal{K}(M,\Sigma)$
of degree $1$
under a suitable orientation convention.  
\end{enumerate}
\end{lemma}

\begin{proof}
Let $\{\delta_i\}_{1\le i\le g}$ be a complete non-separating arc system of $\Sigma_0$,
namely a family of disjoint non-separating arcs which cuts $\Sigma_0$ into a disk.
Then the image of $\{\delta_i\times \RR\}_{1\le i\le g}$ 
in $M=(\Sigma_0 \times \RR , \partial \Sigma_0 \times \RR ) / \sim$
gives a family of disjoint non-separating spheres which cut $M$ into a $3$-ball.
Hence $M \cong \#_g (S^2 \times S^1)$ and $\Sigma$ is a genus $g$ Heegaard surface of $M$.
Since $\#_g (S^2 \times S^1)$ admits a unique Heegaard splitting of genus $g$
by Waldhausen \cite{Waldhausen68b}, Bonahon-Otal \cite{BonahonOtal83} and Haken \cite{Haken68},
we obtain the assertion (1).
Since the monodromy of the open book decomposition is the identity map,
the terminal end $\sigma(1)$ of $\sigma$ is the identity map. 
Thus $\sigma$ 
determines an element of $\mathcal{K}(M,\Sigma)$.
Obviously, $\deg (\sigma) =\deg (\hat\sigma:\Sigma\times S^1\to M) =1$, 
and so we obtain the assertion (2). 
\end{proof}

\section{The group $\mathcal{K}(M,\Sigma)$ for Heegaard surfaces of closed orientable $3$-manifolds}\label{sec:The group K(M,Sigma) for Heegard surfaces}

Let $M = V_1 \cup_\Sigma V_2$ be a Heegaard splitting of 
a closed, orientable $3$-manifold, and 
$j : \Sigma \to M$ the inclusion map. 
Recall the homomorphism
$\deg : \mathcal{K}(M,\Sigma)\to\ZZ$ introduced in 
Definition \ref{def:degree-KernelGroup}, 
and the fact that this homomorphism does not vanish if and only if
the pair $(M, \Sigma)$ 
is dominated by $\Sigma \times S^1$ (cf. Definition \ref{def:Sigma-domination}).
In this section, 
we discuss the problem of 
which pair $(M,\Sigma)$ 
of a closed, orientable $3$-manifold and its Heegaard surface $\Sigma$
is dominated by $\Sigma\times S^1$.

For irreducible $3$-manifolds, we obtain the following complete information
including the structure of the group $\mathcal{K}(M,\Sigma)$,
where $\Phi:\mathcal{K}(M,\Sigma)\to Z(\pi_1(M))$ is the 
the homomorphism intorduced in Definition
\ref{def:classification of pi1(C(S,M), j) for an aspherical manifold}.
(Note that, since $\Sigma$ is a Heegaard surface of $M$, we have $j_*(\pi_1(\Sigma))=\pi_1(M)$, 
thus, the codomain of $\Phi$
is the center $Z(\pi_1(M))$ of the fundamental group of $M$.)

\begin{theorem}
\label{thm:non-zero degree maps 1}
Let $M$ be a closed, orientable, irreducible $3$-manifold 
and $\Sigma$ a Heegaard surface of $M$. 
\begin{enumerate}
\renewcommand{\labelenumi}{$(\arabic{enumi})$}
\item
Suppose that $M$ is aspherical.
Then $(M, \Sigma)$ is not dominated by $\Sigma \times S^1$.
To be precise, 
$\Phi$ gives an isomorphism $\mathcal{K} (M, \Sigma) \cong Z(\pi_1(M))$,
and the homomorphism $\deg : \mathcal{K} (M, \Sigma)\to\ZZ$ vanishes. 
Thus if $M$ is a Seifert fibered space with orientable base orbifold,
then $\mathcal{K} (M, \Sigma)$ is isomorphic to $\ZZ^3$ or $\ZZ$ 
according to whether $M$ is the $3$-torus $T^3$ or not;
otherwise, $\mathcal{K} (M, \Sigma)$ is the trivial group.
\item
Suppose that $M$ is non-aspherical, or equivalently,
$M$ has the geometry of $S^3$.
Then $(M, \Sigma)$ is dominated by $\Sigma \times S^1$.
To be precise, the following holds.
\begin{enumerate}
\item[{\rm (i)}]
If $g(\Sigma)\ge 2$,
then the product homomorphism $\Phi\times \deg$
induces an isomorphism 
$\mathcal{K} (M, \Sigma) \cong Z(\pi_1(M))\times |\pi_1(M)|\cdot\ZZ$.
\item[{\rm (ii)}]
If $g(\Sigma) \leq 1$,
then the homomorphism $\deg$
induces an isomorphism
$\mathcal{K} (M, \Sigma) \cong |\pi_1(M)|\cdot\ZZ$. 
\end{enumerate}
\end{enumerate}
\end{theorem}

\noindent The proof of Theorem \ref{thm:non-zero degree maps 1} will be given in 
Subsections \ref{subsection:aspherical-case} and \ref{subsection:aspherical-case2}. 

For $3$-manifolds which are not necessarily irreducible, 
we obtain the following partial result, whose proof will be given in 
Subsection \ref{subsection:general-case}.

\begin{theorem}
\label{thm:non-zero degree maps 2}
Let $M$ be a closed, orientable $3$-manifold
and $\Sigma$ a Heegaard surface of $M$. 
\begin{enumerate}
\renewcommand{\labelenumi}{$(\arabic{enumi})$}
\item
If $M$ contains an aspherical prime summand, then 
$(M, \Sigma)$ is not dominated by $\Sigma \times S^1$.
\item
If $M=\#_g (S^2 \times S^1)$ for some $g\ge 1$,
then $(M,\Sigma)$ is dominated by $\Sigma \times S^1$.
To be precise, $\deg (\mathcal{K} (M, \Sigma)) = \ZZ$.
\item
If $M=\RP^3 \# \RP^3$,
then $(M,\Sigma)$ is dominated by $\Sigma \times S^1$.
To be precise, $\deg (\mathcal{K} (M, \Sigma)) = 2\ZZ$.
\end{enumerate}
\end{theorem}

By Kneser-Milnor prime decomposition theorem \cite{Kneser1929, Milnor1962} (cf. \cite{Hempel76, Jaco80}),
every closed, orientable $3$-manifold $M$ admits a unique prime decomposition,
and by the geometrization theorem 
established by Perelman 
\cite{Perelman02, Perelman03a, Perelman03b}
(see \cite{BBMBP, CaoZhu06, KleinerLott08, MorganTian07, MorganTian14}
for exposition),
each prime factor admits a unique decomposition by tori into geometric manifolds,
i.e., those which have one of Thurston's $8$ geometries. 
This together with the sphere theorem implies that 
for a closed, orientable $3$-manifold $M$, the following three conditions are equivalent: 
(i) $M$ is aspherical, (ii) $M$ is irreducible and $\pi_1 (M)$ is not finite, (iii) the universal covering space of $M$ is 
homeomorphic to $\RR^3$. 
If $M$ is non-aspherical
then either 
$M$ admits the geometry of $S^3$ or $S^2\times \RR$, or 
$M$ is non-prime  (cf. \cite[Chapter 1]{BBMBP}).
Here $\RP^3 \# \RP^3$ is the unique geometric $3$-manifold which is non-prime. 
Thus, Theorems \ref{thm:non-zero degree maps 1}
and \ref{thm:non-zero degree maps 2}  
especially imply the following corollary.

\begin{corollary}
\label{cor:non-zero degree maps}
Let $M$ be a closed, orientable, $3$-manifold 
which is either prime or geometric,
and let $\Sigma$ a Heegaard surface of $M$. 
Then $(M,\Sigma)$ is dominated by $\Sigma \times S^1$
if and only if $M$ is non-aspherical,
namely $M$ admits the geometry of $S^3$ or $S^2\times \RR$.
\end{corollary}

We do not know, however, what happens when 
$M$ is non-prime and $M$ has no aspherical prime summand, 
in other words, each prime summand of $M$ is $S^2 \times S^1$, or has the geometry of $S^3$, 
except when $M=\#_g (S^2 \times S^1)$ or $\RP^3 \# \RP^3$.

\begin{question}
\label{question:nonprime-general}
{\rm
Let $M=\#_{i=1}^n M_i$ ($n\ge 2$) be a closed, orientable non-prime $3$-manifold 
such that each $M_i$ is either $S^2\times S^1$ or admits the geometry of $S^3$.
When is a Heegaard surface $\Sigma$ of $M$ dominated by $\Sigma \times S^1$?}
\end{question}

The remainder of this section is devoted to the proof of 
Theorems \ref{thm:non-zero degree maps 1} and \ref{thm:non-zero degree maps 2}.

\subsection{Proof of Theorem \ref{thm:non-zero degree maps 1}(1)}
\label{subsection:aspherical-case}
Suppose that $M$ is aspherical.
Since a closed, orientable, irreducible $3$-manifold that admits Heegaard splitting of genus at most $1$ is either $S^3$ or a lens space, 
which are not aspherical, we have 
$g(\Sigma)\ge 2$ and 
$\mathcal{K} (M, \Sigma) \cong \pi_1(C (\Sigma, M), j)$
by Lemma \ref{lem:monodromy group and kernel group}.
We see by Lemma \ref{lem:classification of pi1(C(S,M), j) for an aspherical manifold} that
$\Phi:\mathcal{K} (M, \Sigma) \to Z(\pi_1(M))$ is an isomorphism.
If $Z (\pi_1 (M)) = 1$, there is nothing to prove.
Suppose that $Z (\pi_1 (M))$ is non-trivial. 
By the Seifert fibered space conjecture proved by 
Gabai \cite{Gabai92} and Casson-Jungreis \cite{CassonJungreis94}, 
$M$ is then a Seifert fibered space with orientable base orbifold. 
If $M$ is not the $3$-torus $T^3$, then, $M$ admits a unique Seifert fibration with 
orientable base orbifold, and the center $Z(\pi_1(M))\cong \ZZ$ is generated by an element 
represented by a regular fiber of the Seifert fibration of $M$, see Jaco \cite[VI]{Jaco80}. 
When $M = T^3$, we have $Z(\pi_1(M)) = \pi_1 (M) \cong \ZZ^3$, and 
any primitive element of $Z(\pi_1(M))$ can be realized as a regular fiber 
of a Seifert fibration of $M$. 
In any case, let $z$ be a primitive element of $Z(\pi_1 (M))$. 
Equip $M$ with a Seifert fibration where $z$ is represented by its regular fiber. 
Fix a faithful action of $S^1 = \RR / \ZZ$ on $M$ that is compatible with the Seifert fibration.   
Let $\alpha_z$ be the homotopy motion of $\Sigma$ defined by
$\alpha_z(t)(x)=t\cdot x$ for $t\in I$ and $x\in\Sigma$,
where $t\cdot x$ is the image of $x$ by the action of $t\in S^1$.
Then we see that $\alpha_z$ determines an element of $\mathcal{K} (M, \Sigma)$
and that $\Phi([\alpha_z])=z$.
Thus we have only to show that the degree of the map $\hat\alpha_z :\Sigma\times S^1\to M$ 
is $0$.
To this end, let $Y_1$ be a spine of the handlebody $V_1$ in $M$ 
bounded by the Heegaard surface $\Sigma$, 
and let $\{r_t\}_{t\in I}$ be a strong deformation retraction of $V_1$ onto $Y_1$,
namely $r_0=\id_{V_1}$, $r_t|_{Y_1}=\id_{Y_1}$ ($t\in I$), and 
$r_1(V_1)=Y_1$.
Define a map $H=\{h_s\}_{s\in I}:(\Sigma\times S^1)\times I \to M$ by
$H(x,t,s)=t\cdot r_s(x)$. Then $h_0=\hat\alpha_z$
and $h_1(\Sigma\times S^1)=S^1\cdot Y_1$.
Since $Y_1$ is $1$-dimensional, the image $h_1(\Sigma\times S^1)$
is a proper subset of $M$. 
Hence $\deg ( \hat\alpha_z ) = \deg ( h_1 ) = 0$.
This completes the proof of Theorem \ref{thm:non-zero degree maps 1}(1). 

\subsection{Proof of Theorem \ref{thm:non-zero degree maps 1}(2)}
\label{subsection:aspherical-case2}

Suppose that $M$ is non-aspherical.
Since $M$ is irreducible by assumption of the theorem,
the geometrization theorem implies that 
$M$ admits the geometry of $S^3$, namely
$M\cong S^3/G$ for some finite subgroup $G$
of $\SO(4)$ acting freely on $S^3$.

Suppose first that $g(\Sigma)\ge 2$.
Then $\mathcal{K} (M, \Sigma) \cong \pi_1(C (\Sigma, M), j)$
by Lemma \ref{lem:monodromy group and kernel group}.
We show that the homomorphism
$\Phi\times \deg:\pi_1(C (\Sigma, M), j)\to Z(G)\times \ZZ$
is injective and that its image is
$Z(G)\times |G|\cdot\ZZ$.

To see the injectivity, pick an element 
$[\alpha] \in \ker(\Phi\times \deg)$, and  
consider the maps $\hat \alpha$ and $\hat e:\Sigma\times S^1 \to M$
induced by $\alpha$ and the identity motion $e$, respectively.
Since $[\alpha]\in \ker \Phi$,
the homomorphisms $\hat\alpha_*$ and
$\hat e_*:\pi_1(\Sigma\times S^1)\to \pi_1(M)$ are equivalent.
Since $\deg (\hat\alpha)=\deg([\alpha])=0=\deg([e])=\deg (\hat e)$,
we see by
Proposition \ref{prop:obstruction3}
that $\hat \alpha$ and $\hat e$ are homotopic.
Thus $[\alpha]$ and $[e]$ are conjugate and so identical in 
$\pi_1(C (\Sigma, M), j )$.
Hence $\Phi\times \deg$ is injective.

Next we show that
the image of $\Phi\times \deg$ is equal to $Z(G)\times |G|\cdot\ZZ$.
To this end, we need the lemma below.
Though this should be well-known, we provide a brief proof here,
for we could not find a reference.

\begin{lemma}
\label{lemma:aircle-action}
The center $Z(G)$ is the cyclic group generated 
by the homotopy class of a regular orbit
of a circle action on $M=S^3/G$.
\end{lemma}

\begin{proof}
If $M$ is a lens space, then $M$ has a genus-$1$ Heegaard splitting $V_1\cup V_2$,
and there is a circle action on $M$ such that a regular orbit forms a core circle
of $V_1$.
Since $G=\pi_1(M)$ is the cyclic group generated by the homotopy class 
of the regular orbit, the lemma obviously holds.
So we may assume $M$ is not a lens space.
Then $M$ admits a circle action such that the orbit space $\mathcal{O}$
is the $2$-dimensional spherical orbifold $S^2(p,q,r)$ 
where $(p,q,r)$ is $(2,2,r)$ with $r\ge 2$ or
$(2,3,r)$ with $r\in \{3,4,5\}$ (see \cite{Scott83}).
Let $z$ be an element of $G$ represented by a regular orbit of 
the circle action.
Then the subgroup $\langle z\rangle$ 
is contained in $Z(G)$ and the quotient $G/\langle z\rangle$ is identified with
the orbifold fundamental group $\pi_1^{\mathrm{orb}}(\mathcal{O})$.
The center of this group is trivial unless 
$\mathcal{O}=S^2(2,2,r)$ for some even integer $r\ge 2$.
Thus the assertion holds except for this case.

Let $\psi:S^3\to \SO(3)$ and
$\phi:S^3\times S^3 \to \SO(4)$ be the universal covering projection.
Let $\ZZ_m$ and $\Dh_r$, respectively, be the order-$m$ cyclic subgroup 
and the order-$2r$ dihedral subgroup of $\SO(3)$,
which are unique up to conjugation. 
Set $\tilde\ZZ_m=\psi^{-1}(\ZZ_m)$ and $\tilde\Dh_{r}=\psi^{-1}(\Dh_r)$.
Then, by \cite[Theorem 4.11]{Scott83}, we have 
$G=\phi(\tilde\ZZ_m\times \tilde\Dh_{r})$ 
after conjugation,
in the exceptional case where $\mathcal{O}=S^2(2,2,r)$ with $r\ge 2$ even.
Moreover, the subgroup $\langle z\rangle$ is 
identified with the subgroup $\phi(\tilde\ZZ_m\times \{1\})$.
It is easy to check that any element 
of $G$ that does not belong to
$\phi(\tilde\ZZ_m\times \{1\})$
is not central.
\end{proof}

By the above lemma, $Z(G)$ is generated by an element $z$ 
represented by a regular orbit 
of a circle action on $M$.
Consider the homotopy motion $\alpha_z$ of $M$ 
generated by the circle action, as in Subsection \ref{subsection:aspherical-case},
and consider the element $[\alpha_z]$ of  
$\pi_1(C (\Sigma, M), j )$ it represents.
Then, by the argument in Subsection \ref{subsection:aspherical-case},
we see $\Phi([\alpha_z])=z$ and $\deg([\alpha_z])=0$.
Thus we have shown that the image of 
$\Phi\times \deg$ contains $Z(G)\times\{0\}$.
 
Next, we show that the image of $\Phi\times \deg$ contains
$\{1\}\times |G|\cdot\ZZ$.
Pick a point $x$ in the interior of $\Sigma\times I$,
and modify the trivial motion $e:\Sigma\times I \to M$
on a regular neighborhood $N$ of $x$ as follows
(cf. \cite[p.353]{DavisKirk01}, \cite[(1.8)]{Olum50}).
Let $N_0$ be a smaller regular neighborhood of $x$.
Then, by pinching the $2$-sphere $\partial N_0$ into a point,
we obtain a continuous map $r$ from $N$ onto the wedge sum 
$N\vee S^3=N\cup_{\{x\}} S^3$,
such that 
(i) $r|_{\partial N}$ is the identity map,
(ii) $r|_{N-N_0}$ is a homeomorphism onto $N-\{ x \}$,
(iii) $r|_{N_0}$ induces a homeomorphism from $N_0/\partial N_0$ onto $S^3$.
Let $p:S^3\to M$ be the universal covering projection,
such that, when regarded as a map from the subspace $S^3$ of $N\vee S^3$,
it maps the point $x=N\cap S^3\subset S^3$ 
into the point $e(x)\in M$.
Then we obtain a continuous map $e|_N \vee p:N\vee S^3 \to M$. 
See Figure \ref{fig.alternation}. 
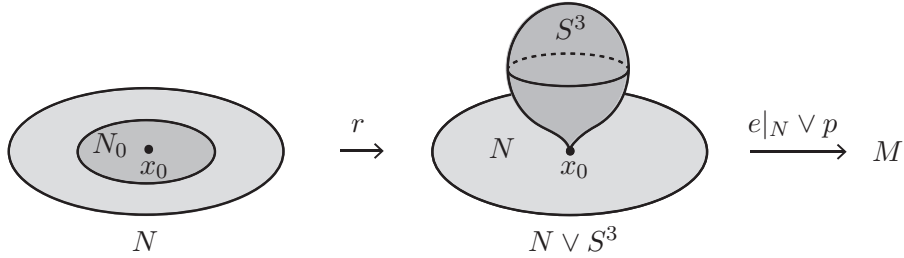
\begin{figure}[htbp]
\centering\includegraphics[width=12cm]{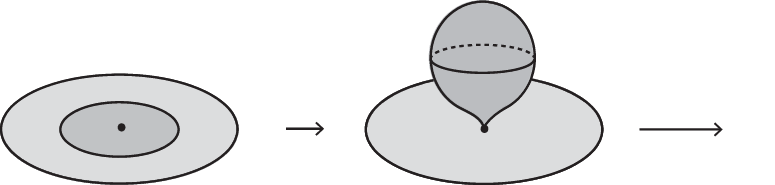}
\begin{picture}(400,0)(0,0)
\put(75,0){$N$}
\put(225,0){$N \vee S^3$}
\put(78,29){$x_0$}
\put(60,37){$N_0$}
\put(237,29){$x_0$}
\put(210,35){$N$}
\put(235,80){$S^3$}
\put(355,34){$M$}
\put(158,45){$r$}
\put(308,45){$e|_N \vee p$}
\end{picture} 
\caption{The pinching map $r$ from $N$ to $N \vee S^3$, and the map $e|_N \vee p$ from 
$N \vee S^3$ to $M$.}
\label{fig.alternation}
\end{figure}
Let $\beta:\Sigma\times I \to M$ the continuous map
obtained from $e$ by redefining $e$ on $N$ to be the composition $(e|_N \vee p)\circ r$.
Then $\beta$ determines an element $\pi_1(C (\Sigma, M), j )$
such that $\Phi([\beta])=1$ and $\deg([\beta])=|G|$.
Thus the image of $\Phi\times \deg$ contains
$\{1\}\times |G|\cdot\ZZ$.

On the other hand,
Proposition \ref{prop:obstruction3} implies, 
for each $z\in Z(G)$, that
the set $\deg([\alpha])=\deg(\hat\alpha)$,
where $[\alpha]$ runs over $\Phi^{-1}(z)$,
is of the form $d+|G|\cdot\ZZ$ for some $d\in\ZZ$,
because the homomorphisms 
$\hat\alpha_*:\pi_1(\Sigma\times S^1)\to \pi_1(M)$
induced by $[\alpha]\in \Phi^{-1}(z)$ are all equivalent.
In a preceding paragraph, we have observed that
$\deg([\alpha_z])=0$ for the element $[\alpha_z]\in \Phi^{-1}(z)$
(for the generator $z$ of $Z(G)$).
Thus we may assume $d=0$.
Hence the mage of $\Phi\times \deg$ is contained in $Z(G)\times |G|\cdot\ZZ$.

Thus we have shown that
the image of $\Phi\times \deg$ is contained in $Z(G)\times |G|\cdot\ZZ$
and that it contains both
$Z(G)\times\{0\}$ and $\{1\}\times |G|\cdot\ZZ$.
Hence the image of $\Phi\times \deg$ is equal to  $Z(G)\times |G|\cdot\ZZ$.
Since the injectivity of $\Phi\times \deg$ is already proved,
$\Phi\times \deg$ induces an isomorphism 
from $\mathcal{K} (M, \Sigma) \cong \pi_1(C (\Sigma, M), j)$
onto $Z(G)\times |G|\cdot\ZZ$,
completing the proof of
of Theorem \ref{thm:non-zero degree maps 1}(2) when $g(\Sigma)\ge 2$.

Suppose that $g(\Sigma) \leq 1$. 
Then $\mathcal{K} (M, \Sigma) \cong \pi_1 ( C (\Sigma , M) , j) / \mathscr{I} ( \pi_1 ( J (\Sigma, M ) , j) )$ by Lemma \ref{lem:monodromy group and kernel group}.
Note that the preceding argument shows that
$\Phi\times \widetilde{\deg}$ induces an isomorphism 
$\pi_1(C (\Sigma, M), j) \cong Z(G)\times |G|\cdot\ZZ$,
where $\widetilde{\deg}$ is the composition of
$\pi_1(C (\Sigma, M), j)\to \mathcal{K} (M, \Sigma)$ and 
$\deg:\mathcal{K} (M, \Sigma) \to \ZZ$.
When $g(\Sigma) = 1$, by 
the argument in the proof of
Lemma \ref{lem:Haken1}, 
we see that the image of the subgroup
$\mathscr{I} ( \pi_1 ( J (\Sigma, M ) , j) )$
by the isomorphism is equal to $Z(G)$. 
When $g(\Sigma) = 0$, $\Sigma$ is the genus-$0$ Heegaard splitting of $M = S^3$, 
thus, $Z(G) = G = 1$. 
Hence, in any case, we see that
$\deg$ induces an isomorphism $\mathcal{K} (M, \Sigma) \to |G|\cdot\ZZ$
by Lemma \ref{lem:monodromy group and kernel group},
completing the proof of
Theorem \ref{thm:non-zero degree maps 1}(2). 

\subsection{Proof of Theorem \ref{thm:non-zero degree maps 2}}
\label{subsection:general-case}

Case 1. $M$ contains a prime aspherical summand.
Then $M$ is a connected sum $M_1\# M_2$ of an aspherical prime manifold $M_1$
and another $3$-manifold $M_2$, which is possibly $S^3$.
Suppose on the contrary that there is a Heegaard surface $\Sigma$ of $M$,
such that $(M,\Sigma)$ admits a $\Sigma$-domination
$\phi:(\Sigma\times S^1, \Sigma\times \{0\}) \to (M,\Sigma)$.
By Haken's theorem on Heegaard surfaces of composite manifolds
(see \cite[Theorem II.7]{Jaco80}),
$(M,\Sigma)$ is a pairwise connected sum $(M_1,\Sigma_1)\#(M_2,\Sigma_2)$
where $\Sigma_i$ is a Heegaard surface of $M_i$ ($i=1,2$).
By pinching $(M_2,\Sigma_2)$ into a point, 
we obtain from $\phi$ a $\Sigma_1$-domination of $(M_1,\Sigma_1)$.
This contradicts Theorem \ref{thm:non-zero degree maps 1}(1). 
Hence we obtain Theorem \ref{thm:non-zero degree maps 2}(1).

\medskip

In order to treat the remaining cases,
recall that a {\it stabilization} of a Heegaard splitting $M = V_1 \cup_\Sigma V_2$ 
(or a Heegaard surface $\Sigma \subset M$) is an operation 
to obtain a Heegaard splitting $M = V_1' \cup_{\Sigma'} V_2'$ 
(or a Heegaard surface $\Sigma' \subset M$)
of higher genus by adding $V_1$ 
a trivial $1$-handle, that is, a $1$-handle whose core is parallel to $\Sigma$ in $V_2$, and 
removing that from $V_2$. 

\begin{lemma}
\label{lem:stabilization}
Let $M$ be a closed, orientable $3$-manifold,  
$\Sigma$ a Heegaard surface for $M$, and 
$\Sigma'$ a Heegaard surface obtained by a stabilization 
from $\Sigma$. 
If there exists a degree-$d$ $\Sigma$-domination of
$(M, \Sigma)$, 
then there exists a degree-$d$ $\Sigma'$-domination of 
$(M, \Sigma')$ as well. 
\end{lemma}
\begin{proof}
Suppose that there exists a degree-$d$ $\Sigma$-domination 
$\phi : (\Sigma \times S^1, \Sigma \times \{ 0 \}) \to (M, \Sigma)$. 
Without loss of generality we can assume that $\phi (x , 0) = x$ for any $x \in \Sigma$. 
Further, we can assume that the stabilization is performed in a $3$-ball $B$ in $M$ that intersects $\Sigma$ in a disk, 
thus, $\Sigma - B = \Sigma' - B$. 
Then there exists a homotopy $F = \{ f_t \}_{t \in I} : \Sigma' \times I \to M$ such that 
\begin{enumerate}
\item
$f_0(x) = x$, for $x \in \Sigma'$; 
\item
$f_t(x) = x$ for $x \in \Sigma' - B$, $t \in I$;  
\item
$f_t (x) \in B$ for $x \in B \cap \Sigma$, $t \in I$; and 
\item
$f_1 (\Sigma') = \Sigma$. 
\end{enumerate}
Using this homotopy, we can define a $\Sigma'$-domination
$\phi' : ( \Sigma' \times S^1 , \Sigma' \times \{ 0 \}) \to (M, \Sigma')$ by 
\[
\phi' (x, \theta) = 
\left\{ 
\begin{array}{ll} 
f_{3\theta} (x) &  (0 \leq \theta \leq 1/3) \\
\phi (f_1(x), 3\theta - 1) &  (1/3 \leq \theta \leq 2/3) \\
f_{3 - 3\theta} (x) &  ( 2/3 \leq \theta \leq 1) . 
\end{array}
\right.
\]
Since the homotopy $F$ moves $\Sigma'$ only inside the local 3-ball $B$, 
the degree of $\phi'$ is $d$. 
\end{proof}

Case 2. $M=\#_g (S^2 \times S^1)$ for some $g\ge 1$.
By Waldhausen \cite{Waldhausen68b}, Bonahon-Otal \cite{BonahonOtal83} and Haken \cite{Haken68}, 
any Heegaard splitting of $M$ is a stabilization of 
the unique genus-$g$ Heegaard splitting $M=V_1 \cup_\Sigma V_2$.
Therefore, by Lemma \ref{lem:stabilization}, we may assume 
$\Sigma$ is the unique genus-$g$ 
Heegaard surface.
Then by Lemma \ref{lemma:trivial mondromy openbook1},
there exists a degree-$1$ $\Sigma$-domination of $(M, \Sigma)$. 
Since $\deg : \mathcal{K}(M,\Sigma)\to\ZZ$
is a homomorphism,
we have $\deg(\mathcal{K}(M,\Sigma))=\ZZ$,
completing the proof of Theorem \ref{thm:non-zero degree maps 2}(2).

\begin{remark}
{\rm
It is proved in \cite[Proposition 4]{KotschickNeofytidis13}
that there exists a double branched covering map 
from $\Sigma\times S^1$ to $\#_g (S^2 \times S^1)$ where $\Sigma$
is a closed, orientable surface of genus $g$.
(See \cite[Lemma 2.3] {Neofytidis18} for a related interesting result.)
That map actually gives a domination of 
the minimal genus Heegaard surface of $\#_g (S^2 \times S^1)$
by $\Sigma\times S^1$.
However, this does not imply the full statement of 
Theorem \ref{thm:non-zero degree maps 2}(2),
because the map has degree $2$ and it gives domination 
of only the minimal genus Heegaard surface. 
}
\end{remark}

Case 3. $M = \RP^3 \# \RP^3$.
By Montesinos-Safont \cite{MontesinosSafont88} and Haken \cite{Haken68}, 
any Heegaard splitting of $M$ is a stabilization of the unique genus-$2$ 
Heegaard splitting $M = V_1 \cup_\Sigma V_2$. 
Therefore, by Lemma \ref{lem:stabilization}, we may assume 
$\Sigma$ is the unique genus-$2$ 
Heegaard surface.

We first show that 
there exists a degree-$d$ 
$\Sigma$-domination of $(M, \Sigma)$ 
for any even integer $d$. 
For this, it is enough to find a degree-$2$ 
$\Sigma$-domination of $(M, \Sigma)$. 

Let $\tau$ be the antipodal map of $S^2$,  and $\eta$ the involution of $S^1$ defined by 
$\eta(\theta)= -\theta$.
Identify $M = \RP^3 \# \RP^3$ with $(S^2\times S^1)/(\tau\times\eta)$,
and let $p:S^2\times S^1\to M$ be the covering projection.
Thus we can regard $M$ as an $S^2$-bundle over the orbifold
$S^1/\eta$ with underlying space $[0,1/2]$.
Choose disjoint disks $D_-$ and $D_+$ in $S^2$ with $\tau (D_-) = D_+$. 
Let $R = I \times I$ be a rectangle in $\cl(S^2-(D_+\cup D_-))$ such that 
$R \cap D_- = \{ 0 \} \times I$ and $R \cap D_+ = \{ 1 \} \times I$. 
Then $\tilde V_1:=( (D_- \cup D_+) \times S^1 ) \cup (R \times [1/6, 2/6]) \cup ( \tau(R) \times [4/6,5/6])$
is a $(\tau\times\eta)$-invariant handlebody of genus $3$,
and its exterior 
$\tilde V_2:=\cl(M-\tilde V_1)$
is also a $(\tau\times\eta)$-invariant handlebody of genus $3$.
Thus the pair $(\tilde V_1, \tilde V_2)$ determines a $(\tau\times\eta)$-invariant
Heegaard splitting of $S^2\times S^1$, and
it projects to the genus-$2$ Heegaard splitting $(V_1,V_2)$ of $M$,
where $V_i:=p(\tilde V_i)$ ($i=1,2$).
See Figure \ref{fig.RP3RP3}. 
\begin{figure}[htbp]
\centering\includegraphics[width=12cm]{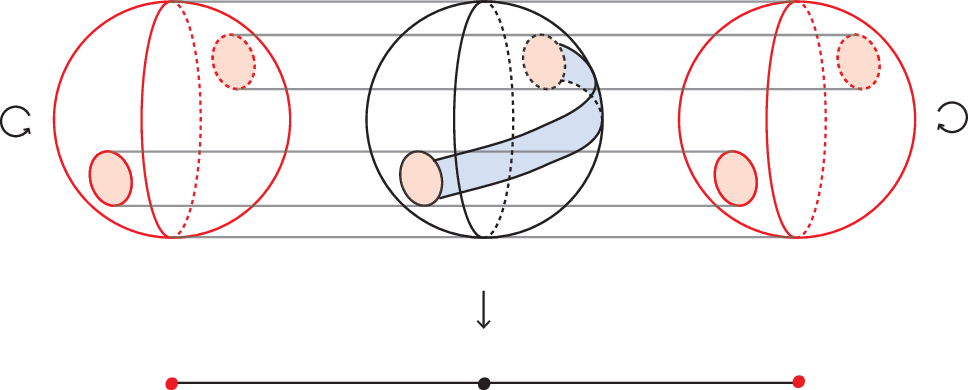}
\begin{picture}(400,0)(0,0)
\put(104,126){\scriptsize \color{red} $D_+$}
\put(213,126){\scriptsize \color{red} $D_+$}
\put(324,126){\scriptsize \color{red} $D_+$}
\put(61,85){\scriptsize \color{red} $D_-$}
\put(170,85){\scriptsize \color{red} $D_-$}
\put(281,85){\scriptsize \color{red} $D_-$}
\put(225,100){\small \color{blue} $R$}
\put(85,0){$0$}
\put(191,0){$1/4$}
\put(308,0){$1/2$}
\put(18,106){$\tau$}
\put(375,106){$\tau$}
\end{picture} 
\caption{The manifold $\RP^3 \# \RP^3$ as an $S^2$-bundle over 
the orbifold $S^1/\eta$.}
\label{fig.RP3RP3}
\end{figure}
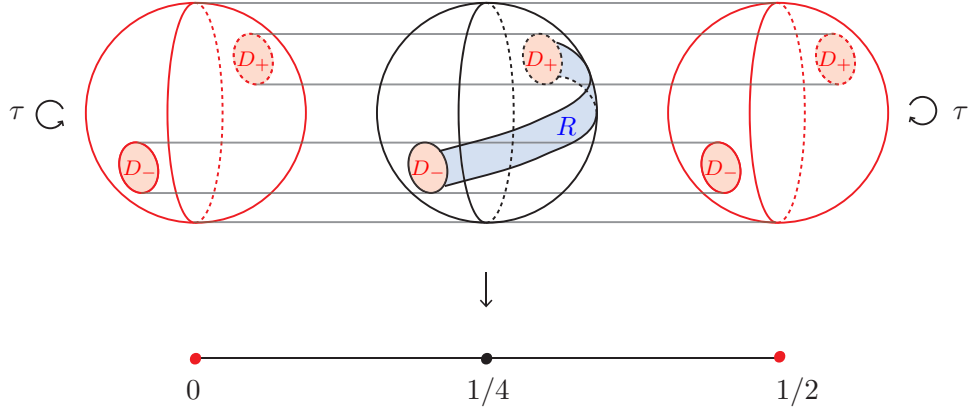
We are going to construct a domination of the Heegaard surface $\Sigma:=V_1\cap V_2$,
by constructing an equivariant domination of the Heegaard surface 
$\tilde\Sigma= p^{-1}(\Sigma)=\tilde V_1\cap \tilde V_2$.

To this end, consider an annulus 
$A:=\delta \times S^1\subset \tilde\Sigma$,
where $\delta$ is an arc in $\partial D_{-}$ disjoint from the rectangle $R$.
Then there is an open book decomposition $(L,\pi)$ 
of $S^2\times S^1$
with $L=\partial A$,
such that $A$ is the page $\pi^{-1}(0)\cup L$.
Let $\{r_t\}_{t\in\RR}$ be the 
book rotation with respect to $(L,\pi)$.
Now consider the conjugate of the above open book decomposition 
by the covering involution $\tau\times\eta$,
and let $\{r'_t\}_{t\in\RR}$ be the associated book rotation 
obtained from $\{r_t\}_{t\in\RR}$
through conjugation by $\tau\times\eta$.
Observe that the page $A':=(\tau\times\eta)(A)$ is disjoint from $A$.
Since the monodromy of the open book decomposition $(L,\pi)$ is the identity,
we can construct a $\ZZ/2\ZZ$-equivariant map
$\tilde\phi:\tilde\Sigma\times S^1\to S^2\times S^1$
 by 
\[ 
\tilde\phi(x,t)=
\begin{cases}
r_t(x) & (x\in A) \\
r'_t(x) & (x\in A') \\
x & (x\in \tilde\Sigma -(A\cup A') ) .
\end{cases}
\]
This map naturally induces a $\Sigma$-domination 
$\phi: (\Sigma \times S^1, \Sigma \times \{ 0 \}) \to (M, \Sigma)$ 
whose restriction to $\Sigma \times \{ 0 \}$ is 
a homeomorphism onto the Heegaard surface $\Sigma$,
that is actually the identity map under a natural identification of 
the two surfaces.
Since the degree of each of  $\tilde\phi$, $p$ and the map $\tilde\Sigma\times S^1\to \Sigma\times S^1$
is $2$, the degree of $\phi$ is $2$. 

To show the other direction, suppose that 
$\phi$ is a degree-$d$ map from 
$\Sigma \times S^1$ to $\RP^3 \# \RP^3$. 
Here we do not need to require that $\phi (\Sigma \times \{ 0 \})$ is a Heegaard surface of 
$\RP^3 \# \RP^3$. 
Let $p: \RP^3 \# \RP^3 \to \RP^3$ be a degree-$1$ map defined by pinching one summand $\RP^3$ to 
a 3-ball in the other summand $\RP^3$. 
Then the composition $p \circ \phi$ is a degree-$d$ map from $\Sigma \times S^1$ to $\RP^3$. 
From Hayat-Legrand-Wang-Zieschang \cite[Theorem 2]{Hayat-LegrandWangZieschang02} 
it follows that $d$ should be an even number. 
This completes the proof of Theorem \ref{thm:non-zero degree maps 2}(3).

\section{Gap between $\Mdy(M,\Sigma)$ and the subgroup $\langle \Mdy (V_1), \Mdy (V_2) \rangle$}
\label{sec:Gap between Mdy(M,Sigma) and the subgroup Mdy (V_1), Mdy (V_2)}

In this section, we show the following theorem, which gives a partial answer to 
Question \ref{question:motivation}(2).
 
\begin{theorem}
\label{thm:well-definedness of homological degree for aspherical manifolds}
Let $M = V_1 \cup_{\Sigma} V_2$ be a Heegaard splitting of 
a closed, orientable $3$-manifold $M$ induced from an open book decomposition. 
If $M$ has an aspherical prime summmand, 
then we have $\langle \Mdy (V_1), \Mdy (V_2) \rangle  \lneq \Mdy (M, \Sigma)$. 
\end{theorem}

In fact, we will see that neither $\rho(1)$ nor 
$\sigma(1)$, defined in Section \ref{sec:Open book rotations}, is not contained in 
$\langle \Mdy (V_1), \Mdy (V_2) \rangle$  under the assumption of Theorem 
\ref{thm:well-definedness of homological degree for aspherical manifolds}. 
To show this, we will define a $\ZZ^2$-valued invariant  $\Deg(f)$
for elements $f$ of 
$\Mdy (M, \Sigma)$. 
To this end, we first define a $\ZZ^2$-valued invariant  $\widehat{\Deg}(\alpha)$
for elements $\alpha$ of the homotopy motion group $\Pi (M, \Sigma)$, and
study its basic properties. We then
show, by using Theorem \ref{thm:non-zero degree maps 2}(1) 
that it descends to an invariant  
for elements of $\Mdy (M, \Sigma)$
when $M$ satisfies the assumption of 
Theorem \ref{thm:well-definedness of homological degree for aspherical manifolds}.

\begin{remark}
{\rm
Let $\Sigma$ be a Heegaard surface of a closed, orientable 3-manifold $M$. 
The existence of a gap between $\langle \Gamma (V_1), \Gamma (V_2) \rangle $ and 
$\Gamma (M, \Sigma)$ given in the above theorem
implies, in particular, that 
the Seifert-van Kampen-like theorem for the homotopy motion group 
$\Pi (M, \Sigma)$ is no longer valid as in the following meaning,
though $\Pi (M, \Sigma)$ could be regarded as a generalization of the fundamental group
(cf. Remark \ref{rem:homotopy motion group and the fundamental group} and 
Theorem \ref{thm:homotopy motion group for a local surface}).
Consider the homotopy motion groups 
$\Pi (V_1, \Sigma)$ and $\Pi (V_2, \Sigma)$.
Recall that the group $\Gamma (V_i)$ ($i=1,2$) is the image of the natural map 
$\partial_+ : \Pi (V_i , \Sigma) \to \MCG (\Sigma)$. 
Since a homotopy motion of $\Sigma$ in $V_i$ is that of $\Sigma$ in $M$ as well, 
we have a canonical map 
$I_i: \Pi (V_i , \Sigma) \to \Pi (M , \Sigma)$.  
Since the manifold $M$ is obtained by gluing $V_1$ and $V_2$ along $\Sigma$, 
one might expect that 
\[ \langle I_1 (\Pi (V_1, \Sigma)) , I_2 ( \Pi (V_2, \Sigma) ) \rangle = \Pi (M, \Sigma),\] 
that is, $\Pi (M , \Sigma)$ is generated by 
elements of $I_1 (\Pi (V_1 , \Sigma))$ and $I_2 (\Pi (V_2 , \Sigma))$ 
as we see in the 
Seifert-van Kampen theorem. 
The incoincidence $\langle \Gamma (V_1), \Gamma (V_2) \rangle \lneq \Gamma (M, \Sigma)$, however, implies that 
this is not true because 
\[  \partial_+ ( \langle I_1 ( \Pi_1 (V_1, \Sigma)) , I_2 ( \Pi_1 (V_2, \Sigma) ) \rangle ) = \langle \Mdy (V_1), \Mdy (V_2) \rangle \]
while 
\[ \partial_+ (\Pi (M, \Sigma)) = \langle \Mdy (M, \Sigma) \rangle,\] 
and they are different.}
\end{remark}

Let $M = V_1 \cup_{\Sigma} V_2$ be a Heegaard splitting of 
a closed, orientable $3$-manifold $M$. 
We will adopt the following convention. 
Given an orientation of $M$, or equivalently, a fundamental class $[M] \in H_3 (M)$, 
we always choose the fundamental classes 
$[V_i] \in H_3 (V_i, \partial V_i)$ ($i=1,2$) and 
$[\Sigma] \in H_2 (\Sigma)$ so as to satisfy the following.

\begin{convention}
\label{conv:orientation}
{\rm
$[M] = [V_2] - [V_1]$ and 
$[\Sigma] = [\partial V_1]= [\partial V_2]$, where $[\partial V_i]$ 
is the one induced from $[V_i]$.
} 
\end{convention}
\noindent

By $[I] \in H_1 (I; \partial I)$ we always mean the fundamental class corresponding 
to the canonical orientation of $I$.  

We define a map $\widehat{\Deg} : \Pi (M, \Sigma) \to \ZZ^2$ as follows. 
First, we fix an orientation of $M$. 
Let 
$\alpha:\Sigma\times I \to M$
be a homotopy motion. 
Consider the homomorphism
\[
\alpha_*:H_3 (\Sigma \times I, \Sigma \times \partial I) \to H_3 (M, \Sigma)
\cong H_3 (V_1, \partial V_1)\oplus H_3 (V_2, \partial V_2),
\]
and let $(d_1,d_2)$ be the pair of integers such that
$\alpha_*([\Sigma \times I])=d_1 [V_1] + d_2 [V_2]$,
where $[\Sigma \times I]$ is the cross product of $[\Sigma]$ and $[I]$. 
This pair is uniquely determined by the equivalence class of $\alpha$.
We then define $\widehat{\Deg} ( \alpha )= (d_1,d_2)$.
We note that this invariant does not depend on the orientation of $M$ 
under the above convention. 
The following examples can be easily checked.

\begin{example}
\label{example:degree of Dehn twist}
{\rm
Let $f$ be a Dehn twist about a meridian of $V_1$. This 
is an element of $\Mdy^+ (V_1)$.
In fact, we can construct a homotopy motion $\alpha$ of
$\Sigma=\partial V_1$ in $V_1$ with terminal end $f$ as follows.
As in Subsection \ref{subsection:aspherical-case},
let $Y_1$ be a spine of $V_1$
and let $\{r_t\}_{t\in I}$ be a strong deformation retraction of $V_1$ onto $Y_1$. 
Define a homotopy motion $\alpha$ of $\Sigma$ in $M=V_1\cup V_2$ by
\[ \alpha (t) (x) =
\begin{cases}
r_{2t}(x) & (0\le t\le 1/2)\\
r_{2-2t}(f(x)) & (1/2\le t\le 1).
\end{cases}
\]
Then we have $\widehat{\Deg}(\alpha)=(0,0)$.
}
\end{example}

\begin{example}
\label{example:degree of I-bundle involution}
{\rm
Let $h$ be a vertical $I$-bundle involution on $V_1$, 
which is an element of $\Mdy^- (V_1)$. 
Then the linear homotopy joining each $x\in\Sigma$ with $h(x)$ in the fiber $I_x$ containing
both $x$ and $h(x)$ determines a homotopy motion 
$\alpha$ of $\Sigma=\partial V_1$ in $V_1$ with terminal end $h|_{\Sigma}$.
Regarding $\alpha$ as a homotopy motion of $\Sigma$ in $M=V_1\cup V_2$, 
we have $\widehat{\Deg}(\alpha)=(-2,0)$
as shown below.
Note that if $x\in \Sigma$ projects to an interior point 
of the base surface of the $I$-bundle structure
then the inverse image of the fiber $I_x$ by $\alpha$
is the disjoint union $(\{x\}\times I)\sqcup (\{h(x)\}\times I)
\subset \Sigma\times I$.
Moreover, for a small regular neighbourhood $B$ of a point
in the interior of $I_x$,
the inverse image $\alpha^{-1}(B)$ is the disjoint union of two $3$-balls
$B_1$ and $B_2$ such that the restriction of $\alpha|_{B_i}:B_i\to B$
is a homeomorphism for $i=1,2$.
By Convention \ref{conv:orientation},
$\alpha|_{B_i}$ is orientation-reversing with respect to the orientations
induced from the fundamental classes $[\Sigma\times I]$ and $[V_1]$.
Since the degree is equal to the sum of local degrees 
(see for example \cite[Section 3.3, Exercise 8]{Hatcher02}),
we have  $\alpha_*([\Sigma \times I])=-2[V_1]$
for the map $\alpha:\Sigma\times I \to V_1$.
This implies $\widehat{\Deg}(\alpha)=(-2,0)$, 
because the image of the map $\alpha:\Sigma\times I \to M$ 
does not contain $V_2$.
Similarly, if $h$ is a vertical $I$-bundle involution on $V_2$,
then we have $\widehat{\Deg}(\alpha)=(0, -2)$
for the corresponding homotopy motion $\alpha$ of $\Sigma$ in $M$.
}
\end{example}

\begin{example}
\label{example:Deg and deg}
{\rm
Recall the homomorphism $\deg : \mathcal{K}(M,\Sigma)\to \ZZ$
introduced in 
Definition \ref{def:degree-KernelGroup}.  
For each $\alpha\in\mathcal{K}(M,\Sigma)$,
we have 
\[
\widehat{\Deg}(\alpha) = 
(-\deg (\alpha), \deg (\alpha)). 
\]
}
\end{example}

\begin{example}
\label{example:degree of book rotation}
{\rm
Suppose that $M=V_1\cup_{\Sigma} V_2$ is the Heegaard spitting  
induced from an open book decomposition, 
and let $\rho$ and $\sigma$, respectively, be the half rotation and the unilateral rotation of $\Sigma$
associated with the open book decomposition. 
Then we check that
\[
\widehat{\Deg}(\rho)=(-1,-1), \quad \widehat{\Deg}(\sigma)=(-1,1)
\] 
by an argument similar to that in Example \ref{example:degree of I-bundle involution}.
However, since the orientation convention is slightly involved,
we give a detailed explanation.
As explained in Section \ref{sec:Open book rotations}, 
we  have the natural identification $( M, L ) = (F \times \RR, \partial F \times \RR ) / \sim$,
where $L$ is the binding of the open book decomposition.
We equip $\RR$ with the standard orientation, and then we orient $F$ so that the orientation of $M$ 
is compatible with that of $F \times \RR$, where 
we note that each chain in $C_* (F \times I)$ naturally projects to that in $C_*(M) = C_*( (F \times I) / \sim)$. 
By abuse of notation, we shall not distinguish notationally between an oriented manifold and 
an $n$-chain representing it. 
Under this identification, we can write the chains $V_1$, $V_2$, $\Sigma$ and $M$ 
according to Convention \ref{conv:orientation} as follows: 
\begin{itemize}
\item
$V_1 = - (F \times [ 0 , 1/2 ]) / \sim$; 
\item 
$V_2 = (F \times [ 1/2 , 1 ]) / \sim$; and 
\item
$\Sigma = (F_0 - F_{1/2}) / \sim = \partial V_1 = \partial V_2$, where $F_t := F \times \{t\}$. 
\item
$M = (F \times I)/ \sim = - V_1 + V_2$.
\end{itemize}

Now, consider the chain map $\rho_\# : C_3 (\Sigma \times I) \to C_3 (M) $ induced by the half-rotation $\rho$. 
By the definition of $\rho$, we have 
\begin{align*}
\rho_\# (\Sigma \times I) &= \rho_\# ( ((F_0 - F_{1/2}) / \sim) \times I ) = \rho_\# ( (F_0 / \sim) \times I - (F_{1/2} / \sim) \times I ) \\
&= ((F \times [0,1/2]) / \sim) - (( F \times [1/2, 1])/ \sim) = -V_1 - V_2. 
\end{align*}
Hence $\widehat{\Deg} ( \rho ) = (-1, -1)$. 

Similarly, for the unilateral rotation $\sigma$, we have 
\begin{align*}
\sigma_\# (\Sigma \times I) &= \sigma_\# ( (F_0 / \sim )  \times I - (F_{1/2} / \sim) \times I ) \\
&= (( F \times I)/ \sim) - 0 = M = - V_1 + V_2. 
\end{align*}
Hence $\widehat{\Deg} ( \sigma ) = (-1, 1)$. 
}
\end{example}

Examples \ref{example:degree of Dehn twist}, 
\ref{example:degree of I-bundle involution}, and
\ref{example:degree of book rotation}
allow us to predict that
$\rho(1)$ and $\sigma(1)$ should give a gap between
$\Mdy(M,\Sigma)$ and the subgroup $\langle \Mdy(V_1), \Mdy(V_2)\rangle$ 
for a Heegaard spitting $M=V_1\cup_{\Sigma} V_2$
induced from an open book decomposition.
We are going to verify that 
when $M$ has an aspherical prime summand. 

\begin{lemma}
\label{lem:basic property of lifted degree}
The invariant $\widehat{\Deg}:\Pi (M, \Sigma) \to \ZZ^2$ has the following properties.
\begin{enumerate}
\renewcommand{\labelenumi}{$(\arabic{enumi})$}
\item
Let $\alpha$ be an element of $\Pi (M, \Sigma)$, and let 
$\widehat{\Deg} (\alpha) = (d_1, d_2)$.
Then we have
\[d_1 + d_2 = - 1 + \deg  (\partial_+(\alpha))  .\]
\item
For any pair $\alpha$, $\beta$ of elements of $\Pi (M, \Sigma)$, 
we have
\[
\widehat{\Deg} (\alpha \cdot \beta) 
= \widehat{\Deg}(\alpha) + \deg  (\partial_+(\alpha)) \cdot \widehat{\Deg}(\beta).
\]

\end{enumerate}
\end{lemma}
In the above lemma, $\deg (\partial_+(\alpha)) \in \{ \pm 1 \}$ is the degree of the terminal end
$\partial_+(\alpha)=\alpha(1) \in\Mdy(M,\Sigma)<\MCG(\Sigma)$ of $\alpha$,
as a mapping class of the closed, orientable surface $\Sigma$.

\begin{proof}
\noindent (1) Let $\alpha$ be a homotopy motion of $\Sigma$ in $M$ with 
terminal end $\alpha(1) = f$. 
Consider the following commutative diagram: 
\[
  \begin{CD}
H_3 (\Sigma \times I, \Sigma \times \partial I) @>{\alpha_*}>> H_3(M, \Sigma) \\
  @VVV    @VVV \\
H_2 (\Sigma \times \partial I )  
@>(\alpha|_{\Sigma\times\partial I})_*>>
H_2(\Sigma) ,
  \end{CD}
\]
where the vertical arrows are the connecting homomorphisms. 
By the connecting homomorphism $H_3 (\Sigma \times I, \Sigma \times \partial I) \to 
H_2(\Sigma \times \partial I)$, the fundamental class $[\Sigma \times I]$ 
is mapped to 
$-[\Sigma \times \{ 0 \}] + [\Sigma \times \{ 1 \}]$. 
The homomorphism 
$(\alpha|_{\Sigma\times\partial I})_*:H_2 (\Sigma\times \partial I) 
\to  H_2(\Sigma)$ 
then takes this to 
\[j_* (-[\Sigma]) + (j \circ f)_* ([\Sigma]) = (-1 + \deg (f)) [\Sigma], \]
where $j : \Sigma \to M$ is the inclusion map. 
On the other hand, the
homomorphism 
$\alpha_* : 
H_3 (\Sigma \times I, \Sigma \times \partial I) \to H_3 (M, \Sigma)$ 
takes $[\Sigma \times I]$ to $d_1 [V_1] + d_2 [V_2]$, and then, 
the connecting map 
$H_3(M, \Sigma) \to H_2(\Sigma)$ takes this to $(d_1 + d_2) [\Sigma]$. 
This implies $d_1 + d_2 = - 1 + \deg (f)$. 

\noindent (2) Let $f=\partial_+(\alpha)=\alpha(1)$ be the terminal end of $\alpha$. 
Then the concatenation $\alpha \cdot \beta$ is given by
\[
\alpha \cdot \beta(x,t) = 
\begin{cases}
\alpha(x,2t) & (0 \leq t \leq 1/2) \\
\beta(f(x), 2t-1)  & (1/2 \leq t \leq 1).
\end{cases}
\]
Let $\widehat{\Deg} (\alpha) = (d_1, d_2)$ and $\widehat{\Deg} (\beta) = (e_1, e_2)$. 
The assertion then follows from  
\begin{align*}
(\alpha \cdot \beta)_* ( [\Sigma \times I] ) 
&= 
\alpha_* ( [\Sigma \times I] ) + ( \beta_* \circ (f \times \id_I)_* ) ([\Sigma \times I]) \\
&= (d_1 [V_1] + d_2 [V_2]) + \deg (f \times \id_I) ( e_1 [V_1] + e_2 [V_2])\\
&= (d_1 [V_1] + d_2 [V_2]) + \deg (f) \cdot ( e_1 [V_1] + e_2 [V_2]).
\end{align*}
\end{proof}

The following corollary 
generalises
Examples \ref{example:degree of Dehn twist} and 
\ref{example:degree of I-bundle involution}.

\begin{corollary}
\label{lem:homological degree for gammai}
Let $\alpha$ be a homotopy motion of $\Sigma$ in $V_i$ $($$i=1$ or $2$$)$ 
with terminal end $f$,
and regard it as a homotopy motion of $\Sigma$ in $M$.
Then $\widehat{\Deg}(\alpha)=( - 1 + \deg (f), 0)$ or $(0, - 1 + \deg (f))$
according to whether $i=1$ or $2$.
\end{corollary}

\begin{proof}
Put $(d_1,d_2)=\widehat{\Deg}(\alpha)$ and 
suppose that $\alpha$ comes from a homotopy motion in $V_1$. 
Then the image of $\alpha$ is equal to $V_1$, and so we have $d_2=0$.
By Lemma \ref{lem:basic property of lifted degree}(1),
$d_1=-1+\deg (f) - d_2=-1+\deg (f)$, 
completing the proof for the case $i=1$.
The remaining case $i=2$ is proved by the same argument.
\end{proof}

The following corollary is a consequence of 
Lemma \ref{lem:basic property of lifted degree}(2)
and the definition of  a semi-direct product.

\begin{corollary}
\label{cor:semi-direct product}
Let $C_2=\{\pm 1\}$ be the order-$2$ cyclic group, and
consider its action on $\ZZ^2$ defined by $(-1)\cdot (d_1,d_2)=(-d_1,-d_2)$.
Let $\ZZ^2\rtimes C_2$ be the semi-direct product
determined by this action.
Then the map $\Pi(M,\Sigma) \to \ZZ^2\rtimes C_2$ defined by
$\alpha\mapsto (\widehat{\Deg}(\alpha) , \deg (\partial_+(\alpha)))$
is a group homomorphism.
\end{corollary}

By the above corollary 
we can define a map $\Deg : \Gamma (M, \Sigma) \to \ZZ^2$ so that the diagram 
\[
  \xymatrix{
    1 \ar[r] & \mathcal{K} (M, \Sigma) \ar[r] & \Pi (M, \Sigma) \ar[r]^{\partial_+} \ar[d]_{\widehat{\Deg}} &  \Mdy (M, \Sigma) \ar[r] \ar[ld]^{\Deg} & 1 \\  
                &                                              &  \ZZ^2                                        & 
}\]
commutes if and only if $\widehat{\Deg}$ vanishes on 
$\mathcal{K} (M, \Sigma)$.
By Example \ref{example:Deg and deg},
the latter condition is satisfied 
if and only if the homomorphism $\deg : \mathcal{K} (M, \Sigma) \to \ZZ$ vanishes,
namely $(M,\Sigma)$ is not dominated by $\Sigma\times S^1$.
Hence, Theorems \ref{thm:non-zero degree maps 1} and 
\ref{thm:non-zero degree maps 2}  
imply the following proposition.

\begin{proposition}
\label{prop:well-definedness of homological degree for aspherical manifolds}
Let $M = V_1 \cup_\Sigma V_2$ be a Heegaard splitting of a closed, orientable 
$3$-manifold $M$.
Then if $M$ has an 
aspherical prime summand, 
the map $\Deg : \Mdy (M, \Sigma) \to \ZZ^2$ is well-defined. 
\end{proposition}

From the properties of $\widehat{\Deg}$ we have the following.

\begin{lemma}
\label{lem:homological degree for a composition}
Let $M = V_1 \cup_\Sigma V_2$ be a Heegaard splitting of a closed, orientable 
$3$-manifold $M$, and assume that the map $\widehat{\Deg}$
vanishes on $\mathcal{K} (M, \Sigma)$ and so the map
$\Deg:\Mdy(M,\Gamma) \to \ZZ$ is defined.
Then the following hold. 
\begin{enumerate}
\renewcommand{\labelenumi}{$(\arabic{enumi})$}
\item
Let $f$ be an element of $\Mdy (M, \Sigma)$, and 
suppose $\Deg (f) = (d_1, d_2)$. 
Then we have $d_1 + d_2 = - 1 + \deg (f)$. 
\item
For any $f, g \in \Mdy (M, \Sigma)$, we have 
$\Deg (g \circ f) 
= \Deg (f) + \deg (f) \cdot \Deg (g)$ 
\item
For any $f \in \Mdy (V_1)$, we have $\Deg (f) = ( - 1 + \deg (f), 0)$; and  
for any $f \in \Mdy (V_2)$, we have $\Deg (f) = (0, - 1 + \deg (f))$. 
\item 
If $f \in \langle \Mdy (V_1), \Mdy (V_2) \rangle$, then $\Deg (f)\equiv (0,0) \pmod{2}$. 
\end{enumerate}
\end{lemma}

\begin{proof}
The assertions (1) and (2) follow from 
Lemma \ref{lem:basic property of lifted degree},
and the assertion (3) follows from 
Corollary \ref{lem:homological degree for gammai}.

(4) If $f$ belongs to either $\Mdy (V_1)$ or $\Mdy (V_2)$,
then the assertion holds by (3).
Since 
\[ 
\Deg (g \circ f) \equiv \Deg (f) + \Deg (g) \pmod{2}
\]
by (2), we obtain the desired result for every
$f \in \langle \Mdy (V_1), \Mdy (V_2) \rangle$. 
\end{proof}

\begin{remark}
{\rm
Lemma \ref{lem:homological degree for a composition}(4) can refined as follows:
If $f \in \langle \Mdy (V_1), \Mdy (V_2) \rangle$, 
then $\Deg (f)$ is one of $(2k, -2k)$ and $(2k-2, -2k)$ for some $k \in \ZZ$,
according to whether $f$ is orientation-preserving or reversing. 
}
\end{remark}

Now we are ready to prove Theorem $\ref{thm:well-definedness of homological degree for aspherical manifolds}$. 

\begin{proof}[Proof of Theorem $\ref{thm:well-definedness of homological degree for aspherical manifolds}$]
Let $M = V_1 \cup_{\Sigma} V_2$ be the Heegaard splitting of 
a closed, orientable $3$-manifold $M$ induced from an open book decomposition, and 
assume that 
$M$ has an aspherical prime summand. 
Then by Proposition \ref{prop:well-definedness of homological degree for aspherical manifolds}, 
the map $\Deg : \Mdy (M, \Sigma) \to \ZZ^2$ is well-defined. 
Let $\rho$ and $\sigma$ be the half rotation and the unilateral rotation of $\Sigma$, respectively.
Then by Example \ref{example:degree of book rotation}, we have
$\Deg(\rho(1)) = (-1, -1)$ and $\Deg (\sigma(1) ) = (-1, 1)$.
Therefore, $\rho(1)$ and $\sigma(1)$ do not belong to $\langle \Mdy (V_1), \Mdy (V_2) \rangle$ by 
Lemma \ref{lem:homological degree for a composition}(4), as desired.
\end{proof}
 
In the above proof, 
we have shown that neither $\sigma(1)$ nor $\rho (1)$ is contained in 
$\langle \Mdy (V_1), \Mdy (V_2) \rangle$. 
Clearly, the same consequence holds for any odd power of $\sigma (1)$ and $\rho (1)$ by 
Lemma $\ref{lem:homological degree for a composition}$. 
We do not know, however, whether $\sigma (1)^2$ or $\rho (1)^2$ 
is contained in $\langle \Mdy (V_1), \Mdy (V_2) \rangle$. 

\begin{question}
{\rm
Under the assumption of Theorem $\ref{thm:well-definedness of homological degree for aspherical manifolds}$, 
is $\sigma (1)^2$ or $\rho (1)^2$ contained in $\langle \Mdy (V_1), \Mdy (V_2) \rangle$?}
\end{question}

We see from a result in the companion paper \cite{KodaSakuma20b} that,
for the genus-$1$ Heegaard surface $\Sigma$ of a lens space $L(p,q)$,
there is a gap between $\Mdy(L(p,q),\Sigma)$ and 
$\langle \Mdy (V_1), \Mdy (V_2) \rangle$ generically.
This and Theorem \ref{thm:well-definedness of homological degree for aspherical manifolds} 
are the only examples of Heegaard splittings we know 
for which there are gaps between the two groups. 
By the way, the Hempel distance of a Heegaard splitting 
induced from an open book decomposition is at most $2$.
Thus, we pose the following question.

\begin{question}
{\rm
Let $\Sigma$ be a Heegaard surface of genus at least $2$ of a closed, orientable $3$-manifold $M$.
Is it true that  
$\Mdy(M,\Sigma)=\langle \Mdy (V_1), \Mdy (V_2) \rangle$
if $\Sigma$ has high Hempel distance?
} 
\end{question}

\section{The virtual branched fibration theorem and the group $\langle \Mdy (V_1), \Mdy (V_2) \rangle$}
\label{sec:The virtual branched fibration theorem and the group Mdy (V_1), Mdy (V_2)}

In this section,
we give yet another motivation for studying the group
$\Mdy(M,\Sigma)$ and its subgroup $\langle \Mdy(V_1), \Mdy(V_2)\rangle$
associated with a Heegaard splitting $M=V_1\cup_{\Sigma}V_2$.
To describe this, let $\Inv(V_i)$ ($\subset \MCG(\Sigma)$)
be the set of torsion elements of $\Mdy(V_i)$.
(In fact, this set will turn out to be equal to the set of 
vertical $I$-bundle involutions of $V_i$ as shown in Lemma \ref{lem:I-bundle-Involution}.)
Then we have the following theorem, 
which refines the observation \cite[Addendum 1]{Sakuma81} that
every closed, orientable $3$-manifold $M$ admits a surface bundle 
as a double branched covering space.

\begin{theorem}
\label{thm:vbf-theorem}
Let $M=V_1\cup_{\Sigma} V_2$ be a Heegaard splitting of 
a closed, orientable $3$-manifold $M$.
Then there exists a double branched covering $p:\tilde M \to M$
that satisfies the following conditions.
\begin{enumerate}
\renewcommand{\labelenumi}{$(\roman{enumi})$}
\item
$\tilde M$ is a surface bundle over $S^1$
whose fiber is homeomorphic to $\Sigma$.
\item
The preimage $p^{-1}(\Sigma)$ of the Heegaard surface $\Sigma$
is a union of two $($disjoint$)$ fiber surfaces.
\end{enumerate}
Moreover, the set $D(M,\Sigma)$ of monodromies of such bundles
is equal to the set $\{h_1 \circ h_2 \mid h_i\in \Inv(V_i)\}$,
up to conjugation and inversion.
\end{theorem}

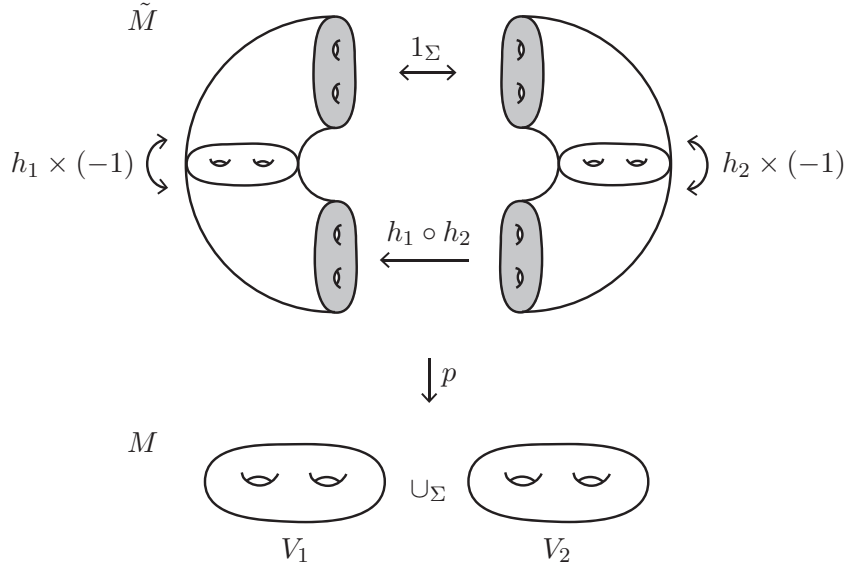
\begin{figure}[htbp]
\centering\includegraphics[width=7.5cm]{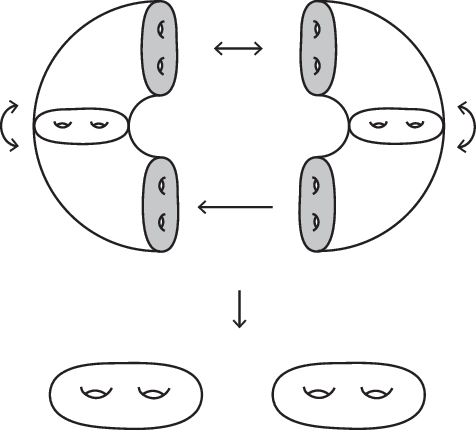}
\begin{picture}(400,0)(0,0)
\put(143,0){$V_1$}
\put(242,0){$V_2$}
\put(41,146){$h_1 \times (-1)$}
\put(310,146){$h_2 \times (-1)$}
\put(192,23){$\cup_{\Sigma}$}
\put(204,67){$p$}
\put(183,120){$h_1 \circ h_2$}
\put(192,190){$1_\Sigma$}
\put(85,200){$\tilde{M}$}
\put(85,40){$M$}
\end{picture} 
\caption{The double branched covering $p : \tilde{M} \to M$.}
\label{fig.vbf}
\end{figure}

\begin{example}
\label{example:Kotschick-Neofytidis}
{\rm
Let $M= \#_g (S^2 \times S^1)$, and  $V_1\cup_{\Sigma}V_2$ the genus-$g$ Heegaard splitting.
Recall that $\Mdy (V_1)=\Mdy(V_2)$ 
(cf. Example \ref{example:Minsky's question for S2 times S1}).
Pick an element $h_1=h_2$ from $\Inv(V_1)=\Inv(V_2)$.
Then $h_1 \circ h_2=\id_{\Sigma}$ and hence the above theorem implies that
$\Sigma\times S^1$ is a double branched covering space of $M= \#_g (S^2 \times S^1)$,
and so $\Sigma\times S^1$ dominates $\#_g (S^2 \times S^1)$.
This gives the construction 
by Kotschick-Neofytidis \cite[Proposition 4]{KotschickNeofytidis13}.
}
\end{example}

We first prove the lemma below, following and correcting the arguments
of Zimmermann \cite[Proof of Corollary 1.3]{Zimmermann79}.

\begin{lemma}
\label{lem:I-bundle-Involution}
Let $V$ be a handlebody with $\partial V=\Sigma$.
Then an element of $\Mdy(V) < \MCG(\Sigma)$
is a nontrivial torsion element
if and only if
it is represented by 
$($the restriction to $\Sigma=\partial V$
of$)$ a vertical $I$-bundle involution of $V$.
\end{lemma}

\begin{proof}
Since the ``if" part is clear, we prove the ``only if" part.
If the genus of $V$ is $0$ or $1$,
then the assertion can be easily proved by using
the facts that 
$\MCG(B^3)=\ZZ_2$ and
$\MCG(S^1\times D^2)\cong \MCG^+(S^1\times D^2)\rtimes \ZZ_2 \cong \Dh_{\infty}\rtimes \ZZ_2$, 
the semi-direct product of the infinite dihedral group $\Dh_{\infty}$
and the order-$2$ cyclic group $\ZZ_2$. 
Assume that the genus of $V$ is greater than $1$. 
Let $h$ be a torsion element of $\Mdy(V)\subset \MCG(\Sigma)$.
Then, by the solution of the Nielsen realization problem  
(see Kerckhoff \cite{Kerckhoff83}),
there exists a conformal structure on $\Sigma=\partial V$
and a conformal (or anti-conformal) map $h'$ of the Riemann surface $\Sigma$
that is isotopic to $h$.
By Bers \cite[Theorem 3]{Bers61},
the Riemann surface $\Sigma$ admits a Schottky uniformization,
i.e., there is a Schottky group $G$
such that the Riemann surface $\Sigma$ is conformally equivalent to
the Riemann surface $\partial \Omega(G)/G$,
where $\Omega(G)$ is the domain of discontinuity of $G$,
and such that the identification of $\Sigma$ with $\partial \Omega(G)/G$
extends to an identification of $V$ with $V(G):=(\HH^3\cup\Omega(G))/G$. 
By Marden's isomorphism theorem \cite[Theorem 8.1]{Marden},
$h'$ extends to an isometry of $V(G)$,
which we continue to denote by $h'$.
Let $\tilde h$ be the lift of $h'$ to $\HH^3$.
Then, by the assumption that $h\in \Mdy(V)$,
the conjugation action of $\tilde h$ on $G$
is an inner-automorphism of $G$, that is, 
there exists an element $k\in G$
such that $\tilde h \circ g \circ \tilde h^{-1}= k \circ g \circ k^{-1}$ for every $g\in G$.
Thus $\tilde h \circ k^{-1}$ belongs to the centralizer, $Z$,
of $G$ in 
the isometry group $\Isom(\HH^3)$. 
Since $G$ is a free group of rank $\ge 2$,
it follows that $Z$ is trivial
except when $G$ preserves a hyperbolic plane.
In the exceptional case, $Z$ is the order-$2$ cyclic group
generated by the reflection in the hyperbolic plane, and so
we may assume that $\tilde h$
is the reflection in the hyperbolic plane preserved by $G$. 
This implies that the isometry $h'$ of $V(G)$
is a vertical $I$-bundle involution.
This completes the proof of the lemma.
\end{proof}

\begin{remark}
{\rm
The assertion in the proof that
there exists a Schottky group $G$ 
such that $h$ is realized by an isometry of $V(G)$
is proved by Zimmermann \cite[Theorem 1.1]{Zimmermann79}
under a more general setting.
In fact, \cite[Theorem 1.1]{Zimmermann79}
says that any finite subgroup of $\MCG(V)$
is realized as a subgroup of the isometry group of $V(G)$. 
His proof is based on Zieschang's partial solution of the Nielsen realization problem,
which was available at that time,
and some delicate consideration on the group structure,
which guarantees that Zieschang's result is applicable to his setting.
Since we only need to consider cyclic groups,
we do not need the consideration of the group structure,
or we may simply appeal to Kerckhoff's full solution of 
the Nielsen realization problem \cite{Kerckhoff83}.
In our terminology,
\cite[Corollary 1.3]{Zimmermann79} should be read as follows:
the orientation-preserving subgroup of $\Mdy^+(V)$ of $\Mdy(V)$
is torsion-free.
(A similar proof of 
this result was also given by Otal \cite[Proposition 1.7]{Otal88},
and an outline of a similar proof, suggested by Minsky, is included in \cite[Introduction]{BestvinaFujiwara17}.)
Thus Lemma \ref{lem:I-bundle-Involution} is a slight extension 
of \cite[Corollary 1.3]{Zimmermann79}.
}
\end{remark}

\begin{lemma}
\label{lem:I-bundle-handlebody}
Let $V$ be a handlebody with $\partial V=\Sigma$,
and suppose that $h$ is 
an orientation-reversing involution of $\Sigma$
that extends to a vertical $I$-bundle involution of $V$.
Then there exists a double branched covering projection
$p:\Sigma\times [-1,1] \to V$ satisfying the following conditions. 
\begin{enumerate}
\renewcommand{\labelenumi}{$(\roman{enumi})$}
\item
$p(x,1)=p(h(x),-1)=x\in\Sigma=\partial V$ for every $x\in \Sigma$.
\item
The covering transformation is given by the involution 
$\hat h:=h\times (-1)$ of $\Sigma\times [-1,1]$
defined by $\hat h(x,t)=(h(x),-t)$. 
In particular, 
the branch set of $p$ is
equal to the image of $\Fix(h)\times \{0 \} \subset \Sigma\times [-1,1]$
in $V$.
\end{enumerate}
\end{lemma}

\begin{proof}
Let $\hat V$ be the quotient of $\Sigma\times[-1,1]$
by the orientation-preserving involution $\hat h$ defined by the formula in (ii),
and let $\hat p:\Sigma\times[-1,1] \to \hat V$ be the projection.
Then $\hat p$ is a double branched covering projection
with branched set the image of $\Fix(h)\times \{0 \} \subset \Sigma\times [-1,1]$ in $\hat V$,
and the restriction of $\hat p$ to $\Sigma\times \{1 \}$ is 
a homeomorphism onto $\partial\hat V$.
We identify $\partial\hat V$ with $\Sigma$ via this homeomorphism,
i.e., identify each $x\in\Sigma$ with $\hat p(x,1)\in\partial\hat V$.
We show that the identification of $\Sigma=\partial V$ with $\partial\hat V$
extends to a homeomorphism from $V$ to $\hat V$.
To this end, recall the assumption that
$h$ extends to a vertical $I$-bundle involution of $V$,
that is, there exists an $I$-bundle structure of $V$
such that $h$ preserves each fiber setwise and acts on it as a reflection.
Then the base space of the $I$-bundle structure is
identified with the quotient surface $F:=\Sigma/h$
and we can construct a complete meridian system of $V$ as follows.
Pick a complete arc system $\{\delta_i\}_{i=1}^g$ of $F$.
Then the preimages of these arcs by the $I$-bundle projection 
form a complete meridian disk system of $V$.
Let $\{\alpha_i\}_{i=1}^g$ be the family of essential loops on $\Sigma=\partial V$
obtained as the boundaries of these meridian disks.
Note that the involution $h$ preserves each $\alpha_i$ and that
$\alpha_i/h=\delta_i\subset \Sigma/h=F$. 
This implies that the quotient $(\alpha_i\times [-1,1])/\hat h$ is a meridian disk
of the handlebody $\hat V=\Sigma\times[-1,1]/\hat h$
bounded by the loop $\alpha_i \subset \Sigma=\partial\hat V$.
Since the meridian loop $\alpha_i$ of $V$ remains to be a meridian loop of 
$\hat V$ under the identification of $\Sigma=\partial V$ with $\partial\hat V$,  
the identification homeomorphism extends to a homeomorphism from $V$ to $\hat V$.
Thus the composition of the branched covering projection 
$\hat p:\Sigma\times[-1,1] \to \hat V$ 
and the identification homeomorphism $\hat V \cong V$
determines the desired branched covering projection 
$p:\Sigma\times [-1,1] \to V$.
\end{proof}

By using the result of Kim-Tollefson \cite[Theorem A]{KimTollefson77}
on involutions of product spaces,
we can obtain the following converse to Lemma \ref{lem:I-bundle-handlebody}.

\begin{lemma}
\label{lem:I-bundle-handlebody_converse}
Let $V$ be a handlebody with $\partial V=\Sigma$,
and $p:\Sigma\times [-1,1] \to V$ a double branched covering projection
such that the restriction $p|_{\Sigma \times \{1 \}} : \Sigma \times \{1 \} \to \partial V = \Sigma$ 
is the identity, i.e., $p(x,1)=x$ for every $x\in\Sigma$.
Then there exists an orientation-reversing involution $h$ of $\Sigma$ 
that extends to a vertical $I$-bundle involution of $V$ 
such that $p$ is equivalent to the covering projection
constructed from $h$ as indicated in Lemma $\ref{lem:I-bundle-handlebody}$.
To be precise, there exists a self-homeomorphism of $\Sigma\times [-1,1]$
that fixes $\Sigma\times \{1 \}$
such that the composition of this homeomorphism and $p$  
is equal to the covering projection constructed in Lemma $\ref{lem:I-bundle-handlebody}$.
\end{lemma}

\begin{proof}
Let $g$ be the covering transformation of the double branched covering $p$.
Since $g$ interchanges the two components of $\Sigma\times\partial I$,
the result \cite[Theorem A]{KimTollefson77} implies that
there exists an orientation-reversing involution $h$ of $\Sigma$
such that $g$ is equivalent to the involution $h\times (-1)$.
To be more precise, we can see that $g$ is conjugate to $h\times (-1)$
by a self-homeomorphism of $\Sigma\times [-1,1]$
that fixes $\Sigma\times \{1 \}$.
Thus we may assume that $g=h\times (-1)$.
By the assumption, $\Sigma\times [-1,1]/g$ is identified with 
the handlebody $V$ in such a way that
the point $[x,1]$ of $\Sigma\times [-1,1]/g$
represented by $(x,1)$ 
is identified with the point $x\in\Sigma=\partial V$
for every $x\in\Sigma$.
Now consider the involution $h\times 1$ of $\Sigma\times [-1,1]$.
This map is commutative with the involution $g=h\times (-1)$
and so it descends to an involution $\bar h$ of
$V=\Sigma\times [-1,1]/g$.
The restriction of $\bar h$ to $\partial V=\Sigma$
is equal to $h$.
Moreover, $\bar h$ is a vertical $I$-bundle involution of $V$, as shown below.
Note that $V=\Sigma\times [-1,1]/g=\Sigma\times [0,1]/(x,0)\sim (h(x),0)$,
and so there exists a deformation retraction of $V$ onto 
the subspace $F:=\Sigma \times \{0 \} /(x,0)\sim (h(x),0)$.
Thus $F$ is a compact surface with nonempty boundary,
which is embedded in the interior of $V$ and is a deformation retract of $V$.
Note that $\Fix (\bar h)$ is equal the union of the image of 
$\Fix(h) \times [0,1]$ and the image of $\Sigma\times \{0 \}$.
The former is a disjoint union of annuli and the latter is equal to $F$.
Thus $\Fix (\bar h)$ is a surface properly embedded in $V$
and contains $F$ as its deformation retract.
This implies that $\bar h$ is an $I$-bundle involution,
where $\Fix (\bar h)\cong F$ is the base space of the $I$-bundle structure of $V$.
Thus we have proved that the involution $h$ of $\Sigma=\partial V$
extends to the vertical $I$-bundle involution $\bar h$ of $V$.
Since the covering involution $g$ of 
the double branched covering projection $p:\Sigma\times [-1,1] \to V$
is given by $g=h\times (-1)$, 
we can say that $p$ is obtained from $h$, satisfying the prescribed condition,
as indicated in Lemma $\ref{lem:I-bundle-handlebody}$.
\end{proof}

\begin{proof}[Proof of Theorem $\ref{thm:vbf-theorem}$]
For $i=1,2$, pick an element $h_i$ of $\Inv(V_i)\subset\MCG(\Sigma)$.
By Lemma \ref{lem:I-bundle-Involution},
$h_i$ is represented by an orientation-reversing involution of $\Sigma$
that extends to a vertical $I$-bundle involution of $V_i$.
We continue to denote the orientation-reversing involution of $\Sigma$ by $h_i$.
Let $p_i:\Sigma\times [-1,1]\to V_i$
the double branched covering projection
given by Lemma \ref{lem:I-bundle-handlebody}.
Take two copies $[-1,1]_i$ of $[-1,1]$,
and regard $p_i$ as a map $\Sigma\times [-1,1]_i\to V_i$.
Let $\tilde M$ be the space obtained from
the disjoint union 
$\sqcup_{i=1}^2 \Sigma\times [-1,1]_i$
through the identification
\[
(x, 1)_1 \sim (x,1)_2, \quad (h_1(x), -1)_1 \sim (h_2(x),-1)_2 \quad (x\in\Sigma).
\] 
Here $(x,t)_i$ denotes the point in $\Sigma\times [-1,1]_i$ 
corresponding to $(x,t)\in\Sigma\times [-1,1]$.
Then $\tilde M$ is a $\Sigma$-bundle over $S^1$ with monodromy 
$h_1^{-1} \circ h_2=h_1 \circ h_2$. 
Moreover we can glue the branched covering projections
$p_i:\Sigma\times [-1,1]_i\to V_i$ ($i=1,2$) together
to obtain a continuous map 
$p:\tilde M\to M=V_1\cup_{\Sigma} V_2$,
because 
\[
p_1((h_1(x),-1)_1)=p_1((x,1)_1)=x=p_2((x,1)_2)=p_2((h_2(x),-1)_2).
\]
Then $p$ is a branched covering projection
whose branch set is the union of those of $p_1$ and $p_2$.
Hence the $\Sigma$-bundle over $S^1$ with monodromy $h_1 \circ h_2$
is a double branched covering space of $M$.
Moreover the preimage $p^{-1}(\Sigma)$ of the Heegaard surface 
$\Sigma=\partial V_1=\partial V_2$
is the image of $\Sigma\times\partial [-1,1]_1$
(and so is that of $\Sigma\times\partial [-1,1]_2$) in $\tilde M$.
Thus, $p^{-1}(\Sigma)$ is a union of two fiber surfaces.
This completes the proof of the first assertion of 
Theorem \ref{thm:vbf-theorem} and the assertion 
$\{h_1 \circ h_2 \mid h_i\in \Inv(V_i)\}\subset D(M,\Sigma)$.

\smallskip

We prove $D(M,\Sigma)\subset \{h_1\circ h_2 \mid h_i\in \Inv(V_i)\}$.
To this end, let 
$p:\tilde M \to M$ be a double branched covering
satisfying the conditions (i) and (ii) of 
Theorem \ref{thm:vbf-theorem},
and let $\tau$ be the covering involution.
By the condition (ii), $p^{-1}(\Sigma)$ consists of two (distinct and so disjoint)
fiber surfaces, $\Sigma_0$ and $\Sigma_1$,
and $\tau$ interchanges these two components.
Set $\tilde V_i=p^{-1}(V_i)$ ($i=1,2$).
Then 
$\tilde V_1\cap \tilde V_2=\partial \tilde V_1=\partial \tilde V_2=
\Sigma_0\sqcup\Sigma_1$
and $\tilde V_1\cong \tilde V_2\cong \Sigma\times [-1,1]$.
We identify the fiber surface $\Sigma_0$ with 
the Heegaard surface $\Sigma$ via the restriction $p|_{\Sigma_0}$.
Then there exists a homeomorphism $\psi_i:\tilde V_i \to \Sigma\times [-1,1]$
such that $\psi_i(x)=(p(x),1)$ for every $x\in \Sigma_0$.
Let $\tau_i$ be the involution of $\Sigma\times [-1,1]$  defined by
$\psi_i\circ \tau|_{\tilde V_i}\circ \psi_i^{-1}$.
Then $p_i:=p|_{\tilde V_i}\circ \psi_i^{-1}:\Sigma\times [-1,1] \to V_i$
is a double branched covering 
whose restriction to $\Sigma\times \{1 \}$ is the identity map
onto $\Sigma=\partial V_i$.
Hence, by Lemma \ref{lem:I-bundle-handlebody_converse}, 
there exists an element $h_i\in\Inv(V_i)$ such that 
the covering $p_i$ is equivalent to that constructed from $h_i$
as indicated in Lemma \ref{lem:I-bundle-handlebody}.
This implies that the monodromy of the $\Sigma$-bundle $\tilde M$
is equal to $h_1 \circ h_2$.
\end{proof}

The characterization of $D(M,\Sigma)$ in Theorem \ref{thm:vbf-theorem}
reminds us of the result of A'Campo \cite[Corollary 1]{A'Campo03}
which says that
the geometric monodromy of an isolated complex hypersurface singularity,
which is defined by a real equation,
is the composition of two orientation-reversing involutions of the fiber,
one of which is the restriction of the complex conjugation. 
Brooks \cite{Brooks85} and Montesinos \cite{Montesinos87} independently
proved that $D(M,\Sigma)$ contains a pseudo-Anosov element 
whenever $g(\Sigma)\ge 2$.
Hirose and Kin \cite{HiroseKin20} studied the asymptotic behavior of 
the minimum of the dilatations of pseudo-Anosov elements
contained in $D(S^3,\Sigma_g)$ as $g\to\infty$, 
where $\Sigma_g$ is the genus-$g$ Heegaard surface of $S^3$.

When $g(\Sigma)=1$,
we will see in 
the companion paper \cite{KodaSakuma20b}
that, for any element $\phi$ of $D(M,\Sigma)$,
the minimum translation length $d(\phi)$ 
of the action of $\phi$ on the curve graph
is comparable with $2d(\Sigma)$, where $d(\Sigma)$
is the Hempel distance of $\Sigma$.
We expect that this toy example may be extended to a result
for general Heegaard splittings.

\begin{question}
\label{question:estimate-translation-length}
{\rm For a Heegaard splitting $M=V_1\cup_{\Sigma} V_2$
of a closed, orientable $3$-manifold $M$
and for an element $\phi\in D(M,\Sigma)$,
is there an estimate of $d(\phi)$, the translation length or the asymptotic translation length 
of the action of $\phi$ on the curve graph of $\Sigma$,
in terms of the Hempel distance $d(\Sigma)$?}
\end{question}

\bibstyle{plain}

\end{document}